\documentclass[11pt]{amsart}
\usepackage{amsmath}
\usepackage{amssymb}
\usepackage{amsthm}
\usepackage{latexsym}
\usepackage{graphicx}
\usepackage{hyperref}
\usepackage{enumerate}
\usepackage[all]{xy}
\usepackage{chngcntr}

\newtheorem{thm}{Theorem}[section]
\newtheorem{cor}[thm]{Corollary}
\newtheorem{lem}[thm]{Lemma}
\newtheorem{prop}[thm]{Proposition}

\newtheorem{thmintro}{Theorem}
\newtheorem{propintro}[thmintro]{Proposition}
\newtheorem{conjintro}[thmintro]{Conjecture}

\theoremstyle{definition}
\newtheorem{defn}[thm]{Definition}

\newtheorem{cond}[thm]{Condition}

\newcommand{\Z}{\mathbb Z}

\newcommand{\R}{\mathbb R}
\newcommand{\C}{\mathbb C}

\newcommand{\mf}{\mathfrak}
\newcommand{\mc}{\mathcal}
\newcommand{\mb}{\mathbf}
\newcommand{\mh}{\mathbb}
\def\cC{{\mathcal C}}

\def\cE{{\mathcal E}}
\def\Irr{{\rm Irr}}
\newcommand{\mr}{\mathrm}
\newcommand{\ind}{\mathrm{ind}}
\newcommand{\enuma}[1]{\begin{enumerate}[\textup{(}a\textup{)}] {#1} \end{enumerate}}
\newcommand{\Fr}{\mathrm{Frob}}
\newcommand{\Sc}{\mathrm{sc}}

\newcommand{\nr}{\mathrm{nr}}

\newcommand{\Rep}{\mathrm{Rep}}
\newcommand{\Res}{\mathrm{Res}}

\newcommand{\der}{\mathrm{der}}

\newcommand{\matje}[4]{\left(\begin{smallmatrix} #1 & #2 \\ 
#3 & #4 \end{smallmatrix}\right)}

\newcommand{\Mod}{\mathrm{Mod}}
\newcommand{\Hom}{\mathrm{Hom}}
\newcommand{\End}{\mathrm{End}}

\newcommand{\isom}{\xrightarrow{\sim}}

\newcommand{\sgn}{\mathrm{sgn}}
\newcommand{\pt}{\mathrm{pt}}

\newcommand{\Modf}[1]{\mathrm{Mod}_{\mathrm{fl}, #1}}
\newcommand{\Ad}{\mathrm{Ad}}
\newcommand{\pr}{\mathrm{pr}}

\newcommand{\IC}{\mathrm{IC}}

\begin{document}

\title{On submodules of standard modules}
\date{\today}
\maketitle

\begin{center}
{\Large Maarten Solleveld} \\[1mm]
IMAPP, Radboud Universiteit Nijmegen\\
Heyendaalseweg 135, 6525AJ Nijmegen, the Netherlands \\
email: m.solleveld@science.ru.nl\\
\end{center}

\begin{abstract}
Consider a standard representation $\pi_{st}$ of a quasi-split reductive $p$-adic group $G$. 
The generalized injectivity conjecture, posed by Casselman and Shahidi, asserts that any generic
irreducible subquotient $\pi$ of $\pi_{st}$ is necessarily a subrepresentation of $\pi_{st}$. 
We will prove this conjecture, improving on the verification for many groups by Dijols.

We study this in a geometric way, motivated by favourable properties of Langlands 
parameters which are open (which means that  
the nilpotent element from the L-parameter belongs to an appropriate open orbit).

Since we do not want to assume a local Langlands correspondence, we involve similar 
parameters via reduction to Hecke algebras. It does not suffice to pass from $G$ to an 
affine Hecke algebra, we further reduce to graded Hecke algebras and from there
to algebras defined in terms of certain equivariant perverse sheaves. 

It is in the geometric setting of graded Hecke algebras from cuspidal local systems on nilpotent
orbits that we can finally put the ``open" condition on L-parameters to good use. The closure 
relations between the involved nilpotent orbits provide useful insights in the internal 
structure of standard modules, which highlight the representations associated with open 
L-parameters and in particular those for which the enhancement of the L-parameter is trivial.

We show that, in the parametrization of irreducible modules of geometric graded 
Hecke algebras, generic modules always have ``open L-parameters". 
This leads to a proof of a version of the generalized injectivity conjecture for 
graded Hecke algebras of geometric type, which is then transferred to reductive 
$p$-adic groups. 
\end{abstract}

\tableofcontents


\section*{Introduction}

Let $\mc G$ be a connected reductive group defined over a non-archimedean local field $F$.
We will simply call $\mc G (F)$ a reductive $p$-adic group. All our $\mc G (F)$-representations
will by default be smooth and on complex vector spaces.

Recall that any standard $\mc G (F)$-representation arises in the following way. Start with 
a parabolic subgroup $\mc P (F) = \mc M (F) \mc{U_P} (F)$ of $\mc G (F)$ and an irreducible
tempered representation $\tau$ of the Levi factor $\mc M (F)$. Let 
$\chi \in \Hom (\mc M (F),\R_{>0})$ be in positive position with respect to $\mc P (F)$. Let 
$I_{\mc P (F)}^{\mc G (F)}$ be the normalized parabolic induction functor. Then
\[
\pi (\mc P (F),\tau, \chi) = I_{\mc P (F)}^{\mc G (F)} (\tau \otimes \chi)
\]
is a standard $\mc G (F)$-representation. Such representations are interesting for several reasons:
\begin{itemize}
\item By the Langlands classification \cite[\S VII.4.2]{Ren}, every standard representation 
has a unique irreducible quotient. This yields a bijection from the set of standard 
$\mc G (F)$-representations (up to isomorphism) to Irr$(\mc G (F))$, the set of irreducible 
$\mc G (F)$-representations up to isomorphism.
\item Standard representations interpolate between Irr$(\mc G (F))$ and the set of irreducible 
tempered representations of Levi subgroups of $\mc G (F)$. 
\item Standard representations are often easier to handle than irreducible representations.
For instance, one can vary $\chi$ in $\pi (\mc P (F),\tau,\chi)$ and then one obtains a 
holomorphic family of $\mc G (F)$-representations (even an algebraic family if we forget 
about the positivity of $\chi$), which can all be defined on the same vector space.
\item There also exist standard representations (or modules) in related settings like 
real reductive groups, semisimple or affine Lie algebras and affine Hecke algebras. Especially
for semisimple Lie algebras, the Jordan--H\"older content of a standard module has 
interesting geometric interpretations. This goes via the Kazhdan--Lusztig conjecture, we 
refer to \cite{KaLu1} and \cite[\S 7.3.10]{Ach} for more background. 
\end{itemize}

\textbf{Conjectures about standard modules and generic representations}\\
A natural follow-up to the Langlands classification is the question: when is a standard
$\mc G (F)$-representation irreducible? Although they are almost always irreducible, the cases where
they are not are usually more interesting. For $\pi (\mc P (F),\tau,\chi)$ with $\mc G (F)$ 
quasi-split and $\tau$ generic, this was the subject of the standard module conjecture, posed by 
Casselman--Shahidi \cite{CaSh} and proven in \cite{HeMu,HeOp}. It says that 
$\pi (\mc P (F),\tau,\chi)$ is irreducible if and only if it is generic.

Next one may wonder: what are the irreducible subrepresentations, or more ge\-ne\-rally the 
irreducible subquotients of a standard representation? The multiplicities with which irreducible
$\mc G (F)$-representations appear as constituents of a standard representations are predicted by
the Kazhdan--Lusztig conjecture, formulated for $p$-adic groups by Vogan 
\cite[Conjecture 8.11]{Vog}. It is phrased in terms of the geometry of Langlands parameters, 
and has been proven in many cases in \cite{SolKL}. However, that does not yet tell us which of 
these constituents occur as subrepresentations. One aspect of that issue is:

\begin{conjintro}\label{conj:A} 
\textup{(Generalized injectivity conjecture \cite{CaSh})} \\
Let $\mc G (F)$ be quasi-split and $\tau$ generic. Then any generic irreducible subquotient of the
standard representation $\pi (\mc P (F),\tau,\chi)$ is a subrepresentation.
\end{conjintro}
We remark that in this setting it is known that $\pi (\mc P(F),\tau,\chi)$ has a unique generic
irreducible subquotient. Conjecture \ref{conj:A} has been verified whenever $\mc G (F)$ has no 
simple factors of exceptional type (and for many Bernstein blocks for other groups as well) 
by Dijols \cite{Dij}. The proof is a tour de force with L-functions, root data and combinatorics.
Unfortunately, it did not give the author of these lines much feeling for why 
Conjecture \ref{conj:A} should hold.
Based on comparisons with other known results and conjectures, we believe that the reasons
for Conjecture \ref{conj:A} should be geometric.  

We will now discuss the background in more detail.
Let $\mc B$ be a minimal parabolic $F$-subgroup of $\mc G$ and let $\mc U$ be the unipotent radical 
of $\mc B$. Fix a character $\xi : \mc{U}(F) \to \C^\times$ which is nondegenerate, that is, 
nontrivial on every root subgroup $\mc{U}_\alpha (F) \subset \mc{U}(F)$ for a simple root $\alpha$. 
We recall that a $\mc G (F)$-representation $\pi$ is generic, or more precisely 
$(\mc{U}(F),\xi)$-generic, if $\Hom_{\mc U (F)} (\pi,\xi)$ is nonzero. An equivalent condition is 
\[
\Hom_{\mc G (F)} \big( \pi, \mr{Ind}_{\mc U (F)}^{\mc G (F)} (\xi) \big) \neq 0,
\] 
where Ind means smooth induction. Following \cite{BuHe}
\[
\text{we call $\pi$ simply generic if }
\dim \Hom_{\mc G (F)} \big( \pi, \mr{Ind}_{\mc U (F)}^{\mc G (F)} (\xi) \big) = 1.
\] 
It has been known for a long time that every irreducible generic representation of a 
quasi-split group is simply generic \cite{Rod,Shal}.
In our view, this multiplicity one property is the essence of genericity for quasi-split groups,
for it enables the normalization of several objects, in particular of intertwining operators
between parabolically induced representations. Therefore it does not come as a surprise that
many properties of generic representations of quasi-split groups also hold for simply generic
representations of arbitrary reductive $p$-adic groups.

The enhanced L-parameters of generic irreducible representations of quasi-split groups are 
known (among others) for classical groups \cite{Art}, for unipotent representations \cite{Ree}
and for principal series representations \cite{SolQS}. All these L-parameters are open, in the
following sense. 

Let ${}^L \mc G = \mc G^\vee \rtimes \mb W_F$ be the L-group of $\mc G (F)$ and write 
$\mf g^\vee = \mr{Lie}(\mc G^\vee)$. Consider a L-parameter for $\mc G (F)$ in Weil--Deligne form, 
so a pair $(\phi,N)$ with 
\begin{itemize}
\item $\phi : \mb W_F \to {}^L \mc G$ is a semisimple smooth homomorphism,
\item $N \in \mf g^\vee$ is nilpotent and belongs to
\[
\mf g^\vee_\phi = 
\{ X \in \mf g^\vee : \mr{Ad}(\phi (w)) X = \|w\| X \; \text{for all } w \in \mb W_F \} .
\]
\end{itemize}
It is known from \cite{KaLu} that $Z_{\mc G^\vee}(\phi)^\circ$ acts on the variety $\mf g^\vee_\phi$ 
with finitely many orbits, of which one is Zariski-open. In this setup 
\begin{equation}\label{eq:0.open}
(\phi,N) \text{ is called open when Ad}(Z_{\mc G^\vee}(\phi)) N \text{ is open in } 
\mf g^\vee_\phi .
\end{equation}
We encountered this terminology in \cite[\S 0.6]{CFZ}, where it is mentioned that bounded 
L-parameters are open. In the same vein, discrete L-parameters are open. Meanwhile these claims
have been proven in \cite{CDFZ}, although not yet for all discrete L-parameters. In Proposition
\ref{prop:7.4} we show, in an alternative way, that bounded or discrete L-parameters are always 
open. For the (conjectural) local Langlands correspondence, this means that the L-parameters of
tempered representations and of essentially square-integrable representations should be open. 

For simply generic representations of groups $\mc G (F)$ that need not be quasi-split, the 
following more relaxed version of open L-parameters is appropriate. 
Let $\rho$ be an enhancement of $(\phi,N)$. Then $(\phi,N,\rho)$ belongs to a unique 
Bernstein component $\Phi_e (\mc G (F))^{\mf s^\vee}$ in the space of all enhanced
L-parameters for $\mc G (F)$, see \cite[\S 8]{AMS1}. The set $\Phi_e (\mc G (F))^{\mf s^\vee}$  
is defined by means of the cuspidal support map for enhanced L-parameters.
Let $\mf g_\phi^{\mf s^\vee}$ be the set of all nilpotent $X \in \mf g_\phi^\vee$ such that 
there exists an enhancement $\rho'$ satisfying $(\phi,X,\rho') \in \Phi_e (\mc G (F))^{\mf s^\vee}$. 
We say that $(\phi,N,\rho) \in \Phi_e (\mc G (F))^{\mf s^\vee}$ is open with respect to 
cuspidal supports if 
\begin{equation}\label{eq:0.scopen}
\text{Ad}(Z_{\mc G^\vee}(\phi)) N \text{ is open in } \mf g^{\mf s^\vee}_\phi .
\end{equation}
Further, we recall that Shahidi \cite[Conjecture 9.4]{Shah} has conjectured that every tempered 
L-packet for a quasi-split group has a generic member. Based on the above, on \cite[Conjecture 
2.6]{GrPr}, on \cite[Conjecture 7.1.(3)]{GrRe} and on the results in this paper, we pose:

\begin{conjintro}\label{conj:B}
Let $\pi$ be an irreducible representation of a reductive $p$-adic group $\mc G (F)$. 
Assume that a local Langlands correspondence exists for the Bernstein block of 
$\Rep (\mc G (F))$ containing $\pi$.
\enuma{
\item If $\pi$ is simply generic, then its L-parameter is open with respect to cuspidal supports.
\item Suppose that $\mc G (F)$ is quasi-split and that the local Langlands correspondence is 
normalized with respect to the Whittaker datum $(\mc{U}(F),\xi)$, in the following sense: every 
irreducible $(\mc U(F),\xi)$-generic representation is matched with an enhanced L-parameter such 
that the enhancement is the trivial representation.

Then $\pi$ is $(\mc U(F),\xi)$-generic if and only if its L-parameter $(\phi,N)$ is open and 
its enhancement is the trivial representation of $\pi_0 (Z_{\mc G^\vee}(\phi,N))$.
}
\end{conjintro}
Part (b), or very similar statements, has been known to several experts. The authors of 
\cite{CDFZ} have, independently from the current paper, arrived at the same formulation.
We remark that for irreducible representations of non-quasi-split groups $\mc G (F)$, the trivial 
representation should never occur as the enhancement of a Langlands parameter, because it should 
already correspond to a representation of the quasi-split inner form of $\mc G (F)$.

It seems to us that the reason why Conjecture \ref{conj:A} should hold in larger generality is
not so much the genericity of $\pi$, but rather that the L-parameter of $\pi$ is open (as in
Conjecture \ref{conj:B}). Triviality of the enhancement also has geometric consequences, but
they are more subtle. The relevant point is that intersection cohomology complexes from
(equivariant) local systems are easier when the local systems are trivial. 
One can say that we replace the analytic motivation for Conjecture 
\ref{conj:A} from \cite{CaSh} by algebro-geometric motivation. 

\begin{conjintro}\label{conj:C}
Let $\pi_{st}$ be a standard representation of a reductive $p$-adic group $\mc G (F)$ and 
let $\pi$ be an irreducible subquotient of $\pi_{st}$. Suppose that (a) or (b) holds:
\enuma{
\item $\pi$ and $\pi_{st}$ are simply generic,
\item a local Langlands correspondence exists for the Bernstein block of 
$\Rep (\mc G (F))$ containing $\pi_{st}$, and $\pi$ corresponds to an open L-parameter 
with trivial enhancement.
}
Then $\pi$ is a subrepresentation of $\pi_{st}$.
\end{conjintro}

We note that Conjecture \ref{conj:C}.b is implied by Conjectures \ref{conj:A} and 
\ref{conj:B}.b. On the other hand Conjecture \ref{conj:C}.a contains conjecture
\ref{conj:A}. 

To state our main result, we focus on a Bernstein block $\Rep (\mc G (F))^{\mf s}$ with 
$\mf s = [\mc M (F),\omega]$ and $\omega$ a simply generic supercuspidal representation
of $\mc M (F)$. It was shown in 
\cite{OpSo,SolEnd} that $\Rep (\mc G (F))^{\mf s}$ is equivalent to the module category 
of a certain extended affine Hecke algebra $\mc H_{\mf s} \rtimes \Gamma_{\mf s}$. 

\begin{thmintro} \textup{(see Section \ref{sec:padic})} 
\label{thm:D} \\
Consider a Bernstein block $\Rep (\mc G (F))^{[\mc M (F),\omega]}$ in the category of
smooth complex representations of a reductive $p$-adic group $\mc G (F)$. Suppose one 
of the following:
\enuma{
\item $\omega$ is simply generic and the algebra $\mc H_{\mf s} \rtimes \Gamma_{\mf s}$
for $\mf s = [\mc M (F),\omega]$ has equal parameters,
\item $\mc G (F)$ is quasi-split and $\omega$ is generic.
}
Then Conjecture \ref{conj:C}.a holds in $\Rep (\mc G (F))^{[\mc M (F),\omega]}$.
\end{thmintro}

In particular, Theorem \ref{thm:D}.b proves the generalized injectivity conjecture from 
\cite{CaSh}, via Conjecture \ref{conj:C}.a.

To obtain results about Conjecture \ref{conj:B} or Conjecture \ref{conj:C}.b
with our techniques, we need to suppose that a good local Langlands correspondence,
constructed via Hecke algebras, is available. The precise assumptions are formulated in
Condition \ref{cond:7.1}. Currently this condition has been shown to hold in the following cases:
\begin{itemize}
\item inner forms of general/special linear groups \cite{ABPSSLn,AMS3},
\item pure inner forms of quasi-split classical $F$-groups \cite{Hei,MoRe,AMS4},
\item principal series representations of quasi-split $F$-groups \cite{SolQS},
\item unipotent representations (of arbitrary connected reductive groups over $F$) 
\cite{Lus-Uni,Lus-Uni2,SolLLCunip,SolRamif},
\item $G_2$ \cite{AuXu}.
\end{itemize}

\begin{thmintro} \textup{(see Theorem \ref{thm:7.2})}
\label{thm:E} \\
Assume that Condition \ref{cond:7.1} holds for $\Rep (\mc G (F))^{\mf s}$, 
so for instance we are in one the cases listed above. Let $\pi \in \Irr (\mc G (F))^{\mf s}$.
\enuma{
\item Suppose that $\pi$ is tempered or essentially square-integrable, 
or that $\mc G (F)$ is quasi-split and $\pi$ is generic. Then the L-parameter of $\pi$ is open. 
\item If $\mf s = [\mc M (F),\omega]$ with $\omega$ simply generic and $\pi$ is generic, 
then the enhanced L-parameter of $\pi$ is open with respect to cuspidal supports.
In particular Conjecture \ref{conj:B}.a holds for $\Rep (\mc G (F))^{\mf s}$.
}
\end{thmintro}

To prove instances of Conjecture \ref{conj:B}.b, one would have to know in addition that the
L-packet of an open L-parameter always contains at least one generic representation. 
That aspect of Conjecture \ref{conj:B} seems rather challenging.

In the remainder of the introduction we discuss how we proved Theorems \ref{thm:D}
and \ref{thm:E}.\\

\textbf{Reduction from $p$-adic groups to graded Hecke algebras}\\
Firstly, one can reduce from $\Rep (G)^{\mf s}$ to the module category of some kind of affine
Hecke algebra. This has been achieved in full generality in \cite{SolEnd}, but in the most
general case some technical difficulties remain, which entail that one does not exactly obtain
the module category of a (twisted) affine Hecke algebra. In Section \ref{sec:padic} we check
that this procedure is still good enough to transfer part (a) of Conjecture \ref{conj:C}
to statements about modules of (twisted) affine Hecke algebras. When the $\omega$ from $\mf s$
is simply generic, $\Rep (G)^{\mf s}$ is really equivalent to the module category of an extended
affine Hecke algebra $\mc H_{\mf s} \rtimes \Gamma_{\mf s}$, and the equivalence of categories
preserves genericity \cite[Theorem E]{OpSo}. 

The next step is reduction from a twisted affine Hecke algebra $\mc H_{\mf s} \rtimes \C 
[\Gamma_{\mf s}, \natural_{\mf s}]$ to a twisted graded Hecke algebra $\mh H_\omega \rtimes \C
[\Gamma_\omega ,\natural_\omega]$, as discussed in Section \ref{sec:AHA}. The procedure for
that is known in general from \cite{Lus-Gr,SolAHA,AMS3}, and preserves all the relevant
properties of representations. This translates Conjecture \ref{conj:C} to statements about
twisted graded Hecke algebras. 
In fact can also reduce directly from $\Rep (G)^{\mf s}$ to twisted graded Hecke algebras,
skipping the slightly messy step with affine Hecke algebras, that is done in \cite{SolEnd}.

To proceed, we need the graded Hecke algebra $\mh H_\omega$ to be of geometric type, by which
we mean that it arises from a cuspidal local system on a nilpotent orbit as in \cite{Lus-Cusp1,
Lus-Cusp2,AMS2}. This puts a condition on the deformation parameters $k_\alpha$ of $\mh H_\omega$,
which can be retraced to a condition on the $q$-parameters of $\mc H_{\mf s}$. That condition is
implied by Lusztig's conjecture \cite{Lus-open} on the $q$-parameters of $\mc H_{\mf s}$, but
it allows a wider choice of parameters than \cite{Lus-open}.
In Theorem \ref{thm:4.7} we show that, when $\mc G (F)$ is quasi-split and $\omega$ is generic,
all the ensuing extended graded Hecke algebras $\mh H_\omega \rtimes \Gamma_\omega$ have equal
parameters, and in particular are of geometric type.\\

\textbf{Representation theory of graded Hecke algebras of geometric type}\\
Finally, we come to the topic of the largest part of the paper: standard modules
for twisted graded Hecke algebras. In Sections \ref{sec:setup}--\ref{sec:generic} the setup is
quite different from above. We start with a complex reductive group $G$ (not related to $\mc G$).
In the body of the paper $G$ may be disconnected, but in this introduction we simplify
the presentation by assuming that $G$ is connected. Let $M$ be a Levi subgroup of $G$ and let
$\cE$ be a $M$-equivariant cuspidal local system on a nilpotent orbit $\cC_v^M$ in $\mf m$. To
these data Lusztig \cite{Lus-Cusp1,Lus-Cusp2} associated a graded Hecke algebra $\mh H (G,M,\cE)$. 
As a vector space it is a tensor product of three subalgebras:
\[
\mc O (\mf t) \otimes \C [\mb r] \otimes \C [W_\cE] ,\quad
\text{where } \mf t = \mr{Lie} (T), T = Z(M), W_\cE = N_G (M) / M .
\]
In the algebra $\mh H (G,M,\cE)$, $\mb r$ is central and the cross relations between $\mc O (\mf t)$
and $\C [W_\cE]$ are determined by parameters $k_\alpha$ for $\alpha \in R(G,T)$. The graded
Hecke algebras $\mh H_\omega$ discussed above arise from $\mh H (G,M,\cE)$ by specializing
$\mb r$ at some $r \in \R_{>0}$.

The irreducible representations of $\mh H (G,M,\cE)$ are naturally parametrized by $G$-conjugacy 
classes of ``enhanced L-parameters for $\mh H (G,M,\cE)$". These are quadruples $(y,\sigma,r,\rho)$ 
where $r \in \C$, $y \in \mf g$ is nilpotent and $\sigma \in \mf g$ is semisimple such that
$[\sigma,y] = 2 r y$. Further $\rho$ is an irreducible representation of $\pi_0 (Z_G (y,\sigma))$,
subject to a certain cuspidal support condition. We fix $v \in \cC_v^M$ and we extend it to an
$\mf{sl}_2$-triple in $\mf m$, with semisimple element $\sigma_v$. The above conditions force 
$\sigma \in \mr{Ad}(G) (\mf t + r \sigma_v)$, so we may assume that $\sigma \in \mf t + r \sigma_v$. 

For every parameter $(y,\sigma,r,\rho)$ there is a ``geometric standard" module 
$E_{y,\sigma,r,\rho}$, constructed using equivariant perverse sheaves. It has an irreducible
quotient $M_{y,\sigma,r,\rho}$, which is unique if $r \neq 0$. In that sense the Langlands
classification holds for $\mh H (G,M,\cE)$. It is preferable to pull back $E_{y,\sigma,r,\rho}$
along the sign automorphism of $\mh H (G,M,\cE)$, given by
\[
\sgn |_{\mc O (\mf t)} = \mr{id} ,\; \sgn (\mb r) = -\mb r ,\; \sgn (s_\alpha) = -s_\alpha
\text{ for every root } \alpha .
\]
When $\Re (r) < 0$, the modules $\sgn^* E_{y,\sigma,r,\rho}$ are precisely the ``analytic
standard" modules of $\mh H (G,M,\cE) / (\mb r + r)$ in the sense of Langlands 
(Proposition \ref{prop:1.13}).

The centre of $\mh H (G,M,\cE)$ is $\mc O (\mf t)^{W_\cE} \otimes \C [\mb r]$, so the space
of central characters is $\mf t / W_\cE \times \C$. Then $Z(\mh H (G,M,\cE))$ acts on 
$\sgn^* E_{y,\sigma,r,\rho}$ with $\sigma \in \mf t + r\sigma_v$ by the character 
$(W_\cE (\sigma -r\sigma_v),-r)$. Moreover, every irreducible $\mh H (G,M,\cE)$-module with this 
central character has the form $\sgn^* M_{y',\sigma,r,\rho'}$ for suitable $(y',\rho')$. This 
allows us to focus on a fixed pair $(\sigma,r) \in \mf t \oplus \C (\sigma_v,1)$ in the 
remainder of the introduction. The nilpotent parameter $y$ lies in
\[
\mf g_N^{\sigma,r} = \{ X \in \mf g \text{ nilpotent} : [\sigma,X] = 2rX \} . 
\]
Analogous to \eqref{eq:0.open}, we say that $(y,\sigma,r)$ is open if Ad$(Z_G (\sigma)) y$
is the unique open orbit in $\mf g_N^{\sigma,r}$.\\
Let us consider the category $\Modf{\sigma,r} (\mh H (G,M,\cE))$ of finite dimensional
$\mh H (G,M,\cE)$-modules all whose irreducible constituents admit the central character 
$(W_\cE (\sigma + r \sigma_v),r)$. According to \cite{Lus-Cusp2,SolKL}, this category is 
canonically equivalent to
\[
\Modf{\sigma,r} \big( \End_{\mc D^b_{Z_G (\sigma)} (\mf g_N^{\sigma,r})} (K_{N,\sigma,r}) \big),
\]
where $K_{N,\sigma,r}$ is a certain $Z_G (\sigma)$-equivariant perverse sheaf on 
$\mf g_N^{\sigma,r}$. With these notations 
\[
\begin{array}{lll}
E_{y,\sigma,r} & = & H^* (\{y\}, i_y^! K_{N,\sigma,r}) ,\\
E_{y,\sigma,r,\rho} & = & \Hom_{\pi_0 (Z_G (\sigma,y))} ( \rho, E_{y,\sigma,r} ) .
\end{array}
\]
In this setting we deduce the crucial geometric step in our chain of arguments:

\begin{propintro} \textup{(see Propositions \ref{prop:1.5} and \ref{prop:1.6})}
\label{prop:F} \\
Let $(y,\sigma,r,\rho)$ be an enhanced L-parameter for $\mh H (G,M,\cE)$ and let $\mc L_\rho$
be the local system on $\mc O_y = \mr{Ad}(Z_G (\sigma)) y$ induced by $\rho$. For another
parameter $(y',\sigma,r, \rho')$, the vector space
\[
\Hom_{\mh H (G,M,\cE)} (\sgn^* E_{y',\sigma,r,\mr{triv}}, \sgn^* E_{y,\sigma,r,\rho}) = 
\Hom_{\mh H (G,M,\cE)} (E_{y',\sigma,r,\mr{triv}}, E_{y,\sigma,r,\rho})
\]
\begin{itemize}
\item is zero if $\mc O_y \not\subset \overline{\mc O_{y'}}$ and $r \neq 0$,
\item is nonzero if $\mc O_y \subset \overline{\mc O_{y'}}$, $\rho'$ is trivial and 
$\mc L_\rho$ appears with nonzero multiplicity in 
$\IC ( \mf g_N^{\sigma,r}, \mc L_{\rho'} ) |_{\mc O_y}$.
\end{itemize}
\end{propintro}

With Proposition \ref{prop:F} we deduce that an irreducible module 
$\sgn^* M_{y',\sigma,r,\mr{triv}}$ with $(y',\sigma,r)$ open occurs as submodule 
of a standard module $\sgn^* E_{y,\sigma,r,\rho}$ whenever it
occurs as a subquotient (Theorem \ref{thm:1.8}).

The condition for genericity of $\mh H (G,M,\cE)$-modules is derived from
\cite[\S 6]{OpSo}:
\[
V \in \Mod (\mh H (G,M,\cE)) \text{ is generic if } \Res^{\mh H (G,M,\cE)}_{\C [W_\cE]} V
\text{ contains the sign representation.}
\]
In Proposition \ref{prop:2.2} we show that every standard module 
$\sgn^* E_{y,\sigma,r,\rho}$ has at most one irreducible generic subquotient, like 
for standard representations of quasi-split reductive $p$-adic groups. Hence generic 
for Hecke algebras corresponds to simply generic for
reductive $p$-adic groups. By reduction to the case $r = 0$, we prove:

\begin{thmintro} \textup{(see Theorem \ref{thm:2.5})}
\label{thm:F}
\enuma{
\item For fixed $(\sigma,r) \in \mf t \oplus \C (\sigma_v,1)$, there is a unique 
(up to conjugacy) pair $(y_g,\rho_g)$ such that $\sgn^* M_{y_g,\sigma,r,\rho_g}$ is generic. 
\item In the setting of part (a) we suppose in addition that $\mh (G,M,\cE)$ has equal 
parameters, e.g. it arises from a generic Bernstein block for a quasi-split reductive 
$p$-adic group. Then $\cC_v^M = \{0\}$, $\cE$ is the trivial local system, 
$\rho_g = \mr{triv}$ and $(y_g,\sigma,r)$ is an open parameter.
}
\end{thmintro}

In Theorem \ref{thm:F}.a the L-parameter has a weaker property instead of being open,
a version of \eqref{eq:0.scopen}.
The analogy between parts (a) and (b) of Theorem \ref{thm:F} is one reason to
believe in Conjecture \ref{conj:C}.a.

Altogether the above results prove a version of Conjecture \ref{conj:C}.b for 
(twisted) graded Hecke algebras of geometric type. Theorem \ref{thm:D} applies that in 
the cases where the reduction from $\Rep (\mc G (F))^{\mf s}$ to twisted graded 
Hecke algebras works well.

In Lemma \ref{lem:1.9} and Theorem \ref{thm:1.10} (which comes from  \cite{AMS2}), 
we check that irreducible $\mh H (G,M,\cE)$-modules which are tempered or essentially
discrete series have open parameters. 
That goes into Theorem \ref{thm:E}, when a nice local Langlands correspondence via 
graded Hecke algebras of geometric type is available.

\renewcommand{\theequation}{\arabic{section}.\arabic{equation}}
\counterwithin*{equation}{section}

\section{Geometric construction of twisted graded Hecke algebras}
\label{sec:setup}

All the groups in Sections \ref{sec:setup}--\ref{sec:AHA} will be complex linear algebraic 
groups. We mainly work in the equivariant bounded derived categories of constructible sheaves 
from \cite{BeLu}. 
For a group $H$ acting on a space $X$, this category will be denoted $\mc D^b_H (X)$.

Let $G$ be a complex reductive group, possibly disconnected. To construct a graded Hecke algebra
geometrically, we need a cuspidal quasi-support $(M,\cC_v^M,q\cE)$ for $G$ \cite{AMS1}. 
This consists of:
\begin{itemize}
\item a \emph{quasi-Levi subgroup} $M$ of $G$, which means that $M^\circ$ is a Levi subgroup of 
$G^\circ$ and $M = Z_G (Z (M^\circ)^\circ)$,
\item $\cC_v^M$ is a Ad$(M)$-orbit in the nilpotent variety $\mf m_N$ in the Lie algebra 
$\mf m$ of $M$,
\item $q\cE$ is a $M$-equivariant cuspidal local system on $\cC_v^M$.
\end{itemize}
The cuspidality of $q\cE$ forces any $v \in \cC_v^M$ to be a distiguished nilpotent element
of~$\mf m$. We write $T = Z(M)^\circ = Z(M^\circ)^\circ$, $\mf t = \mr{Lie}(T) = Z(\mf m)$ and 
\[
W_{q\cE} = \mr{Stab}_{N_G (M)} (q\cE) / M = N_G (M,q\cE) / M ,
\]
which is a finite group. Let $\cE$ be an irreducible $M^\circ$-equivariant local system on
$\cC_v^{M^\circ}$ contained in $q\cE |_{\cC_v^{M^\circ}}$. Then
\[
W_{q\cE} = W_\cE \rtimes \Gamma_{q\cE} ,
\]
where $W_\cE$ is the Weyl group of a root system and $\mc R_{q\cE}$ is the 
$W_{q\cE}$-stabilizer of the set of positive roots.
To these data one associates a twisted graded Hecke algebra
\begin{equation}\label{eq:1.3}
\mh H (G,M,q\cE) = \mh H (\mf t,W_{q\cE},k,\mb r,\natural_{q\cE}) ,
\end{equation}
see \cite[\S 2.1]{SolSGHA}. As vector space it is the tensor product of 
\begin{itemize}
\item a polynomial algebra $\mc O (\mf t \oplus \C) = \mc O (\mf t) \otimes \C [\mb r]$,
\item a twisted group algebra $\C [W_{q\cE},\natural_{q\cE}]$,
\end{itemize}
and there are nontrivial cross relations between these two subalgebras. 
The most important cross relation involves a simple root $\alpha$. It comes with a simple 
reflection $s_\alpha \in W^\circ_{q\cE}$, a basis element $N_{s_\alpha}$ of 
$\C [W_\cE] \subset \C [W_{q\cE},\natural_{q\cE}]$ and a 
parameter $k_\alpha \in \C$. For $\xi \in \mc O (\mf t)$:
\[
N_{s_\alpha} \xi - (\xi \circ s_\alpha) N_{s_\alpha} = 
k_\alpha \mb r (\xi - \xi \circ s_\alpha) / \alpha .
\]
For elements $\gamma \in \Gamma_{q\cE}$ there is a simpler cross relation:
\[
N_\gamma \xi = (\xi \circ \gamma^{-1}) N_\gamma .
\]
Let $\mf g_N$ be the nilpotent variety in the Lie algebra $\mf g$ of $G$. The algebra 
\eqref{eq:1.3} can be realized in terms of suitable equivariant sheaves on $\mf g$ or $\mf g_N$. 
We let $\C^\times$ act on $\mf g$ and $\mf g_N$ by $\lambda \cdot X = \lambda^{-2} X$. Then 
every $M$-equivariant local system 
on $\cC_v^M$, and in particular $q\cE$, is automatically $M \times \C^\times$-equivariant.

Let $P^\circ = M^\circ U$ be the parabolic subgroup of $G^\circ$ with Levi factor $M^\circ$ 
and unipotent radical $U$ matching the aforementioned choice of positive roots.
Then $P = M U$ is a ``quasi-parabolic" subgroup of $G$. Consider the varieties
\begin{align*}
& \dot{\mf g} = \{ (X,gP) \in \mf g \times G/P : 
\Ad (g^{-1}) X \in \cC_v^M \oplus \mf t \oplus \mf u \} , \\
& \dot{\mf g}_N = \dot{\mf g} \cap (\mf g_N \times G/P) .
\end{align*}
We let $G \times \C^\times$ act on these varieties by
\begin{equation}\label{eq:1.7}
(g_1,\lambda) \cdot (X,gP) = (\lambda^{-2} \Ad (g_1) X, g_1 g P) .
\end{equation}
By \cite[Proposition 4.2]{Lus-Cusp1} there are natural isomorphisms of graded algebras
\begin{equation}\label{eq:1.1}
H^*_{G \times \C^\times} (\dot{\mf g}) \cong H^*_{G \times \C^\times} (\dot{\mf g}_N) 
\cong \mc O (\mf t) \otimes_\C \C [\mb r] .
\end{equation}
Consider the maps
\begin{equation}\label{eq:1.2}
\begin{aligned}
& \cC_v^M \xleftarrow{f_1} \{ (X,g) \in \mf g \times G : \Ad (g^{-1}) X \in 
\cC_v^M \oplus \mf t \oplus \mf u\} \xrightarrow{f_2} \dot{\mf g} , \\
& f_1 (X,g) = \mathrm{pr}_{\cC_v^M}(\Ad (g^{-1}) X) , \hspace{2cm} f_2 (X,g) = (X,gP) .
\end{aligned}
\end{equation}
Let $\dot{q\cE}$ be the unique $G \times \C^\times$-equivariant local system on $\dot{\mf g}$ 
such that $f_2^* \dot{q\cE} = f_1^* q\cE$. Let $\mr{pr}_1 : \dot{\mf g} \to \mf g$ be the projection 
on the first coordinate and define
\[
K := \pr_{1,!} \dot{q\cE} \qquad \in \mc D^b_{G \times \C^\times} (\mf g ) .
\]
Let $\dot{q\cE}_N$ be the pullback of $\dot{q\cE}$ to $\dot{\mf g}_N$ and put
\[
K_N := \pr_{1,N,!} \dot{q\cE}_N \qquad \in \mc D^b_{G \times \C^\times} (\mf g_N ) .
\]
It is shown in \cite[\S 2.2]{SolSGHA} that $K_N$ is a semisimple complex isomorphic to the pullback 
of $K$ to $\mf g_N$, and that $K_N$ can be regarded as the parabolic induction of
$\IC_{M \times \C^\times}(\mf m_N, q\cE)$.
 
From \cite[Theorem 2.2]{SolSGHA}, based on \cite{Lus-Cusp1,Lus-Cusp2,AMS2}, we recall:

\begin{thm}\label{thm:1.15}
There exist natural isomorphisms of graded algebras
\[
\mh H (G,M,q\cE) \longrightarrow \End^*_{\mc D^b_{G \times \C^\times}(\mf g)}(K) \longrightarrow
\End^*_{\mc D^b_{G \times \C^\times}(\mf g_N)}(K_N) .
\]
\end{thm}

The irreducible modules and the standard modules of $\mh H (G,M,q\cE)$ have been constructed and
parametrized in \cite{Lus-Cusp1,Lus-Cusp2,AMS2}. The parameters consist of:
\begin{itemize}
\item a semisimple element $\sigma \in \mf g$,
\item $r \in \C$,
\item a nilpotent element $y \in \mf g$ such that $[\sigma,y] = 2 r y$,
\item an irreducible representation $\rho$ of $\pi_0 (Z_{G \times \C^\times} (y))$, such that the
quasi-cuspidal support of $(\sigma,y,\rho)$ is $G$-conjugate to $(M,\cC_v^M,q\cE)$.
\end{itemize}
We call $(y,\sigma,r)$ an L-parameter for $\mh H (G,M,q\cE)$ and $(y,\sigma,r,\rho)$ an enhanced
L-parameter for $\mh H (G,M,q\cE)$. The relation with Langlands parameters for reductive $p$-adic
groups is explained in \cite[\S 1]{AMS2}. We say that two enhanced L-parameters
$(y,\sigma,r,\rho)$ and $(y',\sigma',r',\rho')$ are $G$-associate if $r = r'$ and there
exists a $g \in G$ such that $\mr{Ad}(g) y = y', \mr{Ad}(g) \sigma = \sigma'$ and
$\rho \circ \mr{Ad}(g)^{-1} \cong \rho'$.\\
We say that $(y,\sigma,r)$ is relevant for $\mh H (G,M,q\cE)$ or relevant with 
respect to $(G,M,q\cE)$ if there exists a $\rho$
such that $(y,\sigma,r,\rho)$ is an enhanced L-parameter for $\mh H (G,M,q\cE)$.

\begin{thm}\label{thm:1.1} \textup{\cite[Theorem 4.6]{AMS2}} \\
To each enhanced L-parameter for $\mh H (G,M,q\cE)$ there is associated a standard module 
$E_{y,\sigma,r,\rho}$, which has a distinguished (unique if $r \neq 0$) irreducible quotient 
$M_{y,\sigma,r,\rho}$.

This yields a bijection between $\Irr (\mh H (G,M,q\cE))$ and $G$-association classes of
enhanced L-parameters for $\mh H (G,M,q\cE)$.
\end{thm}

The condition on $\rho$ in enhanced L-parameters for $\mh H (G,M,q\cE)$ is rather subtle and
restrictive. Some instances can be made more explicit:
\begin{equation}\label{eq:2.3}
\text{the quasi-cuspidal support of } (\sigma,y,\mr{triv}) \text{ is always of the form }
(L,\{0\},\mr{triv}) ,
\end{equation} 
where $L$ is a minimal quasi-Levi subgroup of $G$, that is, the $G$-centralizer of a maximal 
torus in $G^\circ$. The reason is that quasi-cuspidal supports are unique up to $G$-conjugation 
and that $(y,\mr{triv})$ already appears in the Springer correspondence for $Z_G (\sigma_0)$, which 
is based on the quasi-cuspidal support $(L,\{0\},\mr{triv})$. In particular $\rho = \mr{triv}$
can only appear in an enhanced L-parameter for $\mh H (G,M,q\cE)$ if $q\cE$ is the trivial
equivariant local system on $\{0\}$.

The centre of $\mh H (G,M,q\cE)$ contains 
\begin{equation}\label{eq:1.4}
\mc O (\mf t \oplus \C)^{W_{q\cE}} = \mc O ( \mf t / W_{q\cE} \times \C) .
\end{equation}
Usually this is the entire centre, and therefore we will just call a character of \eqref{eq:1.4},
ie. an element $(W_{q\cE} \sigma_0 ,r) \in \mf t / W_{q\cE} \times \C$, a central character of 
$\mh H (G,M,q\cE)$.

Not every $(\sigma,r)$ can be extended to an enhanced L-parameter for $\mh H (G,M,q\cE)$, the
existence of $(y,\rho)$ already imposes conditions. We pick an algebraic homomorphism
\[
\gamma_v : SL_2 (\C) \to M \text{ with } \textup{d}\gamma_v \matje{0}{1}{0}{0} = v  
\]
and we put $\sigma_v = \textup{d}\gamma_v \matje{1}{0}{0}{-1} \in \mf m$. According to 
\cite[Lemma 2.1]{SolKL}, in this setting Ad$(G) \sigma - r \sigma_v$ intersects $\mf t$ in a unique 
$W_{q\cE}$-orbit. Therefore we may, and often will, assume that
\begin{equation}\label{eq:1.5}
\text{the semisimple element } (\sigma,r) \text{ lies in } 
\mf t \oplus \C (\sigma_v,1) \subset \mf m \oplus \C .
\end{equation}
In this way $(\sigma,r)$ determines a central character of $\mh H (G,M,q\cE)$. We denote the 
completion of $Z(\mh H (G,M,q\cE))$ with the respect to the powers the ideal
\[
\ker \big( \mr{ev}_{(\sigma,r)} : \mc O (\mf t \oplus \C)^{W_{q\cE}} \to \C \big)
\]
by $\hat{Z} (\mh H (G,M,q\cE))_{\sigma,r}$.

The geometric counterpart of \eqref{eq:1.4} is the commutative graded algebra 
\begin{equation}\label{eq:1.6}
H^*_{G \times \C^\times}(\mr{pt}) \cong \mc O (\mf g \oplus \C )^G ,
\end{equation}
which acts naturally on $\End^*_{\mc D^b_{G \times \C^\times}(\mf g_N)}(K_N)$. The completion
of \eqref{eq:1.6} with respect to the maximal ideal determined by $(\mr{Ad}(G) \sigma,r)$
is denoted $\hat{H}^*_{G \times \C^\times} (\mr{pt})_{\sigma,r}$.

Fix $(\sigma,r)$ as above and write
\[
\mf g_N^{\sigma,r} = \{ X \in \mf g : [\sigma,X] = 2rX, X \text{ is nilpotent} \} .
\]
When $r \neq 0$, the nilpotency is already guaranteed by the first condition, and 
$\mf g_N^{\sigma,r}$ is a vector space. On the other hand, when $r = 0$, $\mf g_N^{\sigma,r}$
is the nilpotent cone in $Z_{\mf g}(\sigma)$. In any case $\mf g_N^{\sigma,r}$ is an
irreducible variety. The group 
\[
C := Z_{G \times \C^\times} (\sigma) = Z_G (\sigma) \times \C^\times
\] 
acts on $\mf g_N^{\sigma,r}$ like in \eqref{eq:1.7}:
\begin{equation}\label{eq:1.8}
(g,\lambda) \cdot X = \lambda^{-2} \mr{Ad}(g) X .
\end{equation}
This action has only finitely many orbits \cite[\S 5.4]{KaLu}, so by the aforementioned
irreducibility of $\mf g_N^{\sigma,r}$
\begin{equation}\label{eq:1.21}
\text{there is a unique open $C$-orbit in } \mf g_N^{\sigma,r} .
\end{equation}
We record the projection and inclusion maps
\[
\mf g_N^{\sigma,r} \xleftarrow{\mr{pr}_{1,N}} \dot{\mf g}_N^{\sigma,r} = 
\mr{pr}_{1}^{-1} (\mf g_N^{\sigma,r}) \xrightarrow{j_{N,\sigma,r}} \dot{\mf g}_N .
\]
With these we define
\[
K_{N,\sigma,r} = (\mr{pr}_{1,N})_! j_{N,\sigma,r}^* (\dot{q\cE}_N) 
\in \mc D_C^b (\mf g_N^{\sigma,r} ).
\]
It was checked in \cite[\S 5.3 and \S 8.12]{Lus-Cusp2} and \cite[Lemma 2.8]{SolSGHA} that this 
is a semisimple complex. The commutative graded algebra $H_C^* (\mr{pt})$ acts naturally on 
$\End_{\mc D^b_C (\mf g_N^{\sigma,r})}(K_{N,\sigma,r})$, by the product in equivariant cohomology.

\begin{thm}\label{thm:1.2} \textup{\cite[Theorem 2.4]{SolKL}} \\
There are natural algebra isomorphisms
\begin{align*}
& \hat Z (\mh H (G,M,q\cE))_{\sigma,r} \underset{Z (\mh H (G,M,q\cE))}{\otimes}
\mh H (G,M,q\cE) \; \isom \\
& \hat{H}^*_{G \times \C^\times} (\pt)_{\sigma,r} \underset{H_{G \times \C^\times}^* 
(\pt)}{\otimes} \End^*_{\mc D^b_{G \times \C^\times}(\mf g_N)}(K_N) \; \isom \\
& \hat{H}^*_{Z_G (\sigma) \times \C^\times} (\pt)_{\sigma,r} \underset{H_{Z_G (\sigma) \times
\C^\times}^* (\pt)} {\otimes} \End^*_{\mc D^b_{Z_G (\sigma) \times 
\C^\times}(\mf g_N^{\sigma,r})} (K_{N,\sigma,r}) .
\end{align*}
These induce equivalences of categories
\begin{align*}
\Modf{\sigma ,r} (\mh H (G,M,q\cE)) & \cong 
\Modf{\sigma ,r} \big( \End^*_{\mc D^b_{G \times \C^\times}(\mf g_N)}(K_N) \big) \\  
& \cong \Modf{\sigma ,r} \big( \End^*_{\mc D^b_{Z_G (\sigma) \times \C^\times}
(\mf g_N^{\sigma,r})}(K_{N,\sigma,r}) \big) ,
\end{align*}
where fl$,\sigma,r$ stands for finite length modules all whose irreducible subquotients admit the
central character given by $(\sigma,r)$.
\end{thm}

In particular Theorem \ref{thm:1.2} can be used to study all irreducible or standard \\
$\mh H (G,M,q\cE)$-modules with central character given by $(\sigma,r)$.

Since the semisimple complexes $K_N$ and $K_{N,\sigma,r}$ are so important in this paper,
we provide alternative descriptions. Consider the spaces and maps
\[
\cC_v^M \to \mf m_N \xleftarrow{\mr{pr}_{\mf m_N}} \mf m_N \oplus \mf u \to 
G \times^P  (\mf m_N \oplus \mf u) \xrightarrow{\mu_N} \mf g_N ,
\]
where $\mu_N (g,X) = (g, \mr{Ad}(g) X)$. By \cite[\S 2.6.3]{BeLu}, for any 
$P \times \C^\times$-variety $Y$ there is an equivalence of categories 
\[
\ind_{P \times \C^\times}^{G \times \C^\times} : 
\mc D^b_{P \times \C^\times} (Y) \to \mc D^b_{G \times \C^\times} (G \times^P Y) ,
\]
sometimes called equivariant induction. The functor
\[
\mu_{N,!} \ind_{P \times \C^\times}^{G \times \C^\times} \mr{pr}_{\mf m_N}^* :
\mc D^b_{P \times \C^\times} (\mf m_N) \to \mc D^b_{G \times \C^\times} (\mf g_N)
\]
can be regarded as a version of parabolic induction for equivariant constructible sheaves.
Notice that we may view $q\cE$ as a $P \times \C^\times$-equivariant sheaf on which $U$
acts trivially. By \cite[(2.30)--(2.33)]{SolSGHA} there is an isomorphism
\begin{equation}\label{eq:1.30}
K_N \cong \mu_{N,!} \ind_{P \times \C^\times}^{G \times \C^\times} \mr{pr}_{\mf m_N}^* 
\IC_{P \times \C^\times} (\mf m_N, q\cE) .
\end{equation}
Let $T_{\sigma,r}$ be the smallest algebraic torus in $G^\circ \times \C^\times$ whose
Lie algebra contains $(\sigma,r)$. It was checked in \cite[(2.9)]{SolKL} that
\begin{equation}\label{eq:1.31}
\mf g_N^{T_{\sigma,r}} = \mf g_N^{\sigma,r} \quad \text{and} \quad
\dot{\mf g}_N^{\sigma,r} = (\dot{\mf g}_N)^{T_{\sigma,r}} = \dot{\mf g} \cap
\big( \mf g_N^{\sigma,r} \times (G/P)^{\exp (\C \sigma)} \big) .
\end{equation}
We let $T_{\sigma,r}$ act on $G \times^P (\mf m_N \oplus \mf u)$ by
$(h,z) \cdot (g,X) = (hg, z^{-2} X)$, and we let 
\[
i_{N,\sigma,r} : (G \times^P \mf m_N \oplus \mf u)^{T_{\sigma,r}} \to
G \times^P (\mf m_N \oplus \mf u)
\]
be the inclusion.

\begin{lem}\label{lem:1.14}
In $\mc D^b_{Z_G (\sigma) \times \C^\times} (\mf g_N^{\sigma,r})$ we have isomorphisms 
\begin{align*}
K_{N,\sigma,r} & \cong \mr{pr}_{1,!} \IC_{Z_G (\sigma) \times \C^\times} 
\big(\mf g_N^{\sigma,r} \times (G/P)^{\exp (\C \sigma)}, j_{N,\sigma,r}^* \dot{q\cE}_N \big) \\
& \cong \mu_{N,!} i_{N,\sigma,r}^* \ind_{P \times \C^\times}^{G \times \C^\times} 
\mr{pr}_{\mf m_N}^* \IC_{P \times \C^\times} (\mf m_N, q\cE) .
\end{align*}
\end{lem}
\begin{proof}
Consider the commutative diagram
\begin{equation}\label{eq:1.33}
\xymatrix{
& (G \times^P \mf m_N \oplus \mf u )^{T_{\sigma,r}} \ar[dl]_{\mu_N} \ar[r]^{i_{N,\sigma,r}} &
G \times^P (\mf m_N \oplus \mf u ) \ar[dr]^{\mu_N} & \\
\mf g_N^{\sigma,r} = \mf g_N^{T_{\sigma,r}} \ar[r]^{\mr{pr}_1} & 
\dot{\mf g}_N^{\sigma,r} = (\dot{\mf g}_N )^{T_{\sigma,r}}  \ar[r]^{j_{N,\sigma,r}} 
\ar[u]^{j_{\mf m_N}} & \dot{\mf g}_N \ar[u]^{j_{\mf m_N}} \ar[r]^{\mr{pr}_1} & \mf g_N
}
\end{equation}
where $j_{\mf m_N} (X,gP) = (g, \mr{Ad}(g) X)$. With the left half of the diagram we rewrite
\begin{equation}\label{eq:1.32}
K_{N,\sigma,r} = \mr{pr}_{1,!} j^*_{N,\sigma,r} \dot{q\cE}_N = \mu_{N,!} j_{\mf m_N,!}
j^*_{N,\sigma,r} \dot{q\cE}_N .
\end{equation}
Like in \eqref{eq:1.31}, one checks that 
\begin{align*}
(G \times^P \mf m_N \oplus \mf u )^{T_{\sigma,r}} & \cong \{ (X,gP) \in \mf g_N 
\times G/P : \mr{Ad}(g^{-1}) X \in \mf m_N \oplus \mf u \}^{T_{\sigma,r}} \\
& \cong \mf g_N^{\sigma,r} \times (G/P)^{\exp (\C \sigma)} .
\end{align*}
Via this isomorphism $j_{\mf m_N}$ becomes $\dot{\mf g}_N^{\sigma,r} \to \mf g_N^{\sigma,r}
\times (G/P)^{\exp (\C \sigma)}$ and $\mu_N$ becomes $\mr{pr}_{1}$. As $j^*_{N,\sigma,r} 
\dot{q\cE}_N$ lies in $\mc D^b_{Z_G (\sigma) \times \C^\times} (\dot{\mf g}_N^{\sigma,r})$, 
this turns the right hand side of \eqref{eq:1.32} into 
\[
\mr{pr}_{1,!} \IC_{Z_G (\sigma) \times \C^\times} \big(\mf g_N^{\sigma,r} \times 
(G/P)^{\exp (\C \sigma)}, j_{N,\sigma,r}^* \dot{q\cE}_N \big) ,
\]
which proves the first isomorphism in the statement. 

The square in \eqref{eq:1.33} is Cartesian, so by base change \cite[Theorem 3.4.3]{BeLu},
\eqref{eq:1.32} becomes $\mu_{N,!} i^*_{N,\sigma,r} j_{\mf m_N,!} \dot{q\cE}_N$. By
\cite[(2.22)]{SolSGHA}, this is isomorphic to
\begin{multline*}
\mu_{N,!} i^*_{N,\sigma,r} j_{\mf m_N,!} \mr{ind}_{P \times \C^\times}^{G \times \C^\times}
\mr{pr}_{\mf m_N}^* q\cE \cong 
\mu_{N,!} i^*_{N,\sigma,r} \mr{ind}_{P \times \C^\times}^{G \times \C^\times} (j_{\mf m_N} 
|_{\cC_v^M \oplus \mf u} )_! \mr{pr}_{\mf m_N}^* q\cE \cong \\
\mu_{N,!} i^*_{N,\sigma,r} \mr{ind}_{P \times \C^\times}^{G \times \C^\times} 
\mr{pr}_{\mf m_N}^* (j_{\mf m_N} |_{\cC_v^M} )_!  q\cE \cong
 \mu_{N,!} i^*_{N,\sigma,r} \mr{ind}_{P \times \C^\times}^{G \times \C^\times} 
\mr{pr}_{\mf m_N}^* \IC_{P \times \C^\times} (\mf m_N, q\cE) . \qedhere
\end{multline*}
\end{proof}

Comparing \eqref{eq:1.30} and Lemma \ref{lem:1.14}, we see that $K_{N,\sigma,r}$ can
be regarded as a restricted parabolic induction of $\IC (\mf m_N, q\cE)$.

\section{The internal structure of standard modules}
\label{sec:internal}

From Theorem \ref{thm:1.2} one sees that all irreducible or standard $\mh H(G,M,q\cE)$-modules with
central character $(W_{q\cE} \sigma - r\sigma_v,r)$ arise in some way from the semisimple complex 
$K_{N,\sigma,r}$ on $\mf g_N^{\sigma,r}$. For $y \in \mf g_N^{\sigma,r}$ we write
\[
C_y = Z_C (y) = Z_{Z_{G \times \C^\times} (\sigma)}(y) = 
(Z_G (\sigma) \times \C^\times) \cap Z_{G \times \C^\times} (y) ,
\] 
where $G \times \C^\times$ acts as in \eqref{eq:1.8}. Let 
\[
\mc O_y = \mr{Ad}(C) y \subset \mf g_N^{\sigma,r}
\]
be the $C$-orbit of $y$. The equivalence of categories
\begin{equation}\label{eq:1.9}
\ind_{C_y}^C : \mc D^b_{C_y} (\{y\}) \to \mc D^b_C (\mc O_y)
\end{equation}
transforms any representation $\rho$ of $\pi_0 (C_y)$ into a $C$-equivariant local system on 
$\mc O_y$. We form the (equivariant) intersection cohomology complex $\IC_C (\mf g_N^{\sigma,r},
\ind_{C_y}^C (\rho))$, a $C$-equivariant perverse sheaf on $\mf g_N^{\sigma,r}$. In the literature
this object is often denoted $\IC_C (\mc O_y ,\ind_{C_y}^C (\rho))$, but we prefer a notation 
that specifies the variety on which it is defined.

\begin{thm}\label{thm:1.3} \textup{\cite[Theorem 4.2]{SolKL}} \\
Every simple direct summand of $K_{N,\sigma,r}$ is (up to a degree shift) isomorphic to
$\IC_C (\mf g_N^{\sigma,r},\ind_{C_y}^C (\rho))$, for $(y,\rho)$ such that $(y,\sigma,r,\rho)$ is
an enhanced L-parameter for $\mh H (G,M,q\cE)$. Conversely, for every such $(y,\sigma,r,\rho)$, 
$\IC_C (\mf g_N^{\sigma,r},\ind_{C_y}^C (\rho))$ is (up to a degree shift) a direct summand of 
$K_{N,\sigma,r}$.
\end{thm}

With Theorems \ref{thm:1.2} and \ref{thm:1.3} and techniques from \cite{Lus-Cusp2,AMS2},
we provide a description of the irreducible $\mh H (G,M,q\cE)$-modules.

\begin{lem}\label{lem:1.4}
The irreducible modules of $\End^*_{\mc D^b_C (\mf g_N^{\sigma,r})} (K_{N,\sigma,r})$ are 
\[
M_{y,\sigma,r,\rho} = \Hom^0_{\mc D^b_C (\mf g_N^{\sigma,r})} \big(
\IC_C (\mf g_N^{\sigma,r},\ind_{C_y}^C (\rho))', K_{N,\sigma,r} \big),
\]
where $(y,\sigma,r,\rho)$ is an enhanced L-parameter for $\mh H (G,M,q\cE)$.
Here the prime indicates a suitable degree shift, so that 
$\IC_C (\mf g_N^{\sigma,r},\ind_{C_y}^C (\rho))'$ becomes a direct summand of $K_{N,\sigma,r}$.
\end{lem}
\begin{proof}
We use a modified version $K'_{N,\sigma,r}$ of $K_{N,\sigma,r}$. The only change is that for
every simple direct summand of $K_{N,\sigma,r}$ the degrees are shifted, so that it becomes 
an actual perverse sheaf. Thus $K'_{N,\sigma,r}$ is a direct sum of simple perverse sheaves.

Then $\End^*_{\mc D^b_C (\mf g_N^{\sigma,r})} (K'_{N,\sigma,r})$ is naturally isomorphic to
$\End^*_{\mc D^b_C (\mf g_N^{\sigma,r})} (K_{N,\sigma,r})$ as algebras, only the gradings 
are different. Recall that
\begin{align}\label{eq:1.10}
& \End^*_{\mc D^b_C (\mf g_N^{\sigma,r})} (K'_{N,\sigma,r}) =
\End^*_{\mc D^b_{C^\circ} (\mf g_N^{\sigma,r})} (K'_{N,\sigma,r})^{C / C^\circ} , \\
\label{eq:1.99} & \End^0_{\mc D^b_C (\mf g_N^{\sigma,r})} \big(
\IC_C (\mf g_N^{\sigma,r}, \ind_{C_y}^C (\rho) ) \big) \cong
\End^0_{\mc D^b_{C_y} (\{y\})} (\rho) = \End_{\pi_0 (C_y)}(\rho) = \C .
\end{align}
From \cite[\S 5]{Lus-Cusp2} and \eqref{eq:1.10}--\eqref{eq:1.99} we see that:
\begin{itemize}
\item $\End^n_{\mc D^b_C (\mf g_N^{\sigma,r})} (K'_{N,\sigma,r}) = 0$ for $n < 0$,
\item $\C_{\sigma,r} \otimes_{H^*_{C^\circ}(\mr{pt})} \bigoplus_{n \in \Z_{>0}} \End^n_{\mc D^b_C 
(\mf g_N^{\sigma,r})} (K'_{N,\sigma,r})$ is the nilpotent radical of\\ $\C_{\sigma,r} 
\otimes_{H^*_{C^\circ}(\mr{pt})} \End^*_{\mc D^b_C (\mf g_N^{\sigma,r})} (K'_{N,\sigma,r})$,
\item $\End^0_{\mc D^b_C (\mf g_N^{\sigma,r})} (K'_{N,\sigma,r})$ is finite dimensional
and semisimple.
\end{itemize}
We conclude that the irreducible modules of 
$\End^*_{\mc D^b_C (\mf g_N^{\sigma,r})} (K_{N,\sigma,r})$ can be identified with the irreducible
modules of $\End^0_{\mc D^b_C (\mf g_N^{\sigma,r})} (K'_{N,\sigma,r})$. By Theorem \ref{thm:1.3},
those are the 
\[
\Hom^0_{\mc D^b_C (\mf g_N^{\sigma,r})} 
\big( \IC_C (\mf g_N^{\sigma,r},\ind_{C_y}^C (\rho)), K'_{N,\sigma,r} \big) ,
\]
where $(y,\sigma,r,\rho)$ is an enhanced L-parameter for $\mh H (G,M,q\cE)$. 
\end{proof}

With Theorem \ref{thm:1.3} and Lemma \ref{lem:1.4} we can decompose
\begin{equation}\label{eq:1.28}
K_{N,\sigma,r} \cong \bigoplus\nolimits_{y,\rho} 
\IC_C (\mf g_N^{\sigma,r}, \ind_{C_y}^C (\rho)) \otimes M_{y,\sigma,r,\rho} .
\end{equation}
We warn that, while this is a direct sum of constructible sheaves, there may be nonzero morphisms 
(in higher degrees) between different intersection cohomology complexes in the sum. 

Let $i_y : \{y\} \to \mf g_N^{\sigma,r}$ and $i_{\mc O_y} : \mc O_y \to \mf g_N^{\sigma,r}$ be the
inclusions. From \cite[\S 10]{Lus-Cusp2} and \cite[\S 3.2]{SolKL} we see that one way to define
the standard $\mh H (G,M,q\cE)$-modules is:
\[
\begin{array}{lll}
E_{y,\sigma,r} & = & H^* (\{y\}, i_y^! K_{N,\sigma,r}) ,\\
E_{y,\sigma,r,\rho} & = & \Hom_{\pi_0 (C_y)} \big( \rho, E_{y,\sigma,r} \big) .
\end{array}
\]
Here and later $H^* (?)$ is an abbreviation of $\bigoplus_{n \in \Z} H^n (?)$.
The action of $\mh H (G,M,q\cE)$ comes from Theorem \ref{thm:1.2} and 
$\End^*_{\mc D^b_C (\mf g_N^{\sigma,r})} (K_{N,\sigma,r})$, which acts via the natural 
homomorphism to $\End^*_{\mc D^b_{C_y} (\{y\})} (i_y^! K_{N,\sigma,r})$.
To define a filtration on $E_{y,\sigma,r}$, we first identify it with $H^* (\{y\}, i_y^!
K'_{N,\sigma,r})$ as $\End^*_{\mc D^b_C (\mf g_N^{\sigma,r})} (K'_{N,\sigma,r})$-module,
as in the proof of Lemma \ref{lem:1.4}. Generalizing \cite[\S 10.2]{Lus-Cusp2} we put, for 
$n \in \Z$:
\[
E^n_{y,\sigma,r} = H^n (\{y\}, i_y^! K'_{N,\sigma,r}) , \qquad
E^{\geq n}_{y,\sigma,r} = \bigoplus\nolimits_{n' \geq n} H^{n'} (\{y\}, i_y^! K'_{N,\sigma,r}) .
\]
Then $E^{\geq n}_{y,\sigma,r}$ is a submodule of $E_{y,\sigma,r}$, because the action of
$\End^*_{\mc D^b_C (\mf g_N^{\sigma,r})} (K'_{N,\sigma,r})$ is graded and
$\End^k_{\mc D^b_C (\mf g_N^{\sigma,r})} (K'_{N,\sigma,r}) = 0$ for $k < 0$. 
We note that
\begin{equation}\label{eq:1.91}
\begin{aligned}
& E^n_{y,\sigma,r} \text{ is a quotient of } E_{y,\sigma,r} \text{ if } n 
\text{ is minimal for } H^n (\{y\}, i_y^! K'_{N,\sigma,r}) \neq 0 ,\\
& E^n_{y,\sigma,r} \text{ is a submodule of } E_{y,\sigma,r} \text{ if } n 
\text{ is maximal for } H^n (\{y\}, i_y^! K'_{N,\sigma,r}) \neq 0 .
\end{aligned}
\end{equation}
From this filtration of $E_{y,\sigma,r}$ we get the graded module
\begin{equation}\label{eq:1.92}
\mr{gr} \, E_{y,\sigma,r} := 
\bigoplus\nolimits_{n \in \Z} E^{\geq n}_{y,\sigma,r} / E^{\geq n + 1}_{y,\sigma,r} .
\end{equation}
This means in effect that the action of $\End^k_{\mc D^b_C (\mf g_N^{\sigma,r})} (K'_{N,\sigma,r})$
with $k > 0$ is killed off. 
Let us denote equality in the Grothendieck of finite length $\End^*_{\mc D^b_C (\mf g_N^{\sigma,r})} 
(K_{N,\sigma,r})$-modules by $\dot{=}$. Like in \cite[\S 10.3]{Lus-Cusp2} we deduce from
\eqref{eq:1.28} and \eqref{eq:1.92} that
\begin{equation}\label{eq:1.11}
E_{y,\sigma,r} \,\dot{=}\, \mr{gr} \, E_{y,\sigma,r} = 
\bigoplus\nolimits_{y',\rho'} H^* \big( \{y\}, i_y^! \IC_C (\mf g_N^{\sigma,r},
\ind_{C_{y'}}^C (\rho') ) \big) \otimes M_{y',\sigma,r,\rho'} ,
\end{equation}
where the sum runs over all enhanced L-parameters $(y',\sigma,r,\rho')$ for $\mh H (G,M,q\cE)$.
On the right hand side of \eqref{eq:1.11}, $\mh H (G,M,q\cE)$ acts only on the tensor factors
$M_{y',\sigma,r,\rho'}$. 
The action of $\pi_0 (C_y)$ on $E_{y,\sigma,r}$ and $\mr{gr} \, E_{y,\sigma,r}$ corresponds to 
the natural action of $\pi_0 (C_y)$ on the tensor factors $H^* (\{y\}, \ldots)$ in \eqref{eq:1.11}.

Recall that $i_y^! = D i_y^* D$ where $D$ denotes the Verdier duality operator. For $\rho \in 
\Irr (\pi_0 (C_y)) \subset \mc D^b_{C_y} (\{y\})$ we have $D \rho = \rho^\vee$, the contragredient 
representation. The analogue of \eqref{eq:1.11} for $E_{y,\sigma,r,\rho}$ involves the space
\begin{equation}\label{eq:1.29}
\begin{aligned}
& \Hom_{\pi_0 (C_y)} \big( \rho, H^* \big( \{y\}, i_y^! \IC_C (\mf g_N^{\sigma,r},
\ind_{C_{y'}}^C (\rho') ) \big) \big) \cong \\
& \Hom_{\pi_0 (C_y)} \big( H^* \big( \{y\}, i_y^* D \, \IC_C (\mf g_N^{\sigma,r},
\ind_{C_{y'}}^C (\rho') ) \big), \rho^\vee \big) \cong \\
& \Hom_{\pi_0 (C_y)} \big( H^* \big( \{y\}, i_y^* \IC_C (\mf g_N^{\sigma,r},
\ind_{C_{y'}}^C (\rho'^\vee) ) \big), \rho^\vee \big) .
\end{aligned}
\end{equation}
By definition the dimension of the last line of \eqref{eq:1.29} is the multiplicity of $\rho^\vee$ in 
$i_y^* \IC_C (\mf g_N^{\sigma,r},\ind_{C_{y'}}^C (\rho'^\vee) )$. 
Equivalently the dimension of \eqref{eq:1.29} equals
\begin{equation}\label{eq:1.20}
\text{the multiplicity of } \ind_{C_y}^C (\rho^\vee) \text{ in }
i_{\mc O_y}^* \IC_C (\mf g_N^{\sigma,r},\ind_{C_{y'}}^C (\rho'^\vee) ).
\end{equation} 
By \cite[\S 10.6]{Lus-Cusp2}, or by applying complex conjugation to both sides, that equals
\begin{equation}\label{eq:1.12}
\text{the multiplicity } \mu (y,\rho,y',\rho') \text{ of } \ind_{C_y}^C (\rho) \text{ in }
i_{\mc O_y}^* \IC_C (\mf g_N^{\sigma,r},\ind_{C_{y'}}^C (\rho') ). 
\end{equation}
Notice that \eqref{eq:1.20} and \eqref{eq:1.12} can only be nonzero if $\mc O_y \subset 
\overline{\mc O_{y'}}$. As in \cite[Corollary 10.7]{Lus-Cusp2} and \cite[Proposition 5.1]{SolKL}, 
\eqref{eq:1.11} and \eqref{eq:1.12} determine the semisimplification of standard modules, namely
\begin{equation}\label{eq:1.13}
E_{y,\sigma,r,\rho} \;\dot{=}\, \text{ direct sum of the } M_{y',\sigma,r,\rho'} 
\text{ with multiplicities } \eqref{eq:1.12} .
\end{equation}
To obtain more information about the submodules of $E_{y,\sigma,r}$, we will provide 
a geometric construction of some special homomorphisms between standard modules, by
employing the underlying constructible sheaves in a more concrete way.

\begin{prop}\label{prop:1.5}
Let $y, y' \in \mf g_N^{\sigma,r}$.
\enuma{
\item If $r \neq 0$ and $\mc O_y \not\subset \overline{\mc O_{y'}}$, then
$\Hom_{\mh H (G,M,q\cE)} (E_{y',\sigma,r}, E_{y,\sigma,r}) = 0$.
\item Suppose that $\mc O_y \subset \overline{\mc O_{y'}}$. There is a homomorphism of 
$\mh H (G,M,q\cE)$-modules 
\[
J_{y',y} : E_{y',\sigma,r} \to E_{y,\sigma,r} ,
\]
canonical up to the action of $\pi_0 (C_{y'})$.}
\end{prop}
\begin{proof}
(a) As $r \neq 0$, by \cite[Theorems 3.20 and 4.6]{AMS2}, every irreducible quotient of 
$E_{y',\sigma,r}$ is isomorphic to $M_{y',\sigma,r,\rho'}$ for some enhancement $\rho'$. Let 
\[
\phi \in \Hom_{\mh H (G,M,q\cE)} (E_{y',\sigma,r}, E_{y,\sigma,r})
\]
and suppose that $\phi \neq 0$. Then $\ker (\phi)$ is a proper submodule,
so $E_{y',\sigma,r} / \ker (\phi)$ has at least one quotient of the form $M_{y',\sigma,r,\rho'}$.
Now $\phi$ induces an injection 
\[
E_{y',\sigma,r} / \ker (\phi) \to E_{y,\sigma,r}, 
\]
which in particular maps the quotient $M_{y',\sigma,r,\rho'}$ injectively to a subquotient of
$E_{y,\sigma,r}$. However, by \eqref{eq:1.13}, \eqref{eq:1.12} and the assumption, $E_{y,\sigma,r}$ 
does not have any subquotients isomorphic to $M_{y',\sigma,r}$. This contradiction shows that 
$\phi \neq 0$ is impossible.\\
(b) Let $K^\vee_{N,\sigma,r}$ be the version of $K_{N,\sigma,r}$ obtained from $q\cE^\vee$ 
instead of $q\cE$. We need to construct a homomorphism of 
$\End^*_{\mc D^b_C (\mf g_N^{\sigma,r})} (K_{N,\sigma,r})$-modules
\[
J_{y',y} : H^* (\{y'\}, i_{y'}^! K_{N,\sigma,r}) \to H^* (\{y\}, i_y^! K_{N,\sigma,r}) 
\]
Via Verdier duality, this is equivalent to the construction of a homomorphism of 
$\End^*_{\mc D^b_C (\mf g_N^{\sigma,r})} (K^\vee_{N,\sigma,r})$-modules
\begin{equation}\label{eq:1.14}
D J_{y',y} : H^* (\{y\}, i_y^* K^\vee_{N,\sigma,r}) \to H^* (\{y'\}, i_{y'}^* K^\vee_{N,\sigma,r}) .
\end{equation}
Recall that $K^\vee_{N,\sigma,r}$ is a bounded complex of $C$-equivariant constructible
sheaves with finite dimensional stalks on $\mf g_N^{\sigma,r}$. Since there are only finitely
many $C$-orbits in $\mf g_N^{\sigma,r}$ \cite[\S 5.4]{KaLu}, there exists a $C_y$-stable 
open neighborhood $U_y$ of $y$ in $\mf g_N^{\sigma,r}$ such that every section of 
$K^\vee_{N,\sigma,r}$ over $U_y$ is completely determined by its stalk at $y$. From 
$\mc O_y \subset \overline{\mc O_{y'}}$ we see that $U_y \cap \mc O_{y'}$ is nonempty. Pick 
$y_1 \in U_y \cap \mc O_{y'}$. Every element of the stalk $i_y^* K^\vee_{N,\sigma,r}$ comes 
from a unique section over $U_y$, so it determines an element 
of the stalk $i_{y_1}^* K^\vee_{N,\sigma,r}$. That yields $\End^*_{\mc D^b_C (\mf g_N^{\sigma,r})}
(K_{N,\sigma,r}^\vee)$-equivariant maps
\begin{align}
\label{eq:1.98} & i_y^* K^\vee_{N,\sigma,r} \to i_{U_y}^* K_{N,\sigma,r}^\vee \to 
i_{y_1}^* K^\vee_{N,\sigma,r} ,\\
\label{eq:1.15} &
H^* (\{y\}, i_y^* K^\vee_{N,\sigma,r}) \to H^* (\{y_1\}, i_{y_1}^* K^\vee_{N,\sigma,r}) .
\end{align}
Pick $g_1 \in C$ with $g_1 \cdot y_1 = y'$. The action of $g_1$ provides an isomorphism of 
$\End^*_{\mc D^b_C (\mf g_N^{\sigma,r})} (K^\vee_{N,\sigma,r})$-modules 
\begin{equation}\label{eq:1.16}
g_1 : H^* (\{y_1\}, i_{y_1}^* K^\vee_{N,\sigma,r}) \to H^* (\{y'\}, i_{y'}^* K^\vee_{N,\sigma,r}).
\end{equation}
The composition of \eqref{eq:1.15} and \eqref{eq:1.16} is the desired map \eqref{eq:1.14}.

It remains to analyse the dependence on the choices of $U_y, y_1$ and $g_1$. Consider different
$y_2 \in U_y$ and $g_2 \in C$ with $g_2 \cdot y_2 = y'$. Then
\begin{equation}\label{eq:1.17}
g_2^{-1} g_1 : H^* (\{y_1\}, i_{y_1}^* K^\vee_{N,\sigma,r}) \to 
H^* (\{y_2\}, i_{y_2}^* K^\vee_{N,\sigma,r}) 
\end{equation}
is an isomorphism, canonical up to multiplying $g_2$ on the right by elements of $C_{y_2}$. 
The isomorphism
\begin{equation}\label{eq:1.18}
g_2 : H^* (\{y_2\}, i_{y_2}^* K^\vee_{N,\sigma,r}) \to H^* (\{y'\}, i_{y'}^* K^\vee_{N,\sigma,r}) 
\end{equation}
is the canonical in the same sense. Equivalently, \eqref{eq:1.17} and \eqref{eq:1.18} are canonical
up to multiplying $g_2$ on the left by elements of $C_{y'}$. The constructibility of the involved
sheaves entails that this action of $C_{y'}$ factors through $\pi_0 (C_{y'})$. 

From this we deduce that the choice of $U_y$ was inessential. For any alternative $\tilde{U}_y$,
the intersection $U_y \cap \tilde{U}_y$ contains an open neighborhood of $y$ with the same property.
We can take $y_2$ in that smaller neighborhood, and by \eqref{eq:1.17} and \eqref{eq:1.18} that
is just as good as $y_1$. Altogether the only non-canonicity of \eqref{eq:1.14} comes from the
$\End^*_{\mc D^b_C (\mf g_N^{\sigma,r})} (K^\vee_{N,\sigma,r})$-linear action of $\pi_0 (C_{y'})$
on $H^* (\{y'\}, i_{y'}^* K^\vee_{N,\sigma,r})$.
\end{proof}

We note that at this point it is still possible that $E_{y,\sigma,r} = 0$ or $E_{y',\sigma,r} = 0$,
because it may be impossible to extend $(\sigma,y)$ or $(\sigma,y')$ to an enhanced L-parameter
for $\mh H (G,M,q\cE)$. Even when $E_{y,\sigma,r}$ and $E_{y',\sigma,r}$ have irreducible 
subquotients in common, the map $J_{y',y}$ from Proposition \ref{prop:1.5} could still be zero.
To improve on that, we bring in enhanced L-parameters $(y,\sigma,r,\rho)$
and $(y',\sigma,r,\rho')$. The next result shows that $\rho' = \mr{triv}$ behaves better than
other $\rho'$, so we focus on that.

Let $U_y \subset \mf g_N^{\sigma,r}$ be the $C_y$-stable neighborhood of $y$ from the proof
of Proposition \ref{prop:1.5}.b. We may assume that $y' \in U_y$. Let $(U_y \cap \mc O_{y'})_c$ 
be the connected component of $U_y \cap \mc O_{y'}$ containing $y'$ and let $S_{y,y'}$ be the
stabilizer of $(U_y \cap \mc O_{y'})_c$ in $C_y$. The connected group $C_y^\circ$ stabilizes
every connected component of $U_y \cap \mc O_{y'}$, so $C_y^\circ \subset S_{y,y'}$.

\begin{prop}\label{prop:1.16}
\enuma{
\item The $\pi_0 (C_y)$-representation $H^* \big( \{y\}, i_y^* \IC_C (\mf g_N^{\sigma,r}, 
\ind_{C_{y'}}^C (\mr{triv}) ) \big)$ can be embedded in a direct sum of copies of
$\ind_{S_{y,y'}}^{C_y} (\mr{triv})$.
\item Let $d$ be the degree in which $\IC_C (\mf g_N^{\sigma,r}, \ind_{C_{y'}}^C (\mr{triv}) )
|_{\mc O_{y'}}$ is nontrivial. Then
\[
H^d \big( \{y\}, i_y^*\, \IC_C (\mf g_N^{\sigma,r}, \ind_{C_{y'}}^C (\mr{triv}) ) \big) \cong
\ind_{S_{y,y'}}^{C_y} (\mr{triv}) \quad \text{as } \pi_0 (C_y)\text{-representations.}
\] 
}
\end{prop}
\begin{proof}
(a) We regard $\IC_C (\mf g_N^{\sigma,r}, \ind_{C_{y'}}^C (\mr{triv}) )$ as a (fixed)
complex of $C$-equivariant constructible sheaves. 
As arranged in the proof of Proposition \ref{prop:1.5}.b, $i_y^*$ gives a bijection
\begin{equation}\label{eq:2.80}
\{ \text{sections of } i_{U_y}^* \IC_C (\mf g_N^{\sigma,r}, \ind_{C_{y'}}^C (\mr{triv})) 
\text{ over } U_y \} \to
i_y^* \IC_C (\mf g_N^{\sigma,r}, \ind_{C_{y'}}^C (\mr{triv}))
\end{equation}
and every section over $U_y$ is determined by its values on $U_y \cap \mc O_{y'}$. 
Let $A$ be a connected component of $U_y \cap \mc O_{y'}$ and let $\bar A$ be its
closure in $U_y$. Pullback from $U_y$ to $\bar A$, for all possible $A$, provides a map
\begin{equation}\label{eq:2.81}
i_{U_y}^* \IC_C (\mf g_N^{\sigma,r}, \ind_{C_{y'}}^C (\mr{triv})) \to
\bigoplus\nolimits_A i^*_{\bar A} \IC_C (\mf g_N^{\sigma,r}, \ind_{C_{y'}}^C (\mr{triv})) ,
\end{equation}
which is injective on the level of sections. Since $\mc O_{y'}$ is one $C$-orbit, 
$C_y$ acts transitively on $\pi_0 (U_y \cap \mc O_{y'})$. Hence the right hand side of
\eqref{eq:2.81} is isomorphic, as $C_y$-representation, to 
\begin{equation}\label{eq:2.82}
\ind_{S_{y,y'}}^{C_y} \big( i^*_{\bar A} \IC_C (\mf g_N^{\sigma,r}, \ind_{C_{y'}}^C 
(\mr{triv})) \big) \qquad \text{with } A = (U_y \cap \mc O_{y'})_c .
\end{equation}
Consider $g \in S_{y,y'}$ and a section $s$ of 
$i^*_{\bar A} \IC_C (\mf g_N^{\sigma,r}, \ind_{C_{y'}}^C (\mr{triv}))$ over
$\bar A = \overline{(U_y \cap \mc O_{y'})_c}$. Then $g(y'') \in (U_y \cap \mc O_{y'})_c$ 
for all $y'' \in (U_y \cap \mc O_{y'})_c$ and, as the
intersection cohomology complex arises from a trivial $C$-equivariant local system,
\begin{equation}\label{eq:1.97}
s (g(y'')) \text{ equals } g (s(y'')) .
\end{equation}
Hence $S_{y,y'}$ acts trivially on such sections $s$.  It follows that the sections over
$\bar A$ in \eqref{eq:2.82} give rise to a $C_y$-representation isomorphic to a direct sum 
of copies of $\ind_{S_{y,y'}}^{C_y} (\mr{triv})$. By \eqref{eq:2.81} and \eqref{eq:2.80},
$i_y^* \IC_C (\mf g_N^{\sigma,r}, \ind_{C_{y'}}^C (\mr{triv}) )$
can be embedded in such a $C_y$-representation. Hence the same goes for its cohomology.\\
(b) We use some explicit information about intersection cohomology complexes from
trivial local systems, for which \cite{Rie} is a good reference. Let $d_y$ be the 
dimension of $\mc O_y$ over $\R$, and similarly for $d_{y'}$. Let $L$ be the link of $y$, 
another stratified manifold. By the axioms of a stratifications \cite[\S 2.1]{Rie}, 
$U_y$ can be written as
\begin{equation}\label{eq:1.96}
\R^{d_{y'} - d_y} \times \mr{cone}^\circ (L) ,
\end{equation}
where $\mr{cone}^\circ (L)$ denotes the open cone $L \times [0,\infty) / L \times \{0\}$.
Moreover, as $\mf g_N^{\sigma,r}$ is endowed with a smooth $C$-action, we may assume
that the diffeomorphism of $U_y$ with \eqref{eq:1.96} is $C_y$-equivariant. Now
$H^j \big( \{y\}, i_y^* \IC_C (\mf g_N^{\sigma,r}, \ind_{C_{y'}}^C (\mr{triv}) ) \big)$
is computed in terms of $L$ in \cite[\S 4.4]{Rie}. In degree $j = d$ we obtain
\begin{equation}\label{eq:1.95}
H_{d-1-d_{y'}+d_y} (L) \cong H_d \big( L \times \R^{d_{y'} - d_y + 1} \big) .
\end{equation}
Here we are still taking into account in which degree $\ind_{C_{y'}}^C (\mr{triv})$ is 
placed. That is responsible for a shift by $d$ degrees, so \eqref{eq:1.95} is actually
the $\pi_0 (C_y)$-representation $H_0 \big( L \times \R^{d_{y'} - d_y + 1} \big)$.
This depends only on $\pi_0 (L)$, which by \eqref{eq:1.96} is canonically in bijection 
with $\pi_0 (U_y \cap \mc O_{y'})$. Further $C_y$ permutes $\pi_0 (U_y \cap \mc O_{y'})$
transitively, with $S_{y,y'}$ as the stabilizer of one element. Hence the
$\pi_0 (C_y)$-representation \eqref{eq:1.95} can be expressed as
\[
H_0 (L) = \C [\pi_0 (U_y \cap \mc O_{y'})] \cong \ind_{S_{y,y'}}^{C_y} (\mr{triv}).
\qedhere
\]
\end{proof}

We note that Proposition \ref{prop:1.16} may fail with an arbitrary $\rho' \in \Irr (\pi_0 (C_y))$ 
instead of triv. Especially \eqref{eq:1.97} would cause problems, it could fail for 
$g \in C_y^0$, and then the constructibility of $i_y^* \, \IC_C \big( \mf g_N^{\sigma,r},
\ind_{C_{y'}}^C (\mr{triv})\big)$ would mean that $s$ cannot define a nonzero 
element there.

\begin{prop}\label{prop:1.6}
\enuma{
\item If the multiplicities \eqref{eq:1.20} and \eqref{eq:1.12} are zero and $r \neq 0$, then
\[
\Hom_{\mh H (G,M,q\cE)}(E_{y',\sigma,r,\rho'}, E_{y,\sigma,r,\rho}) = 0.
\]
\item Let $\rho'$ be the trivial representation of $\pi_0 (C_y)$. 
Suppose that the multiplicities $\mu (y,\rho,y',\rho')$ and \eqref{eq:1.20} are nonzero, 
so in particular $\mc O_y \subset \overline{\mc O_{y'}}$. Then $J_{y',y}$ induces a nonzero 
$\mh H (G,M,q\cE)$-module homomorphism from $E_{y',\sigma,r,\rho'}$ to the $\rho$-isotopic 
component of $E_{y,\sigma,r}$. In particular there exists a nonzero
$\mh H (G,M,q\cE)$-homomorphism from $E_{y',\sigma,r,\rho'}$ to $E_{y,\sigma,r,\rho}$.
}
\end{prop}
\begin{proof}
(a) This can be shown in the same way as Proposition \ref{prop:1.5}.a.\\
(b) The map $D J_{y',y}$ from \eqref{eq:1.14} sends the linear subspace
\begin{equation}\label{eq:1.19}
H^* \big (\{y\}, i_y^* \IC_C (\mf g_N^{\sigma,r},\ind_{C_{y'}}^C (\rho'^\vee) ) \big) 
\otimes M_{y',\sigma,r,\rho'^\vee} \; \subset \; H^* (\{y\}, i_y^* K^\vee_{N,\sigma,r})
\end{equation}
to the linear subspace
\begin{equation}\label{eq:1.27}
H^* \big( \{y'\}, i_{y'}^* \IC_C (\mf g_N^{\sigma,r},\ind_{C_{y'}}^C (\rho'^\vee) ) \big) 
\otimes M_{y',\sigma,r,\rho'^\vee} \cong \rho'^\vee \otimes M_{y',\sigma,r,\rho'^\vee}.
\end{equation}
It is the identity on $M_{y',\sigma,r,\rho'^\vee}$.
In the proof of Proposition \ref{prop:1.5}.b we can vary $g_2$ and $y_2$ used to construct
$D J_{y,y'}$, on \eqref{eq:1.19} that makes no difference because $\rho'$ is trivial.
Therefore $D J_{y,y'}$ is canonical and $\pi_0 (C_{y'})$-equivariant. 

From Proposition \ref{prop:1.16}.b we know that \eqref{eq:1.19} contains
\begin{equation}\label{eq:1.94}
H^d \big( \{y\}, i_y^* \IC_C (\mf g_N^{\sigma,r}, \ind_{C_{y'}}^C (\mr{triv}) ) \big)
\cong \ind_{S_{y,y'}}^{C_y} (\mr{triv}) 
\end{equation}
tensored with $M_{y',\sigma,r,\rho'^\vee}$. The construction in \eqref{eq:1.15}--\eqref{eq:1.16}
shows that on \eqref{eq:1.94} $D J_{y,y'}$ becomes the projection
\begin{equation}\label{eq:1.93}
\ind_{S_{y,y'}}^{C_y} (\mr{triv}) \to \mr{res}^{S_{y,y'}}_{C_{y'}} (\mr{triv}) = (\rho')^\vee 
\end{equation}
with kernel $\sum_{g \in C_y \setminus S_{y,y'}} g \, \mr{triv}_{S_{y,y'}}$.
As $\mu (y,\rho,y',\rho')$ and \eqref{eq:1.20} are (by assumption) nonzero, Proposition 
\ref{prop:1.16}.a shows that \eqref{eq:1.94} contains a copy of $\rho^\vee$. The projection of 
this $\rho^\vee$ on at least one of the subspaces $g \, \mr{triv}_{S_{y,y'}}$ of 
$\ind_{S_{y,y'}}^{C_y} (\mr{triv})$ is nonzero. All these subspaces are equally good, because
we may replace $y'$ by any element in its $C_y$-orbit. Therefore \eqref{eq:1.93} and
$D J_{y,y'}$ are nonzero on this copy of $\rho^\vee$. 

Via Verdier duality we see that the map
\[
J_{y,y'} : \rho' \otimes M_{y',\sigma,r,\rho'} \to H^* \big( \{y\}, i_y^! \IC_C 
(\mf g_N^{\sigma,r}, \ind_{C_{y'}}^C (\mr{triv}) ) \big) \otimes M_{y',\sigma,r,\rho'}
\]
is the identity on $M_{y',\sigma,r,\rho'}$ and nonzero on $\rho'$.
More precisely, when we project from $E_{y,\sigma,r}$ on its $\rho$-isotypic component
$V_\rho \otimes E_{y,\sigma,r,\rho}$, the map $J_{y,y'}$ composed with that projection sends
$\rho'$ nontrivially to a copy of $\rho$. Thus $J_{y,y'}$ induces a nonzero 
$\mh H (G,M,q\cE)$-module homomorphism from $E_{y',\sigma,r,\rho'} \supset \rho' 
\otimes M_{y',\sigma,r,\rho'}$ to $V_\rho \otimes E_{y,\sigma,r,\rho}$. 

We can compose that with a suitable linear map $V_\rho \to \C$ to obtain 
a nonzero module map from $E_{y',\sigma,r,\rho'}$ to $E_{y,\sigma,r,\rho}$.
\end{proof}

\section{Open parameters for twisted graded Hecke algebras}
\label{sec:open}

We say that an L-parameter $(y,\sigma,r)$ or $(y,\sigma,r,\rho)$ for $\mh H (G,M,q\cE)$ is 
open if the $Z_{G \times \C^\times}(\sigma)$-orbit of $y$ is open in $\mf g_N^{\sigma,r}$. 
We may also use $Z_G (\sigma)$ instead of $C = Z_{G \times \C^\times}(\sigma) = 
Z_G (\sigma) \times \C^\times$, for they have the same nilpotent orbits. Since there is a unique 
open orbit in $\mf g_N^{\sigma,r}$ \eqref{eq:1.21}, we could equivalently require that the 
$C$-orbit of $y$ is dense in $\mf g_N^{\sigma,r}$. 

For $r = 0$ we can reformulate the above condition in easier terms:
\begin{equation}\label{eq:1.23}
(y,\sigma,0) \text{ is open if and only if $y$ is regular nilpotent in } Z_{\mf g}(\sigma) .
\end{equation}
Sometimes it is useful to relax the notion of an open L-parameter. 

\begin{defn}\label{def:3.1}
We say that $(y,\sigma,r)$ is open with respect to $(G,M,q\cE)$ if
\[
\mr{Ad}(Z_{G \times \C^\times}(\sigma,r)) y \text{ is open in }
\{ y' \in \mf g_N^{\sigma,r} : (y',\sigma,r) \text{ is relevant for } (G,M,q\cE) \} .
\]
If in this situation $\rho$ is $(G,M,q\cE)$-relevant, then we also call $(y,\sigma,r,\rho)$ 
open with respect to $(G,M,q\cE)$.
\end{defn}

We recall that the condition on $y'$ in Definition \ref{def:3.1} is equivalent to: 
there exists a $\rho' \in \Irr \big( \pi_0 (Z_G (\sigma,y'))\big)$ such that the cuspidal 
quasi-support of $(y',\sigma,\rho')$ is $G$-conjugate to $(M,\cC_v^M,q\cE)$. 

Some algebras $\mh H(G,M,q\cE)$ do not have relevant open L-parameters, but they always admit 
L-parameters that are relevant with respect to $(G,M,q\cE)$.
Further, an open L-parameter is open with respect to $(G,M,q\cE)$ if only if it is relevant
with respect to $(G,M,q\cE)$.

\begin{lem}\label{lem:1.7}
Let $(y,\sigma,r,\rho)$ be an enhanced L-parameter for $\mh H (G,M,q\cE)$ which is open
with respect to $(G,M,q\cE)$.
Then $E_{y,\sigma,r,\rho}$ is irreducible and equals $M_{y,\sigma,r\rho}$.
\end{lem}
\begin{proof}
For connected $G$ and $(y,\sigma,r)$ open this is \cite[Corollary 10.9.c]{Lus-Cusp2}. We spell 
out that argument, to see that it applies in our generality. By \eqref{eq:1.13}
any constituent of $E_{y,\sigma,r,\rho}$ is of the form $M_{y',\sigma,r,\rho'}$, where 
$\mc O_y \subset \overline{\mc O_{y'}}$ and $y'$ is relevant with respect to $(G,M,q\cE)$. 
The ``open'' property of $y$ forces $\mc O_y = \mc O_{y'}$, so we may assume that $y' = y$. 
Then \eqref{eq:1.12} and \eqref{eq:1.13} show that $\ind_{C_y}^C (\rho)$ must appear in 
\[
i_{\mc O_y}^* \IC_C (\mf g_N^{\sigma,r},\ind_{C_{y'}}^C (\rho') ) = \ind_{C_y}^C (\rho')
\]
That forces $\ind_{C_y}^C (\rho) \cong \ind_{C_y}^C (\rho')$, which is equivalent to
$M_{y,\sigma,r,\rho'} \cong M_{y,\sigma,r,\rho}$. From \eqref{eq:1.13} and \eqref{eq:1.12} we see
that $M_{y,\sigma,r,\rho}$ appears with multiplicity one in $E_{y,\sigma,r,\rho}$. We conclude
that the two modules coincide. 
\end{proof}

\noindent
The main result of this section is a quick consequence of the insights collected so far.

\begin{thm}\label{thm:1.8}
Let $(y,\sigma,r,\rho)$ be an enhanced L-parameter for $\mh H (G,M,q\cE)$. Suppose that 
$(y',\sigma,r)$ is an open L-parameter for $\mh H (G,M,q\cE)$, such that 
$M_{y',\sigma,r,\mr{triv}}$ occurs as a subquotient of $E_{y,\sigma,r,\rho}$. 
Then $M_{y',\sigma,r,\mr{triv}}$ is isomorphic to a submodule of $E_{y,\sigma,r,\rho}$.
\end{thm}
\begin{proof}
From \eqref{eq:1.13} we see that the multiplicity \eqref{eq:1.12} is nonzero and that
$\mc O_y \subset \overline{\mc O_{y'}}$. By Proposition \ref{prop:1.6}.b there exists a nonzero 
$\mh H (G,M,q\cE)$-module homomorphism from $E_{y',\sigma,r,\mr{triv}}$ to $E_{y,\sigma,r,\rho}$.
By Lemma \ref{lem:1.7} $E_{y',\sigma,r,\mr{triv}} = M_{y',\sigma,r,\mr{triv}}$
is irreducible, so this homomorphism is injective.
\end{proof}

The parameters $(y,\sigma,r,\rho)$ can also be presented in another way. By \cite[\S 2.4]{KaLu}
we can find a homomorphism of algebraic groups 
\begin{equation}\label{eq:1.22}
\begin{aligned}
\gamma_y : & \; SL_2 (\C) \to G^\circ \text{ such that:}\\
& \bullet \; \textup{d}\gamma_y ( \matje{0}{1}{0}{0} ) = y,\\
& \bullet \; \sigma_0 := \sigma - \textup{d}\gamma_y ( \matje{r}{0}{0}{-r} ) 
\text{ commutes with the image of d}\gamma_y. \\
\end{aligned}
\end{equation} 
Moreover the $G^\circ$-conjugacy class of $(y,\sigma,r)$ determines the $G^\circ$-conjugacy
class of $(y,\sigma_0,r)$ and conversely. We recall from \cite[Lemma 3.6]{AMS2} that there is
a na\-tu\-ral isomorphism $\pi_0 (C_y) \cong \pi_0 (Z_G (\sigma_0,y))$, and that the data
$(y,\sigma,r,\rho)$ up to $G$-association carry exactly the same information as $(y,\sigma_0,r,\rho)$
up to $G$-association. 

\begin{lem}\label{lem:1.12}
Let $(y,\sigma,r)$ be an L-parameter for $\mh H (G,M,q\cE)$, with $r \neq 0$. Equivalent are:
\begin{itemize}
\item[(i)] $(y,\sigma,r)$ is open,
\item[(ii)] $\mf g_N^{\sigma,r}$ equals 
$\{ X \in Z_{\mf g}(\sigma_0) : [\sigma - \sigma_0,X] = 2rX \}$
\end{itemize}
\end{lem}
\begin{proof}
Since $\mf g_N^{\sigma,r} = \mf g_N^{r^{-1} \sigma, 1}$ and $Z_{G^\circ}(\sigma) = Z_{G^\circ}
(r^{-1} \sigma)$, we may replace $(y,\sigma,r)$ by $(y,r^{-1} \sigma,1)$ and assume that $r = 1$.\\
(i) $\Rightarrow$ (ii)
Pick a maximal toral subalgebra $\mf t'$ of $\mf g$ containing $\sigma$, and let $R$ be the root
system of $(\mf g,\mf t')$. Then
\[
\mf g_N^{\sigma,1} = \bigoplus\nolimits_{\alpha \in R : \alpha (\sigma) = 2} \mf g_\alpha . 
\]
The open $Z_{G^\circ}(\sigma)$-orbit in $\mf g_N^{\sigma,1}$ contains an element with nonzero
parts in every root subspace $\mf g_\alpha$ with $\alpha (\sigma) = 2$. Conjugating $y$ by
an element of $Z_{G^\circ}(\sigma)$ if necessary, we may assume that $y$ has this property.
Let $R'$ be the minimal parabolic root subsystem of $R$ containing $\{ \alpha \in R : 
\alpha (\sigma) = 2\}$. It gives a Levi subalgebra
\[
\mf g' = \mf t \oplus \bigoplus\nolimits_{\alpha \in R'} \mf g_\alpha ,
\]
which contains $y$ as a distinguished nilpotent element. Let $G'$ be the algebraic Lie subgroup 
of $G^\circ$ with Lie algebra $\mf g'$, and let $\gamma_y : SL_2 (\C) \to G'$ be as in 
\eqref{eq:1.22}. As $Z_{\mf g'}(\sigma_0)$ is a Levi subalgebra of $\mf g'$ containing $y$,
\begin{equation}\label{eq:1.24}
\text{the distinguishedness of $y$ forces } \sigma_0 \in Z (\mf g') \subset \mf t' .
\end{equation} 
In particular any root $\alpha$ with $\alpha (\sigma) = 2$ satisfies $\alpha (\sigma_0) = 0$ and
$\alpha (\sigma - \sigma_0) = 2$. Hence (ii) holds for this specific $\sigma_0$. Since $\sigma_0$
is determined by $(y,\sigma,r)$ up to $G^\circ$-conjugacy, (ii) holds for all possible $\sigma_0$.\\
(ii) $\Rightarrow$ (i)
Notice that $y, \textup{d}\gamma_y (\matje{0}{0}{1}{0}), \textup{d}\gamma_y (\matje{1}{0}{0}{-1})$
is an $\mf{sl}_2$-triple in $Z_{\mf g}(\sigma_0 )$. It is known from \cite[Lemma 4.2.c]{Kos} or 
\cite[Lemma 3.7.24]{ChGi} that the orbit of $Z_{G^\circ}(\sigma_0) \cap Z_{G^\circ} 
\big( \textup{d}\gamma_y \matje{1}{0}{0}{-1}\big)$ through $y$ is open in 
\[
\{ X \in Z_{\mf g}(\sigma_0) : [ \textup{d}\gamma_y (\matje{1}{0}{0}{-1}),X] = 2X \} =
\{ X \in Z_{\mf g}(\sigma_0) : [\sigma - \sigma_0,X] = 2X \} = \mf g_N^{\sigma,1}.
\]
The group $Z_{G^\circ}(\sigma)$ contains $Z_{G^\circ}(\sigma_0) \cap Z_{G^\circ} 
\big( \textup{d}\gamma_y \matje{1}{0}{0}{-1}\big)$, so its orbit through $y$ is also open in 
$\mf g_N^{\sigma,1}$. 
\end{proof}

By \cite[Proposition 1.4]{AMS3} every enhanced L-parameter for $\mh H (G,M,q\cE)$ is
$G$-associated to one with $\sigma_0 \in \mf t = Z(\mf m), \sigma \in \mf m$ and 
$\mr{Sc}(y,\sigma_0,\rho)_{Z_G (\sigma_0)}$ represented by\\ $(M,v,\sigma_0, (q\cE)_v)$. 
We can reinterpret Definition \ref{def:3.1} in terms of $\sigma_0$:

\begin{lem}\label{lem:3.A}
A L-parameter $(y,\sigma,r)$ is open with respect to $(G,M,q\cE)$ if and only if
$\mr{Ad}(Z_G (\sigma_0)^\circ) y$ contains a dense subset of $\cC_v^M \oplus Z_{\mf u}(\sigma_0)$.
\end{lem}
\begin{proof}
Let $\epsilon$ be an irreducible constituent of $(q\cE)_v$ as 
$\pi_0 (Z_{M^\circ}(v))$-representation. The construction of quasi-cuspidal supports in
\cite[\S 5]{AMS1} shows that $(y',\sigma,r,\rho')$ is relevant for $\mh H (G,M,q\cE)$ if and only
if $(y',\sigma,r,\rho'^\circ)$ is relevant for $\mh H (G^\circ, M^\circ, \epsilon)$ for some
irreducible $\pi_0 (Z_{G^\circ}(\sigma_0,y'))$-subrepresentation $\rho'^\circ$ of $\rho'$.

Therefore it suffices to prove the lemma when $G$ is connected. Then \cite[Proposition 1.4]{AMS3},
which is based on \cite[\S 8]{Lus-Cusp1}, shows that
\begin{equation}\label{eq:3.B}
\mr{Ad}(Z_G (\sigma_0))(v + Z_{\mf u}(\sigma_0)) = 
\mr{Ad}(Z_G (\sigma_0)) (\cC_v^M + Z_{\mf u}(\sigma_0))
\end{equation}
is precisely the set of $y$ such that $(y,\sigma,r)$ (and equivalently $(y,\sigma_0,r))$ is
$(G,M,q\cE)$-relevant. Hence $(y,\sigma,r)$ is open with respect to $(G,M,q\cE)$ if and only if
Ad$(Z_G (\sigma_0)) y$ is dense in \eqref{eq:3.B}. That is equivalent to: Ad$(Z_G (\sigma_0)) y$
contains a dense subset of $\cC_v^M \oplus Z_{\mf u}(\sigma_0)$.
\end{proof}

An L-parameter $(y,\sigma,r)$ or $(y,\sigma_0,r)$ (or with $\rho$ included) is called bounded if
a $G$-conjugate of $\sigma_0$ lies in $i\R \otimes_\Z X_* (T)$. We recall that by 
\cite[Lemma 2.2]{SolKL} Ad$(G)\sigma_0$ intersects $\mf t = \C \otimes_\Z X_* (T)$ whenever
there exists an enhanced L-parameter for $\mh H (G,M,q\cE)$ with this $\sigma$. To explain the
terminology, we note that $\exp (i\R \otimes_\Z X_* (T))$ is the maximal compact subgroup of $T$.
Thus a parameter is bounded if and only if $\exp (\sigma_0)$ lies in a bounded closed subgroup 
of $G$. More generally we say that $(y,\sigma,r)$ is essentially bounded if a $G$-conjugate of 
$\sigma_0$ lies in $Z(\mf g) \oplus ( i\R \otimes_\Z X_* (T) )$.

The next result is a variation on a property of Langlands parameters for reductive groups,
announced in \cite[\S 0.6]{CFZ}.

\begin{lem}\label{lem:1.9}
Let $(y,\sigma,r)$ be an L-parameter for $\mh H (G,M,q\cE)$.
\enuma{
\item If $y$ is distinguished in $\mf g$, then $(y,\sigma,r)$ is essentially bounded.
\item If $\Re (r) \neq 0$ and $(y,\sigma,r)$ is essentially bounded, 
then it is an open parameter.
}
\end{lem}
\begin{proof}
(a) By the same reasons as for \eqref{eq:1.24}, $\sigma_0$ lies in $Z(\mf g)$.\\
(b) By \cite[Proposition 1.4]{AMS3} we may assume that $\sigma_0 ,\sigma - r\sigma_v \in \mf t$. 
Let $T'$ be a maximal torus of $Z_G (\sigma_0)$ whose Lie algebra $\mf t'$ contains
$\mf t + \C \sigma = \mf t + \C \sigma_v$. Then 
\[
\sigma = \sigma_0 + \textup{d}\gamma_y ( \matje{r}{0}{0}{-r} ) 
\text{ with } \sigma_0 \in Z(\mf g) \oplus i\R \otimes_\Z X_* (T) \text{ and } 
\textup{d}\gamma_y ( \matje{r}{0}{0}{-r} ) \in \R \otimes_\Z X_* (T') .
\]
In particular $\sigma_0$ takes imaginary values on all roots of $(\mf g,\mf t')$, while 
$\sigma - \sigma_0 = \textup{d}\gamma_y ( \matje{r}{0}{0}{-r} )$ takes values in $r \R$ on all roots.
As $r \R \cap i \R = \{0\}$, (ii) from Lemma \ref{lem:1.12} holds, and we conclude by applying 
Lemma \ref{lem:1.12}.
\end{proof}

We warn that Lemmas \ref{lem:1.12} and \ref{lem:1.9}.b are false for $r=0$. Just take $\sigma = 0$
and note that $(y,0,0)$ is a bounded parameter satisfying (ii), for any nilpotent element $y$.

Bounded parameters are related to $\mh H (G,M,q\cE)$-modules which are tempered in the sense of
\cite[Definition 3.24]{AMS2}. The definition says that $\mc O (\mf t)$-weights of a module must lie
in a certain negative cone. In particular any subquotient of a tempered module is again tempered.

Besides tempered representations, in an important role in harmonic analysis is played by
(essentially) discrete series representations. For graded Hecke algebras they are also defined
in \cite[Definition 3.24]{AMS2}, it says that their weights must belong to the interior of a
suitable negative cone. 

To see the connection between tempered representations and bounded parameters best, we involve the 
sign automorphism of $\mh H (G,M,q\cE)$. Let $\det : W_{q\cE} \to \{\pm 1\}$ be the determinant of 
the action of $W_{q\cE}$ on $X^* (T)$, an extension of the sign character of the Weyl group 
$W^\circ_{q\cE}$. Then $\sgn : \mh H (G,M,q\cE) \to \mh H (G,M,q\cE)$ is defined by
\[
\sgn (N_w) = \det (w) N_w ,\; \sgn (\mb r) = - \mb r ,\; \sgn (\xi) = \xi 
\qquad w \in W_{q\cE}, \xi \in \mc O (\mf t \oplus \C) .
\]

\begin{thm}\label{thm:1.10} 
\textup{\cite[\S 3.5 and \S 4]{AMS2} and \cite[Theorem 3.4]{SolKL}} 
\enuma{
\item Let $(y,\sigma,r,\rho)$ be an enhanced L-parameter for $\mh H (G,M,q\cE)$.
The following are equivalent when $\Re (r) \leq 0$:
\begin{itemize}
\item the parameter $(\sigma,y,r)$ is bounded,
\item $E_{y,\sigma,r,\rho}$ is tempered,
\item $M_{y,\sigma,r,\rho}$ is tempered.
\end{itemize}
When $\Re (r) < 0$, $M_{y,\sigma,r,\rho}$ is essentially discrete series if and only if
$y$ is distinguished nilpotent in $\mf g$.
\item Let $(y,\sigma,-r,\rho)$ be an enhanced L-parameter for $\mh H (G,M,q\cE)$.
The following are equivalent when $\Re (r) \geq 0$:
\begin{itemize}
\item the parameter $(\sigma,y,-r)$ is bounded,
\item $\sgn^* E_{y,\sigma,-r,\rho}$ is tempered,
\item $\sgn^* M_{y,\sigma,-r,\rho}$ is tempered.
\end{itemize}
When $\Re (r) > 0$, $\sgn^* M_{y,\sigma,-r,\rho}$ is essentially discrete series if and only if
$y$ is distinguished nilpotent in $\mf g$.
\item When $\Re (r) = 0$, $\mh H (G,M,q\cE)$ does not have nonzero essentially discrete modules 
on which $\mb r$ acts as $r$.
}
\end{thm}

We will refer to $\sgn^* E_{y,\sigma,-r,\rho}$ as an analytic standard module. These are useful
because often $\mb r$ has to be specialized to a positive real number $r$, and then Theorem 
\ref{thm:1.10}.b yields tempered or essentially discrete $\mh H (G,M,q\cE)$-modules on which 
$\mb r$ acts as $r$. To emphasize the contrast, we will sometime call $E_{y,\sigma,r,\rho}$
a geometric standard module.

There are yet other standard modules, namely those appearing in the Langlands classification for
graded Hecke algebras \cite{Eve}. To construct those one starts with an irreducible tempered 
representation $\tau$ of a parabolic subalgebra $\mh H^P$ of $\mh H$ (determined by a set of
simple roots $P$). One twists $\tau$ by a character $t$ in positive position with respect to the 
set of simple roots outside $P$, and then induces to $\mh H$.

This works when $\Gamma_{q\cE}$ is trivial. When we include a nontrivial $\Gamma_{q\cE}$ or
$\C [\Gamma_{q\cE},\natural_{q\cE}]$ we have to be careful because it might mess up the uniqueness
of irreducible quotients in the Langlands classification. A solution is provided by
\cite[Corollary 2.2.5]{SolAHA} and \cite[(8.11)]{SolEnd}: one must choose the subgroup of 
$\Gamma_{q\cE}$ that occurs in a parabolic subalgebra of 
\[
\mh H (G,M,q\cE) = \mh H (\mf t,W_{q\cE}, k, \mb r, \natural_{q\cE})
\]
depending on the data $\tau,t$. Namely, one takes the largest subgroup $\Gamma_{P,t}$ of
$\Gamma_{q\cE}$ that stabilizes $P$ and $t$. Next one replaces $\tau$ by an irreducible 
representation $\tau'$ of $\mh H (\mf t,W_P \Gamma_{P,t}, k,\mb r,\natural_{q\cE})$ whose 
restriction to $\mh H^P$ contains $\tau$, or equivalently an irreducible direct summand of
$\ind_{\mh H (\mf t,W_P,k,\mb r)}^{\mh H (\mf t,W_P \Gamma_{P,t},k,\mb r,\natural_{q\cE})} \tau$.
We call modules of the form 
\begin{equation}\label{eq:1.25}
\ind_{\mh H (\mf t,W_P \Gamma_{P,t},k,\mb r,\natural_{q\cE})}^{
\mh H (\mf t,W_{q\cE}, k, \mb r, \natural_{q\cE})} (\tau' \otimes t)
\end{equation}
Langlands standard modules. With Clifford theory, in the version of \cite[\S 11]{SolGHA} and 
\cite[\S 1]{AMS3}, we can write
\[
\tau' = \ind^{\mh H (\mf t,W_P \Gamma_{P,t}, k,\mb r,\natural_{q\cE})}_{
\mh H (\mf t,W_P \Gamma_{P,t,\tau}, k,\mb r,\natural_{q\cE})} (\rho \otimes \tau) 
\]
for a suitable projective representation $\rho$ of $\Gamma_{P,t,\tau}$. In that notation, 
\eqref{eq:1.25} becomes 
\[
\ind_{\mh H (\mf t,W_P \Gamma_{P,t,\tau},k,\mb r,\natural_{q\cE})}^{
\mh H (\mf t,W_{q\cE}, k, \mb r, \natural_{q\cE})} (\rho \otimes \tau \otimes t) =
\ind_{\mh H (\mf t,W_\cE \Gamma_{P,t,\tau},k,\mb r,\natural_{q\cE})}^{
\mh H (\mf t,W_{q\cE}, k, \mb r, \natural_{q\cE})} \big( \rho \otimes
\ind_{\mh H (\mf t,W_P,k,\mb r)}^{\mh H (\mf t,W_\cE, k, \mb r)} (\tau \otimes t) \big) .
\]
This shows that every Langlands standard module of 
$\mh H (\mf t,W_{q\cE}, k, \mb r, \natural_{q\cE})$ is an indecomposable direct summand of
\begin{equation}\label{eq:1.26}
\ind^{\mh H (\mf t,W_{q\cE}, k, \mb r, \natural_{q\cE})}_{\mh H (\mf t,W_\cE, k, \mb r)} V
\end{equation}
for some Langlands standard $\mh H (\mf t,W_\cE, k, \mb r)$-module $V$, namely 
$\ind_{\mh H (\mf t,W_P,k,\mb r)}^{\mh H (\mf t,W_\cE, k, \mb r)} (\tau \otimes t)$ above.
If $Q \subset G$ is the standard parabolic subgroup such that $W_{q\cE}^Q = W_P \mc R_{P,t}$, then
\[
\mh H (Q,M,q\cE) = \mh H (\mf t,W_P \Gamma_{P,t}, k,\mb r,\natural_{q\cE}) 
\]
and the Langlands standard module \eqref{eq:1.25} depends only on the data $(Q,\tau',t)$ 
up to $G$-conjugacy.
The relations between the various kinds of standard $\mh H (G,M,q\cE)$-modules are as follows.

\begin{prop}\label{prop:1.13}
Let $r \in \C^\times$.
\enuma{
\item When $\Re (r) > 0$, the Langlands standard modules of $\mh H (G,M,q\cE) / (\mb r - r)$ 
are precisely the analytic standard modules.
\item When $\Re (r) < 0$, the Langlands standard modules of $\mh H (G,M,q\cE) / (\mb r - r)$ 
are precisely the geometric standard modules.
\item The Langlands standard modules of $\mh H (G,M,q\cE) / (\mb r)$ are precisely its 
irreducible modules.
}
\end{prop}
\begin{proof}
(a) This is shown in \cite[Proposition B.4]{SolKL}.\\
(b) This can be shown in the same way as \cite[Proposition B.4]{SolKL}, just apply $\sgn^*$
to all the modules in the proof of parts (c) and (d).\\
(c) By Clifford theory in the form \cite[Lemma 8.4]{SolEnd}, the set of irreducible 
representations of
\[
\mh H (G,M,q\cE) / (\mb r) = \mc O (\mf t) \rtimes \C [W_{q\cE},\natural_{q\cE}]
\]
with central character $W_{q\cE} c$ is naturally in bijection with $\Irr (\C [(W_{q\cE})_c,
\natural_{q\cE}])$, via $\ind_{\mc O (\mf t) \rtimes \C [(W_{q\cE})_c,
\natural_{q\cE}]}^{\mc O (\mf t) \rtimes \C [(W_{q\cE}),\natural_{q\cE}]}$. 
Consider a Langlands standard module
\[
V = \ind_{\mh H (\mf t,W_P \Gamma_{P,t},k,\mb r,\natural_{q\cE})}^{
\mh H (\mf t,W_{q\cE}, k, \mb r, \natural_{q\cE})} (\tau' \otimes t) = 
\ind_{\mc O (\mf t) \rtimes \C [W_P \Gamma_{P,t},\natural_{q\cE}]}^{
\mc O (\mf t) \rtimes \C [W_{q\cE},\natural_{q\cE}]} (\tau' \otimes t) .
\]
Since $\tau' \in \Irr (\mc O (\mf t) \rtimes \C [W_P \Gamma_{P,t},\natural_{q\cE}])$ is tempered, 
the real part of its central character $W_P \Gamma_{P,t}c$ is 0. Hence the real part of the central 
character $W_P \Gamma_{P,t}ct$ of $\tau' \otimes t$ is $t$. Furthermore $t$ is positive with 
respect to the simple roots outside $P$, so 
\[
(W_{q\cE})_t \subset \mr{Stab}_{W_{q\cE}} (\Z P, t) \subset W_P \mr{Stab}_{\Gamma_{q\cE}}(\Z P, t) 
= W_P \Gamma_{P,t}.
\] 
Now \cite[Lemma 8.4]{SolEnd} says that $\Irr (\C [(W_{q\cE})_{ct},\natural_{q\cE}])$ parametrizes
the irreducible modules with central character represented by $ct$ of both 
$\mc O (\mf t) \rtimes \C [W_{q\cE},\natural_{q\cE}]$ and 
$\mc O (\mf t) \rtimes \C [W_P \Gamma_{P,t},\natural_{q\cE}]$. It follows that
$\ind_{\mc O (\mf t) \rtimes \C [W_P \Gamma_{P,t},\natural_{q\cE}]}^{
\mc O (\mf t) \rtimes \C [W_{q\cE},\natural_{q\cE}]}$ preserves the irreducibility of 
$\tau' \otimes t$, so $V$ is irreducible.

By the Langlands classification from \cite{Eve}, in the form \cite[Proposition 8.5]{SolEnd}, 
every irreducible $\mh H (G,M,q\cE)/(\mb r)$-module $V'$ occurs as a quotient
of some Langlands standard module $V$. By the irreducibility of $V$ we have $V' = V$.
\end{proof}

\section{Generic representations of graded Hecke algebras}
\label{sec:generic}

In this section we assume that the 2-cocycle $\natural_{q\cE}$ involved in $\mh H (G,M,q\cE)$
is trivial. Then \eqref{eq:1.3} simplifies to
\begin{equation}\label{eq:2.1}
\mh H (G,M,q\cE) = \mh H (G^\circ,M^\circ,\cE) \rtimes \Gamma_{q\cE} ,
\end{equation} 
see \cite[\S 4]{AMS2}. In other words, $\mh H (G,M,q\cE)$ is a graded Hecke algebra extended
with a finite group. The triviality of $\natural_{q\cE}$ is known when $\mh H (G,M,q\cE)$ arises: 
\begin{itemize}
\item[(i)] from an extended affine Hecke algebra $\mc H \rtimes \Gamma$
via localization, as in \cite{Lus-Gr,SolAHA,AMS3},
\item[(ii)] from a classical $p$-adic group \cite{Hei,AMS4},
\item[(iii)] from a Bernstein component of a reductive $p$-adic group, such that the
underlying supercuspidal representations are simply generic \cite[Theorem~A.1]{OpSo}.
\end{itemize}
In the references to (ii) and (iii) this is shown for the relevant extended affine Hecke 
algebras, and then one can apply (i).

Recall that $\det : W_{q\cE} \to \{\pm 1\}$ denotes the determinant of the action of $W_{q\cE}$
on $X^* (T)$. It can also be regarded as a onedimensional representation of $\C [W_{q\cE}]$.
We say that 
\begin{equation}\label{eq:2.13}
\text{a } \mh H (G,M,q\cE)\text{-module } V \text{ is generic if }
\Res^{\mh H (G,M,q\cE)}_{\C [W_{q\cE}]} V \text{ contains det.}
\end{equation} 
This compares well with the definition of genericity for representations of extended affine 
Hecke algebras, see Theorem \ref{thm:3.1}.
Like for quasi-split reductive groups \cite{Rod,Shal}, there are multiplicity one properties
for generic representations of extended graded Hecke algebras.

\begin{prop}\label{prop:2.2}
Let $Q \subset G$ be a quasi-Levi subgroup containing $M$, so that $\mh H (Q,M,q\cE)$ is a
parabolic subalgebra of $\mh H (G,M,q\cE)$. Let $(\pi,V)$ be a $\mh H (Q,M,q\cE)$-module.
\enuma{
\item The multiplicity of $\det$ in $\Res^{\mh H (G,M,q\cE)}_{\C [W_{q\cE}} \big( 
\ind_{\mh H (Q,M,q\cE)}^{\mh H (G,M,q\cE)} V \big)$ equals the multiplicity of $\det$ in $V$,
as representations of the version $W_{q\cE}^Q$ of $W_{q\cE}$ for $(Q,M,q\cE)$. In particular
$V$ is generic if and only if $\ind_{\mh H (Q,M,q\cE)}^{\mh H (G,M,q\cE)} V$ is generic.
\item Suppose that $(\pi,V)$ is irreducible and generic. Then 
\[
\dim \Hom_{\C [W_{q\cE}]} \big( \ind_{\mh H (Q,M,q\cE)}^{\mh H (G,M,q\cE)} V, \det \big) = 1
\]
and $\ind_{\mh H (Q,M,q\cE)}^{\mh H (G,M,q\cE)} V$ has a unique generic irreducible
subquotient, appearing with multiplicity one.
\item $\dim \Hom_{\C [W_{q\cE}]} (\sgn^* (E_{y,\sigma,r,\rho}), \det )\leq 1$ for every 
enhanced L-parameter $(y,\sigma,r,\rho)$ for $\mh H(G,M,q\cE)$.
}
\end{prop}
\begin{proof}
(a) and (b) These can be shown in the same way as for extended affine Hecke algebras, see 
\cite[Lemma 3.5]{SolQS} and \cite[Lemma 7.2]{OpSo}. Alternatively, one can apply 
\cite[Theorems 6.1 and 6.2]{SolQS} to \cite[Lemma 3.5]{SolQS}.\\
(c) Conjugating the parameters by a suitable element of $N_{G^\circ} (T)$, we may assume that
$\Re (\sigma_0 )$ lies in the closed positive cone in $\mf t_\R = \R \otimes_\Z X_* (T)$. 
Alternatively, we can maneuver $\Re (-\sigma_0 )$ to the closed positive cone. Therefore we 
can arrange that we are in one of the situations where \cite[Lemma B.3]{SolKL} applies, with
$Q = Z_G (\sigma_0)$. It says that $\epsilon (\sigma,r) \neq 0$, which is needed to use 
\cite[Theorem B.2]{SolKL}. That result tells us
\[
E_{y,\sigma,r,\rho} = \ind_{\mh H (Q,M,q\cE)}^{\mh H (G,M,q\cE)} \, E^Q_{y,\sigma,r,\rho} .
\]
This remains valid upon applying $\sgn^*$ on both sides, by the isomorphism
\[
\begin{array}{ccc}
\ind_{\mh H (Q,M,q\cE)}^{\mh H (G,M,q\cE)} (\sgn^* E^Q_{y,\sigma,r,\rho}) & \to &
\sgn^* \big( \ind_{\mh H (Q,M,q\cE)}^{\mh H (G,M,q\cE)} \, E^Q_{y,\sigma,r,\rho} \big) \\
h \otimes v & \mapsto & \sgn (h) \otimes v 
\end{array}.
\]
Now part (a) shows it suffices to prove that 
\begin{equation}\label{eq:2.14}
\dim \Hom_{\C [W^Q_{q\cE}]} \big( \sgn^* (E^Q_{y,\sigma,r,\rho}), \det \big) \leq 1 .
\end{equation}
Notice that $\sigma_0 \in Z (\mf q)$, so that $(y,\sigma,r)$ is an essentially bounded parameter
for $\mh H (Q,M,q\cE)$.
Suppose for the moment that $r \neq 0$. Then Lemma \ref{lem:1.9}.b says that $(y,\sigma,r)$ 
is an open parameter, and by Lemma \ref{lem:1.7} $E^Q_{y,\sigma,r,\rho}$ is irreducible. 
In this case part (b) proves \eqref{eq:2.14}.

Recall from \cite[Lemma 3.6]{AMS2} that $\pi_0 (Z_G (\sigma_r,y)) = \pi_0 ( Z_G (\sigma_0,y))$ 
does not depend on $r$, where $\sigma_r = \sigma_0 + \textup{d}\gamma_y (\matje{r}{0}{0}{-r})$. 
Hence there is a family of $\mh H (Q,M,q\cE)$-modules $\sgn^* E_{y,\sigma_r,r,\rho}$, parametrized 
by $r \in \C$. It follows from \cite[Theorem 3.2.b]{SolKL} that the underlying family of 
$\C [W_{q\cE}^Q]$-modules is constant. We already showed that for $r \neq 0$ it contains 
det at most one time, so the same holds when $r = 0$.
\end{proof}

It is known, for unipotent representations of (adjoint) $p$-adic groups from \cite{Ree}
and for principal series representations of quasi-split $p$-adic groups from \cite{SolQS},
that the Langlands parameters of generic representations are precisely the open parameters,
with the trivial representation of a component group as enhancement. We intend to prove
an analogous statement for extended graded Hecke algebras.

\begin{lem}\label{lem:2.3}
Fix $(\sigma,r) = (\sigma_0 + r \sigma_v,r) \in \mf t \oplus \C (\sigma_v,1)$. 
\enuma{
\item Every irreducible $\mh H (G,M,q\cE)$-module with central character $(W_{q\cE}\sigma_0 ,r)$
is a subquotient of $\ind_{\mc O (\mf t \oplus \C)}^{\mh H (G,M,q\cE)} \C_{\sigma_0 ,r}$.
\item Up to $Z_G (\sigma)$-conjugacy, there exists precisely one enhanced L-parameter
$(y,\sigma,r,\rho)$ for $\mh H(G,M,q\cE)$ such that $\sgn^* (M_{y,\sigma,r,\rho})$ is generic.
}
\end{lem}
\begin{proof}
(a) Any irreducible $\mh H (G,M,q\cE)$-module $V$ with central character $(W_{q\cE}\sigma_0 ,r)$
has an $\mc O (\mf t \oplus \C)$-weight $(\sigma'_0 ,r)$ with $\sigma' \in W_{q\cE}\sigma_0$. Then
\[
\Hom_{\mh H (G,M,q\cE)} \big( \ind_{\mc O (\mf t \oplus \C)}^{\mh H (G,M,q\cE)} \C_{\sigma'_0,r}  
,V \big) \cong \Hom_{\mc O (\mf t \oplus \C)} (\C_{\sigma'_0,r},V) \neq 0 ,
\] 
so $V$ is a quotient of $\ind_{\mc O (\mf t \oplus \C)}^{\mh H (G,M,q\cE)} \C_{\sigma'_0 ,r}$.
On the other hand, $\ind_{\mc O (\mf t \oplus \C)}^{\mh H (G,M,q\cE)} \C_{\sigma'_0 ,r}$ and
$\ind_{\mc O (\mf t \oplus \C)}^{\mh H (G,M,q\cE)} \C_{\sigma_0 ,r}$ have the same irreducible
subquotients, with the same multiplicities \cite[Lemma 9.1.a]{SolGHA}.\\
(b) The group $W_{q\cE}^M$ for $\mh H (M,M,q\cE) \cong \mc O (\mf t \oplus \C)$ is trivial, so
\[
\Hom_{\C [W_{q\cE}^M]} (\C_{\sigma_0 ,r}, \det) \cong \C .
\]
By Proposition \ref{prop:2.2}.a also
\begin{equation}\label{eq:2.2}
\Hom_{\C [W_{q\cE}]} \big( \ind_{\mc O (\mf t \oplus \C)}^{\mh H (G,M,q\cE)} \C_{\sigma_0 ,r},
\det \big) \cong \C .
\end{equation}
By \cite[Theorem 4.6]{AMS2} every irreducible $\mh H (G,M,q\cE)$-module with central character
$(W_{q\cE}\sigma_0 ,-r)$ is of the form $\sgn^* (M_{y,\sigma,r,\rho})$. By part (a) and
\eqref{eq:2.2}, exactly one of these modules is generic. That corresponds to a unique
$G$-conjugacy class of $(y,\sigma,\rho)$, and since $(\sigma,r)$ was fixed we find that
$(y,\rho)$ is unique up to $Z_G (\sigma)$.
\end{proof}

The $\C[W_{q\cE}]$-module structure of $M_{y,\sigma,r,\rho}$ can be studied more easily in
the case $r = 0$, so we consider that first. 

\begin{prop}\label{prop:2.4}
Let $(y,\sigma_0,0)$ be an L-parameter which is open with respect to $(G,M,q\cE)$.
\enuma{
\item  There exists a unique enhancement $\rho_0 \in \Irr \big( \pi_0 
(Z_G (\sigma_0,y)) \big)$ such that\\ $\sgn^* (M_{y,\sigma_0,0,\rho_0})$ is a generic
$\mh H (G,M,q\cE)$-module.
\item Suppose that $v = 0$ and that $q\cE$ is the trivial equivariant local system on 
$\cC_v^M = \{0\}$. Then $\rho_0 =$ triv in part (a), $(y,\sigma_0,0)$ is open and
$y$ is regular in $Z_{\mf g}(\sigma_0)$.
}
\end{prop}
\begin{proof}
(a) The condition to be checked is equivalent to:
$\Res^{\mh H (G,M,q\cE)}_{\C [W_{q\cE}]} M_{y,\sigma_0,0,\rho}$ contains $\mr{triv}_{W_{q\cE}}$.
First we consider the analogous question for $\mh H (G^\circ,M^\circ,\cE)$. To avoid confusion,
we endow modules for this algebra with a superscript $\circ$. Write $Q = Z_G (\sigma_0)$ and 
$\mf q = \mr{Lie}(Q) = Z_{\mf g}(\sigma_0)$. By \cite[(34)]{AMS2}, for 
$\rho^\circ \in \Irr (Z_{G^\circ}(\sigma_0,y))$:
\begin{equation}\label{eq:2.4}
M^\circ_{y,\sigma_0,0,\rho^\circ} = \ind_{\mh H (Q^\circ,M^\circ,\cE)}^{\mh H (G^\circ,M^\circ,\cE)}
M^{Q^\circ}_{y,\sigma_0,0,\rho^\circ} .
\end{equation}
The $\C[W_\cE]$-module structure of \eqref{eq:2.4} follows from \cite[(33)]{AMS2}:
\begin{equation}\label{eq:2.5}
\Res^{\mh H (G^\circ,M^\circ,\cE)}_{\C [W_\cE]} M^\circ_{y,\sigma_0,0,\rho^\circ} =
\ind_{\C [W_\cE^{Q^\circ}]}^{\C [W_\cE]} M_{y,\rho^\circ} ,
\end{equation}
where $M_{y,\rho^\circ}$ comes from the generalized Springer correspondence \cite{Lus-Int} for \\
$(Q^\circ,M^\circ,\cE)$. By Frobenius reciprocity \eqref{eq:2.5} contains $\mr{triv}_{W_\cE}$
if and only if $M_{y,\rho^\circ}$ contains $\mr{triv}_{W_\cE^{Q^\circ}}$. 
By \cite[Theorem 9.2]{Lus-Int}, the latter happens if and only if 
\[
\mr{Ad}(Q^\circ) y \cap (\cC_v^M \oplus \mf u \cap \mf q) \quad \text{is dense in} \quad
\cC_v^M \oplus (\mf u \cap \mf q) .
\]
This holds in our setting, by Lemma \ref{lem:3.A}. From \cite[Theorem 9.2]{Lus-Int} 
or Lemma \ref{lem:2.3} we obtain that $\rho^\circ$ is unique. 

Consider a $\rho \in \Irr \big( \pi_0 (Z_Q (y)) \big)$ whose restriction to 
$\pi_0 (Z_{Q^\circ}(y))$ contains $\rho^\circ$. In the notation from \cite[Lemma 3.18]{AMS2} we 
have $\rho = \rho^\circ \rtimes \tau^\vee$, where $\tau^\vee$ is an irreducible 
representation of the stabilizer $S_{\rho^\circ}$ of $\rho^\circ$ in $Z_Q (y) / Z_{Q^\circ}(y)$. 
Then \cite[(67)]{AMS2} says that
\begin{equation}\label{eq:2.6}
M_{y,\sigma_0,0,\rho} \cong \tau \ltimes M^\circ_{y,\sigma_0,0,\rho^\circ} , 
\end{equation}
where the latter module is described explicitly in \cite[Lemma 3.16]{AMS2}. This description 
simplifies a bit in our setup, because the 2-cocycles in \cite[\S 3]{AMS2} are by assumption 
trivial. Namely, the structure of \eqref{eq:2.6} as $\C[W_{q\cE}]$-module is 
\begin{equation}\label{eq:2.7}
\ind^{\C [W_{q\cE}]}_{\C [W'_{q\cE}]} (\tau \otimes J)
\end{equation}
for some extension $J$ of $\mr{triv}_{W_\cE}$ to $W'_{q\cE}$. Here 
$W_\cE \subset W'_{q\cE} \subset W_{q\cE}$ such that $W'_{q\cE} / W_\cE$ is naturally isomorphic to 
$S_{\rho^\circ}$. This $J$ comes from \cite[Proposition 3.15]{AMS2}, and it is only unique up 
characters. As the underlying vector space of $J$ is that of $\mr{triv}_{W_\cE}$, we may renormalize
the operators $J(\gamma)$ with $\gamma \in W'_{q\cE}$, and arrange that $J =$ triv. Now it is clear
that, if we take $\tau = \mr{triv}_{S_{\rho^\circ}}$, then \eqref{eq:2.7} contains 
$\mr{triv}_{W_{q\cE}}$. Thus $\rho_0 := \rho^\circ \rtimes \mr{triv}$ fulfills the requirements. 
By Lemma \ref{lem:2.3}.a it is unique.\\
(b) Under these assumptions, the generalized Springer correspondence for $(Q^\circ,M^\circ,\cE)$,
as encountered in the proof of part (a), becomes the classical Springer correspondence for
$Q^\circ$. Then \eqref{eq:2.5} contains $\mr{triv}_{W_\cE}$ if and only $y$ is regular nilpotent 
in $\mf q$ and $\rho^\circ = \mr{triv}$. (Since these constructions are in the end based on
\cite{Lus-Int}, we have to use the normalization of the Springer correspondence from there.) 
Then $\rho_0$ in part (a) reduces to $\mr{triv}_{\pi_0 (Z_{Q^\circ}(y))} \rtimes \mr{triv} =
\mr{triv}_{\pi_0 (Z_Q (y))}$. As $y$ is regular nilpotent, $(y,\sigma_0,0)$ is open.
\end{proof}

We are ready to complete the analysis of the L-parameters of generic irreducible 
$\mh H (G,M,q\cE)$-modules.

\begin{thm}\label{thm:2.5}
Consider an enhanced L-parameter $(y,\sigma,r)$ for $\mh H (G,M,q\cE)$.
\enuma{
\item If $(y,\sigma,r)$ is open with respect to $(G,M,q\cE)$, then 
$\sgn^* (M_{y,\sigma,r,\rho})$ is generic for a unique enhancement $\rho$, say $\rho_g$.
\item If, in the setting of part (a), $q\cE$ is the trivial equivariant local system 
on $\cC_v^M = \{0\}$, then $\rho_g = \mr{triv}$ and $(y,\sigma,r)$ is open.
\item If $(y,\sigma,r)$ is not open with respect to 
$(G,M,q\cE)$, then $\sgn^* (M_{y,\sigma,r,\rho})$ is not generic, for any $\rho$.
}
\end{thm}
\begin{proof}
(a) and (b) Recall from \eqref{eq:1.22} that 
\[
\sigma = \sigma_r = \sigma_0 + \textup{d}\gamma_y ( \matje{r}{0}{0}{-r}) ,
\]
and knowing $(y,\sigma,r)$ up to $G$-conjugacy is the same as knowing
$(y,\sigma_0,r)$ up to $G$-conjugacy. By \cite[Lemma 3.6]{AMS2} there is a natural
isomorphism $\pi_0 (C_y) \cong \pi_0 (Z_G (\sigma_0,y))$. For any $\rho \in 
\pi_0 (Z_G (\sigma_0,y))$ we can vary $r$ in $\C$ and obtain an algebraic family of 
$\mh H (G,M,q\cE)$-modules $E_{y,\sigma_r,r,\rho}$. 
By \cite[Theorem 3.2.b]{SolKL}, which is based on \cite{Lus-Cusp1}, the underlying 
family of $\C [W_{q\cE}]$-modules is constant. By Lemma \ref{lem:1.7} we have
$E_{y,\sigma_r,r,\rho} = M_{y,\sigma_r,r,\rho}$. Now Proposition \ref{prop:2.4} 
proves the claims.\\
(c) By Lemma \ref{lem:2.3}.b there are unique $(y',\rho')$ such that
$\sgn^* M_{y',\sigma,r,\rho'}$ is generic. By part (a) $(y',\sigma,r)$ is open
with respect to $(G,M,q\cE)$. Hence $\mc O_{y'} \neq \mc O_y$ and 
$\sgn^* M_{y,\sigma,r,\rho}$ is not generic.
\end{proof}

We conclude this section with a proof of the generalized injectivity conjecture for 
geometric graded Hecke algebras.

\begin{cor}\label{cor:2.7}
Assume that $q\cE$ is a trivial equivariant local system on $\cC_v^M = \{0\}$.
Let $E$ be an analytic standard $\mh H (G,M,q\cE)$-module and let $M$ be a generic
irreducible subquotient of $E$. Then $M$ is a submodule of $E$.
\end{cor}
\begin{proof}
Write $E = \sgn^* E_{y,\sigma,r,\rho}$. Since $M$ has the same central character as $E$,
it equals $\sgn^* M_{y',\sigma,r,\rho'}$ for some $y',\rho'$. By Theorem \ref{thm:2.5}.b 
$(y',\sigma,r)$ is open and $\rho'$ is trivial. By Theorem \ref{thm:1.8} 
$M_{y',\sigma,r,\rho'}$ is isomorphic to a submodule of $E_{y,\sigma,r,\rho}$, and 
that remains true if we apply $\sgn^*$ to both. We know from Proposition \ref{prop:2.2}.c
that $\sgn^* E_{y,\sigma,r,\rho}$ has at most one generic irreducible subquotient,
so $M$ must be a submodule of $E$.
\end{proof}

\section{Transfer to affine Hecke algebras}
\label{sec:AHA}

We will show how the representation theoretic results from the previous sections can be translated
to suitable affine Hecke algebras. This section is largely based on \cite{Lus-Gr,SolAHA,AMS3}.

Let $\mc R = (X,R,Y,R^\vee,\Delta)$ be a based root datum and let $W$ be the Weyl group of $R$.
Let $\lambda, \lambda^* : R \to \Z_{\geq 0}$ be $W$-invariant functions such that
$\lambda^* (\alpha) = \lambda (\alpha)$ whenever $\alpha^\vee \notin 2 Y$. Let $\mb q$ be an 
invertible indeterminate. To these data one can associate an affine Hecke algebra
\[
\mc H = \mc H (\mc R,\lambda,\lambda^*, \mb q),
\] 
as for instance in \cite{Lus-Gr,SolHecke}. The underlying vector space is $\C [X] \otimes \C[W] 
\otimes \C [\mb q, \mb q^{-1}]$ and the quadratic relation for a simple reflection $s_\alpha$ is
\[
(T_{s_\alpha} + 1) (T_{s_\alpha} - \mb q^{2 \lambda (\alpha)}) = 0 .
\]
Let $\Gamma$ be a finite group acting on $\mc R$ and on $T = \Hom (X,\C^\times)$. Assume 
that $\lambda$ and $\lambda^*$ are $\Gamma$-invariant and that $\alpha (\gamma (1_T)) = 1$ for
all $\gamma \in \Gamma, \alpha \in R$. For any 2-cocycle $\natural : \Gamma^2
\to \C^\times$ we can build the twisted affine Hecke algebra
\[
\mc H \rtimes \C[ \Gamma,\natural] = 
\mc H (\mc R,\lambda,\lambda^*, \mb q) \rtimes \C[ \Gamma,\natural] ,
\]
see \cite[Proposition 2.2]{AMS3}.
As a vector space it is the tensor product of its subalgebras $\mc H$ and $\C [\Gamma,\natural]$,
and for a standard basis element $T_\gamma$ of $\C [\Gamma,\natural]$ we have the cross relations
\[
T_\gamma T_w \theta_x T_\gamma^{-1} = T_{\gamma w \gamma^{-1}} \gamma (\theta_x)
\qquad w \in W, x \in X.
\]
We can specialize $\mb q$ to any $q \in \C^\times$, and then we obtain Hecke algebras denoted
\[
\mc H (\mc R,\lambda,\lambda^*, q) \quad \text{and} \quad 
\mc H (\mc R,\lambda,\lambda^*, q) \rtimes \C[ \Gamma,\natural] .
\]
In practice we will only specialize to $q \in \R_{>0}$.

When $\natural$ is trivial, \cite[\S 6 and (8.9)]{OpSo} provide a good notion of genericity
for $\mc H \rtimes \Gamma$-modules, as follows. The elements $T_{w\gamma}$ with $w \in W$ and
$\gamma \in \Gamma$ form a $\C[\mb q,\mb q^{-1}]$-basis of a subalgebra 
$\mc H (W,q^\lambda) \rtimes \Gamma$. The Steinberg representation of 
$\mc H (W,{\mb q}^\lambda) \rtimes \Gamma$ (with $\mb q$ specialized to some chosen $q \in 
\C^\times$) has dimension one and is defined by 
\begin{equation}\label{eq:3.3}
\mr{St}(T_{w\gamma}) = \det (w\gamma) .
\end{equation} 
Here det denotes the determinant of the action of $W\Gamma$ on $X$. We say that a 
$\mc H \rtimes \Gamma$-module $V$ is generic if $\mb q$ acts as multiplication by some 
$q \in \C^\times$ and $\Res^{\mc H \rtimes \Gamma}_{\mc H (W,{\mb q}^\lambda) \rtimes \Gamma} V$ 
contains St.

The centre of $\mc H \rtimes \C [\Gamma,\natural]$ contains 
\begin{equation}\label{eq:3.1}
\mc O (T)^{W\Gamma} \otimes \C[\mb q,\mb q^{-1}] \cong \mc O (T / W\Gamma \times \C^\times ).
\end{equation}
Often we will analyse representations of $\mc H \rtimes \C [\Gamma,\natural]$ via localization to
suitable subsets of $T / W\Gamma \times \C^\times$. That involves decomposing representations along
their weights for \eqref{eq:3.1}, which works well for finite length representations but does not
always apply to infinite dimensional representations. Therefore we will usually restrict our
attention to the category $\Mod_{\mr{fl}} (\mc H \rtimes \C [\Gamma,\natural])$ of finite length (or
equivalently finite dimensional) modules.

There is a two-step reduction procedure which assigns to $\mc H \rtimes \C[ \Gamma,\natural]$ 
a twisted graded Hecke algebra that governs a well-defined part of its representation theory.
A suitable family of such twisted graded Hecke algebras covers the entire category
$\Mod_{\mr{fl}} (\mc H \rtimes \C [\Gamma,\natural])$.

We write $\mf t_\R = R \otimes_\Z X_* (T)$ and $T_\R = \exp (\mf t_\R)$.
We fix a unitary element $u \in \Hom (X,S^1) \subset T$, and we want to study representations
whose $\mc O (T)^{W\Gamma}$-weights are close to $W \Gamma u T_\R$ in $T / W \Gamma$. 
There is a subroot system $R_u = \{\alpha \in R : s_\alpha (u) = u\}$, with a basis $\Delta_u$
determined by $\Delta$. These fit into a based root datum 
\[
\mc R_u = (X,R_u,Y,R_u^\vee,\Delta_u) .
\]
The group $(W\Gamma)_u$ decomposes as 
\[
(W\Gamma)_u = W(R_u) \rtimes \Gamma_u ,\qquad 
\Gamma_u = \{ \gamma \in (W \Gamma)_u : \gamma (\Delta_u) = \Delta_u \} .
\]
Let $\lambda_u, \lambda_u^*$ be the restrictions of $\lambda,\lambda^*$ to $R_u$ and let $\natural_u$
be the restriction of $\natural : (W\Gamma)^2 \to \Gamma^2 \to \C^\times$ to $\Gamma_u^2$. Altogether
these objects yield a new twisted affine Hecke algebra 
\[
\mc H_u \rtimes \C [\Gamma_u,\natural_u] = 
\mc H (\mc R_u, \lambda_u,\lambda_u^*,\mb q) \rtimes \C [\Gamma_u,\natural_u] ,
\]
a subalgebra of $\mc H \rtimes \C [\Gamma,\natural]$. An advantage is that $u$ is fixed by $W(R_u)$,
so $\alpha (u) \in \{\pm 1\}$ for all $\alpha \in R_u$. There is a $(W\Gamma)_u$-equivariant map
\[
\exp_u : \mf t \to T , \exp_u (\sigma) = u \exp (\sigma) .
\]
It is a local diffeomorphism around $\mf t_\R$ and restricts to a diffeomorphism 
$\mf t_\R \to u T_\R$.
Via this map we can pass from $\mc H_u \rtimes \C[\Gamma_u,\natural_u]$ to a twisted graded Hecke
\[
\mh H_u \rtimes \C[\Gamma_u,\natural_u] = \mh H (\mf t,(W \Gamma)_u,k_u,\mb r,\natural_u) .
\]
Here the parameter function $k_u : R_u \to \Z$ is given by
\begin{equation}\label{eq:3.2}
k_{u,\alpha} = (\lambda (\alpha) + \alpha (u) \lambda^* (\alpha) ) / 2.
\end{equation}
The next theorem was proven in \cite[Theorems 2.5, 2.11 and Proposition 2.7]{AMS3}, based on
similar results in \cite[\S 8--9]{Lus-Gr} and \cite[\S 2.1]{SolAHA}. The part about genericity
was checked in \cite[Theorems 6.1 and 6.2]{SolQS}.

\begin{thm}\label{thm:3.1}
The following three categories are canonically equivalent:
\begin{itemize}
\item finite dimensional $\mc H \rtimes \C [\Gamma,\natural]$-modules, all whose $\mc O (T /
W \Gamma \times \C^\times)$-weights belong to $W\Gamma u T_\R \times \R_{>0}$,
\item finite dimensional $\mc H_u \rtimes \C [\Gamma_u,\natural_u]$-modules, all whose 
$\mc O (T / (W \Gamma_u \times \C^\times)$-weights belong to $u T_\R \times \R_{>0}$,
\item finite dimensional $\mh H_u \rtimes \C [\Gamma_u,\natural_u]$-modules, all whose 
$\mc O (\mf  t / (W \Gamma)_u \times \C)$-weights belong to $\mf t_\R \times \R$.
\end{itemize}
The equivalences have the following features:
\begin{enumerate}[(i)]
\item They are compatible with parabolic induction and parabolic restriction.
\item They respect temperedness.
\item They respect essentially discrete series when $\mr{rk}(R_u) = \mr{rk}(R)$, and otherwise
the involved category of $\mc H \rtimes \C [\Gamma,\natural]$-modules does not contain essentially 
discrete series representations.
\item They respect genericity whenever $\natural$ is trivial.
\item Any $\mc O (\mf t \oplus \C)$-weight of a $\mh H_u \rtimes \C [\Gamma_u,\natural_u]$-module
is transformed into a\\ $\mc O (T \times \C^\times)$-weight $(\exp_u (\sigma), \exp (r))$ for
$\mc H_u \rtimes \C[\Gamma_u,\natural_u]$ and into a collection of $\mc O (T \rtimes \C^\times)
$-weights $(w \exp_u (\sigma),\exp (r))$ for $\mc H \rtimes \C [\Gamma,\natural]$, where $w$ 
runs through a certain set of representatives for $W \Gamma / (W\Gamma)_u$.
\end{enumerate}
\end{thm}

Langlands standard modules for twisted affine Hecke algebras can be defined like for twisted graded 
Hecke algebras, see \eqref{eq:1.25}. This provides satisfactory
collections of standard modules in each of the three categories in Theorem \ref{thm:3.1}.
In each case they are in bijection with the irreducible modules in that category, via taking
irreducible quotients of standard modules.

\begin{lem}\label{lem:3.2}
The equivalences of categories in Theorem \ref{thm:3.1} restrict to bijections between the three
sets of Langlands standard modules.
\end{lem}
\begin{proof}
Theorem \ref{thm:3.1} respects almost all the operations and properties involved in (Langlands)
standard modules, the only potential issue being the weights in part (v). Between $\mh H_u \rtimes
\C [\Gamma_u,\natural_u]$ and $\mc H_u \rtimes \C [\Gamma_u,\natural_u]$, Theorem \ref{thm:3.1}
induces a bijection on weights, so the equivalence of categories provides a bijection between
the respective sets of standard modules.

It may seem that Theorem \ref{thm:3.1} does not necessarily match those two sets with standard
modules for $\mc H \rtimes \C [\Gamma,\natural]$. The problem lies in part (v) at the level
of parabolic subalgebras (associated to a set of simple roots $P$), which entails that a positive 
character $t$ for $\mc H_u^P \rtimes \C [\gamma_{P,u},\natural_u]$ may be moved by $\Gamma_P$ 
even if it is fixed by $\Gamma_{P,t}$. In such a situation the essentially
tempered irreducible representation $\tau \otimes t$ of is sent by Theorem \ref{thm:3.1} to
\[
\ind_{\mh H^P \rtimes \C [\Gamma_{P,t},\natural]}^{\mc H^P \rtimes \C[\Gamma_P,\natural]} 
(\tau' \otimes t) ,
\]
where $\tau'$ is the image of $\tau$ via Theorem \ref{thm:3.1} for the appropriate subalgebras.
The standard $\mc H \rtimes \C[\Gamma,\natural]$-module associated to $(P,\tau',t)$ is 
\[
\ind_{\mh H^P \rtimes \C [\Gamma_{P,t},\natural]}^{\mc H \rtimes \C[\Gamma,\natural]} 
(\tau' \otimes t) .
\]
In Theorem \ref{thm:3.1} this is matched with the standard $\mc H_u \rtimes 
\C[\Gamma_u,\natural_u]$-module
\[
\ind_{\mh H_u^P \rtimes \C [\Gamma_{P,u},\natural_u]}^{\mc H_u \rtimes \C[\Gamma_u,\natural_u]} 
(\tau \otimes t) .
\]
Thus Theorem \ref{thm:3.1} sends standard modules for $\mc H_u \rtimes \C[\Gamma_u,\natural_u]$
to standard modules for $\mc H \rtimes \C[\Gamma,\natural]$. Since we have an equivalence of
categories and on both sides the standard modules are canonically in bijection with the
irreducible modules, the equivalence is also bijective on standard modules.
\end{proof}

There are always classifications of irreducible $\mc H \rtimes \C [\Gamma,\natural]$-modules, see
\cite{SolHecke}, but in general these do not involve parameters like Langlands parameters for
reductive $p$-adic groups. To get the geometry from Sections \ref{sec:setup}--\ref{sec:generic}
into play, we need fairly specific parameter functions $\lambda,\lambda^*$, and the 2-cocycle
$\natural$ cannot be arbitrary either. Some twisted affine Hecke algebras that can be analysed
geometrically feature in \cite[\S 2]{AMS3}, they are based on reductive complex groups and cuspidal
local systems like our graded Hecke algebras. 

But the class of twisted affine Hecke algebras to which Sections 
\ref{sec:setup}--\ref{sec:open} can be applied is larger, we only need that for every 
fixed $u$ Theorem \ref{thm:3.1} yields a geometric graded Hecke algebra. From now on
we are mainly interested in genericity, and therefore we restrict to trivial $\natural$.
To transfer the results about submodules of standard modules, we impose the
following conditions.

\begin{cond}\label{cond:3.3}
The twisted affine Hecke algebra $\mc H (\mc R,\lambda,\lambda^*,\mb q) \rtimes \Gamma$ 
is such that, for each twisted graded Hecke algebra 
\[
\mh H (\mf t,(W\Gamma)_u ,k_u, \mb r ) = 
\mh H_u \rtimes \Gamma_u
\]
involved in Theorem \ref{thm:3.1} for some unitary element $u \in \Hom (X,S^1) \subset T$,
there are data $(G_u,M_u,q\cE_u,\cE_u)$ like in Section \ref{sec:setup} with $\cE_u$ trivial,  
and a Lie group isomorphism $\mf t \rtimes W_{q\cE_u} \cong \mr{Lie}(T_u) \rtimes (W\Gamma)_u$ 
which induces an algebra isomorphism $\mh H_u \cong \mh H (G_u^\circ,M_u^\circ,\cE_u)$.
\end{cond}

In Condition \ref{cond:3.3} we could take $G_u$ of the form $G_u^\circ \rtimes \Gamma_u$
where the action of $\Gamma_u$ on $G_u^\circ$ preserves a pinning.

\begin{thm}\label{thm:3.4}
Let $\mc H \rtimes \Gamma$ be an extended affine Hecke algebra satisfying 
Condition \ref{cond:3.3}. Let $E$ be a Langlands standard $\mc H \rtimes \Gamma
$-module on which $\mb q$ acts as multiplication by $q \in \R_{>1}$ and let $V$ be an
irreducible generic subquotient of $E$. Then $V$ is isomorphic to a submodule of $E$. 
\end{thm}
\begin{proof}
Since $E$ is standard, it admits a central character, say $(W\Gamma t,q)$. Put $u = t |t|^{-1} 
\in \Hom (X,S^1)$, so that $t \in u T_\R$. By Theorem \ref{thm:3.1}, the category 
$\Modf{W\Gamma t,q}(\mc H)$ is equivalent with the category 
\[
\Modf{(W\Gamma)_u |t|, \log q} (\mh H_u \rtimes \Gamma_u), 
\]
where $\mh H_u = \mh H (G_u,M_u,q\cE_u)$ by Condition \ref{cond:3.3}. By Lemma \ref{lem:3.2}
$E$ corresponds to a Langlands standard module $E_u$ of $\mh H_u \rtimes \Gamma_u$,
on which $\mb r$ acts as $\log q \in \R$. By \eqref{eq:1.26}, $E_u$ is a direct summand of 
$\ind_{\mh H_u}^{\mh H_u \rtimes \Gamma_u} E_u^\circ$ for some Langlands standard
$\mh H_u$-module $E_u^\circ$. More concretely, the steps from \eqref{eq:1.25} to \eqref{eq:1.26} 
show that
\begin{equation}\label{eq:3.4}
E_u = \ind_{\mh H_u \rtimes \Gamma'_u}^{\mh H_u \rtimes \Gamma_u} (\rho_u \otimes E_u^\circ),
\end{equation}
where $\Gamma'_u$ is the stabilizer of $E_u^\circ$ in $\Gamma_u$ and $\rho_u$ is an irreducible
representation of $\Gamma'_u$. By Proposition \ref{prop:1.13} $E_u^\circ$ is analytic standard.

Via Theorem \ref{thm:3.1}, $V$ corresponds to an irreducible 
subquotient $V_u$ of $E_u$. By Clifford theory, see for instance \cite[Appendix]{RaRa}, 
\cite[\S 11]{SolGHA} and \cite[\S 1]{AMS3},
$\Res^{\mh H_u \rtimes \Gamma_u}_{\mh H_u} V_u$ is completely reducible, and all
its irreducible summands are in one $\Gamma_u$-orbit. Via a composition series of $\rho_u \otimes 
E_u^\circ$ we see that $V_u$ arises from a subquotient of that, unique up to $\Gamma_u$.
It follows that $V_u$ contains an irreducible subquotient of $E_u^\circ$, say $V_u^\circ$, 
and is generated by $V_u^\circ$ as $\C [\Gamma_u]$-module. 

Clifford theory tells us that $V_u$ is a direct summand of $\ind_{\mh H_u}^{\mh H_u \rtimes 
\Gamma_u} V_u^\circ$. As $V_u$ is generic, \cite[(8.13)]{OpSo} says that
$V_u = \det \ltimes V_u^\circ$ and $V_u^\circ$ is generic. Thus 
$V_u^\circ$ has the same property as supposed for $V_u$. Now Corollary \ref{cor:2.7} proves 
the theorem for the subquotient $V_u^\circ$ of $E_u^\circ$.

Let soc$(E_u)$ denote the socle of $E_u$, that is, the sum of all irreducible submodules. Since
$\ind_{\mh H_u}^{\mh H_u \rtimes \Gamma_u}$ preserves completely reducibility
\cite[Theorem 11.2]{SolGHA} and $E_u$ is a direct summand of $\ind_{\mh H_u}^{\mh H_u \rtimes 
\Gamma_u} E_u^\circ$:
\[
\mr{soc}(E_u) = \C [\Gamma_u] \cdot \mr{soc}(E_u^\circ) .
\]
We already saw that $V_u^\circ$ is an irreducible submodule of $E_u^\circ$ which generates 
$V_u$, so $V_u \subset \C [\Gamma_u] \cdot \mr{soc}(E_u^\circ)$. Thus $V_u \subset 
\mr{soc}(E_u)$, which means that it is isomorphic to a submodule. We can go back to 
$\mc H \rtimes \Gamma$-modules via Theorem \ref{thm:3.1}, from which we conclude that $V$ 
is isomorphic to a submodule of $E$. 
\end{proof}

\section{Transfer to reductive $p$-adic groups}
\label{sec:padic}

Let $F$ be a non-archimedean local field and let $\mc G$ a connected reductive $F$-group. 
We will call $\mc G (F)$ a reductive $p$-adic group, although char$(F) > 0$ is allowed. 
We warn that $\mc G$ is not related to $G$ from Section \ref{sec:setup}.

We are interested in smooth complex representations of $\mc G (F)$, which form a ca\-te\-gory
$\Rep (\mc G (F))$. Let $\Rep (\mc G (F))^{\mf s}$ be a Bernstein block in there, coming 
from a unitary supercuspidal representation $\omega$ of a Levi subgroup 
$\mc M (F) \subset \mc G (F)$.

It is well-known that in many cases $\Rep (\mc G (F))^{\mf s}$ is closely related to the module
category of a (twisted) affine Hecke algebra. At the same time, it is known from \cite{SolEnd}
that one can increase the generality of such comparison results by using graded instead of
affine Hecke algebras.

Let $X_\nr (\mc M (F))$ be the Lie group of unramified characters of $\mc M (F)$, and let 
$X_\nr^+ (\mc M (F))$ be the subgroup $\Hom (\mc M (F), \R_{>0})$. Recall that 
$\Rep (\mc G (F))^{\mf s}$ consists of all smooth $\mc G (F)$-representations $\pi$ such that 
every irreducible subquotient of $\pi$ has cuspidal support in 
$(\mc M (F), X_\nr (\mc M (F)) \omega)$
up to $\mc G (F)$-conjugacy. Let $W_{\mf s}$ be the finite group associated 
to $\mf s = [\mc M (F),\omega]$ by Bernstein, and let $W_{\mf s,\omega}$ be the subgroup that
stabilizes $\omega$. Let $\mf t$ be the Lie algebra of $X_\nr (\mc M (F))$, identified with the
tangent space to $X_\nr (\mc M (F)) \omega$ at $\omega$. Since $W_{\mf s}$ operates faithfully on 
$X_\nr (\mc M (F)) \omega$ \cite[\S 2.16]{BeDe}, $W_{\mf s,\omega}$ acts faithfully on $\mf t$.

We define a root system $R_\omega$ as in \cite[\S 6.1]{SolEnd}, where it is called 
$\Sigma_{\sigma \otimes u}$. Parameters $k^\omega$ and a 2-cocycle $\natural_\omega$ of 
$\Gamma_\omega \cong W_{\mf s,\omega} / W(R_\omega)$ (denoted $\natural_u^{-1}$ in \cite{SolEnd}) 
are constructed in \cite[\S 7]{SolEnd}.

\begin{thm}\label{thm:4.1} \textup{\cite[Theorems B and C and (8.2)]{SolEnd}} \\
There exists an equivalence between the following categories:
\begin{itemize}
\item finite length smooth $\mc G (F)$-representations $\pi$, such that all irreducible
subquotients of $\pi$ have cuspidal support in $(\mc M (F),X_\nr^+ (\mc M (F)) \omega)$
up to $\mc G (F)$-conjugacy,
\item finite dimensional modules of the twisted graded Hecke algebra
\[
\mh H (\mf t,W_{\mf s,\omega},k^\omega,\natural_\omega) = 
\mh H (\mf t,W (R_\omega),k^\omega) \rtimes \C [\Gamma_\omega,\natural_\omega],
\]
all whose $\mc O (\mf t)$-weights belong to $\mf t_\R = \mr{Lie}\big(X_\nr^+ (\mc M (F)) \big)$.
\end{itemize}
This equivalence is canonical up to the choice of the 2-cocycle $\natural_\omega$, and it has
the following properties:
\begin{enumerate}[(i)]
\item compatibility with normalized parabolic induction and restriction,
\item respects temperedness,
\item sends essentially square-integrable $\mc G (F)$-representations to essentially discrete
series $\mh H (\mf t,W_{\mf s,\omega},k^\omega,\natural_\omega)$-modules 
(but not always conversely),
\item compatibility for twisting $\mc G (F)$-representations by elements of $X_\nr^+ (\mc M (F))$ 
and $\mh H (\mf t,W_{\mf s,\omega},k^\omega,\natural_\omega)$-modules by elements of
$\mr{Lie} \big( X_\nr^+ (\mc G (F)) \big) \subset \mf t_\R$.
\end{enumerate}
\end{thm}

Notice that there is no $\mb r$ in the graded Hecke algebras in Theorem \ref{thm:4.1}. 
They relate to Sections \ref{sec:setup}--\ref{sec:generic} by specializing $\mb r$ 
at some $r > 0$.

Let $\mc L(F) \subset \mc G (F)$ be a Levi subgroup containing $\mc M (F)$ and let $\tau \in
\Irr (\mc L (F))$ be a tempered representation with cuspidal support in $(\mc M (F),
X_\nr^+ (\mc M (F))\omega )$. Let $\chi \in X_\nr^+ (\mc L (F))$ be in positive position with
respect to a parabolic subgroup $\mc P (F) \subset \mc G (F)$ with Levi factor $\mc L (F)$. Then
$I_{\mc P (F)}^{\mc G (F)} (\tau \otimes \chi)$ is a standard representation as in the Langlands
classification for $\mc G (F)$. Moreover every standard representation $\mc G (F)$-representation 
with cuspidal support in $(\mc M (F),X_\nr^+ (\mc M (F))\omega)$ (up to $\mc G (F)$-conjugation) 
is of this form.

\begin{lem}\label{lem:4.2}
The equivalence in Theorem \ref{thm:4.1} restricts to a bijection between the sets of 
Langlands standard representations in both categories.
\end{lem}
\begin{proof}
Let $\mh H (\mf t,W_{\mf s,\omega}^{\mc L},k^\omega,\natural_\omega)$ be the parabolic subalgebra
of $\mh H (\mf t,W_{\mf s,\omega},k^\omega,\natural_\omega)$ determined by $\mc L (F)$. The
properties in Theorem \ref{thm:4.1} imply that $I_{\mc P (F)}^{\mc G (F)} (\tau \otimes \chi)$
is matched with 
\begin{equation}\label{eq:4.1}
\ind_{\mh H (\mf t,W_{\mf s,\omega}^{\mc L},k^\omega,\natural_\omega)}^{\mh H 
(\mf t,W_{\mf s,\omega},k^\omega,\natural_\omega)} \big( \tau_{\mh H} \otimes \log (\chi) \big) ,
\end{equation}
where $\tau_{\mh H}$ denotes the image of $\tau$ under Theorem \ref{thm:4.1} for $\mc L (F)$
and $\log (\chi)$ is fixed by $W_{\mf s,\omega}^{\mc L}$.
By assumption $\chi$ is positive with respect to all roots of $Z^\circ (\mc M) (F)$ in
$\mr{Lie}(\mc P (F) / \mc L (F))$. The root system $R_\omega$ consists of scalar multiples
of the roots of $Z^\circ (\mc M)(F)$ in $\mr{Lie}(\mc G (F))$, but some of those roots may be
left out depending on $\omega$. As a consequence the condition for a character of 
$\mh H (\mf t,W_{\mf s,\omega}^{\mc L},k^\omega,\natural_\omega)$ to be in positive position
may be weaker than the corresponding condition for $X_\nr (\mc L (F))$. Thus $\log (\chi)$
is in positive position (but one cannot conclude that in opposite direction). This shows that
\eqref{eq:4.1} is a standard module in the traditional sense, and since $W_{\mf s,\omega}^{\mc L}$
fixes $\log (\chi)$ it is also a Langlands standard module as in \eqref{eq:1.25}.

Thus the equivalence of categories in Theorem \ref{thm:4.1} sends standard representations to
Langlands standard modules. These two ``standard" sets are canonically in bijection with the
irreducible representations in the respective categories. Hence the equivalence of categories
is bijective on standard representations.
\end{proof}

When $\omega$ is simply generic \cite{BuHe}, one can improve on Theorem \ref{thm:4.1}. Let 
$\mc U$ be the unipotent radical of a minimal parabolic $F$-subgroup $\mc B$ of $\mc G$. For a 
nondegenerate character $\xi$ of $\mc U (F)$, the $\mc G (F)$-orbit of the pair $(\mc U (F),\xi)$ 
is called a Whittaker datum for $\mc G (F)$. By conjugating with a suitable element of $\mc G$, 
we may assume that $\mc M$ contains a Levi factor of $\mc B$. We recall that an 
$\mc M(F)$-representation $\pi$ is called simply generic if 
$\Hom_{\mc U (F) \cap \mc M (F)}(\pi,\xi)$ has dimension one. Although this depends on the
choice of the Whittaker datum for $\mc G (F)$, we suppress that in our terminology.

\begin{thm}\label{thm:4.3} \textup{\cite[Theorem E]{OpSo} and \cite[Theorem 10.9]{SolEnd}}\\
Assume that the supercuspidal unitary representation $\omega \in \Irr (\mc M (F))$ is simply
generic. There exists an extended affine Hecke algebra $\mc H_{\mf s} \rtimes \Gamma_{\mf s}$
whose module ca\-te\-gory is canonically equivalent with $\Rep (\mc G (F))^{\mf s}$. This 
$\mc H_{\mf s}$ is constructed from the following data:
\begin{itemize}
\item the complex torus $X_\nr (\mc M (F)) \omega \subset \Irr (\mc M(F))$,
\item a root system $R_{\mf s}$ such that $W(R_{\mf s}) \rtimes \Gamma_{\mf s} = W_{\mf s}$,
\item $q$-parameters in $\R_{\geq 1}$ as in \cite[(3.7)]{SolEnd} and $\lambda,\lambda^*$
as in \cite[(9.5)]{SolEnd}.
\end{itemize}
The equivalence of categories $\Rep (\mc G (F))^{\mf s} \cong 
\Mod (\mc H_{\mf s} \rtimes \Gamma_{\mf s})$:
\begin{enumerate}[(i)]
\item is compatible with normalized parabolic induction and restriction,
\item respects temperedness,
\item sends essentially square-integrable $\mc G (F)$-representations to essentially discrete
series $\mc H_{\mf s} \rtimes \Gamma_{\mf s}$-modules and conversely,
\item preserves genericity.
\item is compatible with twisting by unramified characters of $\mc G (F)$,
\end{enumerate}
\end{thm}

From Theorem \ref{thm:4.3} one can obtain Theorem \ref{thm:4.1} (when $\omega$ is simply generic)
by applying a variation on Theorem \ref{thm:3.1} to $\mc H_{\mf s} \rtimes \Gamma_{\mf s}$,
that is essentially what happens in \cite[\S 6--7]{SolEnd}. We need the version of Theorem 
\ref{thm:3.1} proven in \cite[\S 2.1]{SolAHA}, with $\mb q$ specialized to $q \in \R_{>1}$, $\mb r$ 
specialized to $r \in \R_{>0}$ and $\lambda,\lambda^*,k^u$ real-valued (but not necessarily
integral). From Theorem \ref{thm:4.3}.iv and Theorem \ref{thm:3.1}.iv we deduce:

\begin{cor}\label{cor:4.4}
If $\omega$ is simply generic, then $\natural_\omega = 1$ and the equivalence of categories in 
Theorem \ref{thm:4.1} preserves genericity.
\end{cor}

Like in Section \ref{sec:AHA}, to apply our results from Sections 
\ref{sec:setup}--\ref{sec:generic} we need the graded Hecke algebras in Theorem 
\ref{thm:4.1} to be of geometric type, a condition on the parameter functions 
$k^\omega$. Lusztig \cite{Lus-open} has conjectured that it is valid in general. 
To apply our results about genericity, we use the following:

\begin{cond}\label{cond:4.5}
Let $\mc G (F),\mc M (F),\omega$ and $\mh H (\mf t,W(R_\omega),k^\omega) 
\rtimes \Gamma_\omega]$ be as in Theorem \ref{thm:4.1}. There must exist data 
$(G_\omega,M_\omega, q \cE_\omega,\cE_\omega)$ as in Section \ref{sec:setup}
with $\cE_\omega$ trivial, $r \in \R_{>0}$, and a Lie group isomorphism 
$\mf t \rtimes W_{\mf s,\omega} \cong \mr{Lie}(T_\omega) \rtimes W_{q\cE_\omega}$, 
which induce an algebra isomorphism 
\[
\mh H (\mf t, W(R_\omega),k^\omega) \cong 
\mh H (G_\omega^\circ, M_\omega^\circ,\cE_\omega) / (\mb r - r) .
\]
\end{cond}

\begin{thm}\label{thm:4.6}
Assume that Condition \ref{cond:4.5} holds for a unitary supercuspidal re\-pre\-sentation 
$\omega$ of a Levi subgroup $\mc M (F) \subset \mc G (F)$. Let $\pi_{st}$ be a standard 
$\mc G (F)$-representation with cuspidal support in $(\mc M (F), X_\nr^+ (\mc M (F)) \omega)$ 
and let $\pi$ be an irreducible subquotient of $\pi_{st}$. Suppose that $\omega$ is simply 
generic and $\pi$ is generic. Then $\pi$ is a subrepresentation of $\pi_{st}$.
\end{thm}
\begin{proof}
This can be shown exactly like in Theorem \ref{thm:3.4}, using the results in Section 
\ref{sec:padic} instead of those in Section \ref{sec:AHA}.
\end{proof}

For quasi-split groups, we will improve on Theorem \ref{thm:4.6} by verifying Condition 
\ref{cond:4.5}. The next result was already known for principal series representations 
\cite[Lemma 6.4]{SolQS}, and anticipated in \cite[Appendix]{OpSo}.

\begin{thm}\label{thm:4.7}
Let $\omega$ be a generic unitary supercuspidal representation of a Levi subgroup 
$\mc M (F)$ of $\mc G (F)$. Assume that (a) or (b) holds:
\enuma{
\item $\mc G(F)$ is quasi-split,
\item $\omega$ is simply generic and the extended affine Hecke algebra
$\mc H_{\mf s} \rtimes \Gamma_{\mf s}$ from Theorem \ref{thm:4.3} has equal parameters.
}
Then the twisted graded Hecke algebra $\mh H (\mf t,W(R_\omega),k^\omega) \rtimes
\C [\Gamma_\omega,\natural_\omega]$ from Theorem \ref{thm:4.1} is isomorphic to an 
extended graded Hecke algebra with equal parameters.
\end{thm}
\begin{proof}
(a) As observed in and before Corollary \ref{cor:4.4}, $\natural_\omega = 1$ and 
$\mh H (\mf t,W(R_\omega),k^\omega) \rtimes \Gamma_\omega$ can be obtained from the
extended affine Hecke algebra $\mc H_{\mf s} \rtimes \Gamma_{\mf s}$ in Theorem 
\ref{thm:4.3} by applying the
reduction procedure from \cite[\S 2.1]{SolAHA} and  more generally 
\cite[Theorems 2.5 and 2.11]{AMS3}. The effect on the parameters is given by 
\eqref{eq:3.2}. In view of \cite[(95)]{SolEnd}, this works out to
\begin{equation}\label{eq:4.2}
k^\omega_\alpha = \log (q_\alpha) / \log (q_F) \quad \text{or} \quad 
k^\omega_\alpha = \log (q_{\alpha*}) / \log (q_F) ,
\end{equation}
with $q_\alpha$ and $q_{\alpha*}$ as in Harish-Chandra's $\mu$-function 
\cite[(3.7)]{SolEnd}.
Which of the options from \eqref{eq:4.2} depends on $\omega$. We must use $q_\alpha$ if 
$\alpha$ (as a function on $X_\nr (\mc M (F)) \omega$) takes the same value at $\omega$ 
and at the base point of $X_\nr (\mc M (F)) \omega$ chosen in \cite[\S 3]{SolEnd}, and 
we must use $q_{\alpha*}$ otherwise.

Thus $k^\omega$ agrees with the function $k^{\sigma'}$ (for $\sigma' \cong \omega$) from
\cite[Proposition A.2]{OpSo}, except that the domain of $k^{\sigma'}$ is obtained from 
the domain of $k^\omega$ by omitting the $\alpha$'s with $k^\omega_\alpha = 0$. Put 
\[
R_{\sigma'} = \{ \alpha \in R_\omega : k^\omega_\alpha \neq 0 \}
\]
and let $\Gamma_{\sigma'}$ be the stabilizer in $W(R_\omega)$ of the set of positive 
roots in $R_{\sigma'}$. As in \cite[Lemma 6.3]{SolQS}, one checks that 
\[
\mh H( \mf t, W(R_\omega),k^\omega) \rtimes \Gamma_\omega 
\cong \mh H( \mf t, W(R_{\sigma'}),k^{\sigma'}) \rtimes (\Gamma_{\sigma'} \rtimes
\Gamma_\omega) .
\]
The setup for $\mc H$ and $\mh H (\mf t,W(R_\omega),k^\omega)$ in \cite[\S 3]{SolEnd}
entails that the roots in our setting are the coroots in \cite[Appendix A]{OpSo}. 
More precisely, from \cite[(A.4) and (A.5)]{OpSo} one sees that $\alpha^\vee$ over 
there corresponds to $h_\alpha^\vee$ from \cite{SolEnd}, which is just 
$\alpha \in R_{\mf s}$ in Theorem \ref{thm:4.3}. Consider an 
irreducible component $R$ of $R_{\sigma'}$. Let $\alpha \in R$ be long and $\beta \in R$ 
be short. Then \cite[Proposition A.2]{OpSo} says that 
$\kappa_R := k^{\sigma'}_\alpha / k^{\sigma'}_\beta$
equals either 1 or the square of the ratio of the lengths of $\alpha$ and $\beta$ (which is
1, 2 or 3). If $\kappa_R \neq 1$, then we can divide all long roots in $R$ by $\kappa_R$,
and obtain a new root system $R'$ with the same Weyl group. As observed in 
\cite[Example 5.4]{SolHecke}, this gives rise to an algebra isomorphism
\[
\mh H (\mf t,W(R),k^{\sigma'}) \to \mh H (\mf t,W(R'),k') ,
\]
which is the identity on $\mc O (\mf t)$, such that $k'$ takes the value 
$k^{\sigma'}_\beta \in \R_{>0}$ on all roots in $R'$. Rescaling all the elements of 
$R'$ by a factor $2 / k^{\sigma'}_\beta$ (still not touching $\mf t$), we may further
assume that $k ' = 2$ on $R'$. We do this for all irreducible components $R$ 
of $R^{\sigma'}$, and we obtain an algebra isomorphism
\begin{equation}\label{eq:4.3}
\mh H (\mf t,W(R_{\sigma'}),k^{\sigma'}) \to \mh H (\mf t,W(R'_{\sigma'}),k')
\end{equation}
which is the identity on $\mc O (\mf t)$, and $k' = 2$ on $R'_{\sigma'}$. In 
\eqref{eq:4.3} each root is scaled by a factor that depends only on $k^{\sigma'}$. 
Since $k^{\sigma'} = k^\omega |_{R_{\sigma'}}$ is $\Gamma_{\sigma'} \rtimes 
\Gamma_\omega$-invariant, the isomorphism \eqref{eq:4.3} is 
$\Gamma_{\sigma'} \rtimes \Gamma_\omega$-equivariant. Hence it extends to an
algebra isomorphism
\begin{align*}
\mh H (\mf t, W(R_\omega),k^\omega) \rtimes \Gamma_\omega & \cong \mh H 
(\mf t,W(R_{\sigma'}),k^{\sigma'}) \rtimes (\Gamma_{\sigma'} \rtimes \Gamma_\omega) \\
& \to \mh H (\mf t,W(R'_{\sigma'}),k') \rtimes 
(\Gamma_{\sigma'} \rtimes \Gamma_\omega) . 
\end{align*}
(b) This is analogous to part (a). By the equal parameter assumption and \eqref{eq:3.2}
we have
\[
k^\omega_\alpha = (\lambda (\alpha) \pm \lambda^* (\alpha))/2 \in 
\{\lambda (\alpha), 0 \} \quad \text{for all } \alpha \in R_\omega .
\]
Hence $k^\omega$ is an equal parameter function on $\{ \alpha \in R_\omega :
k^\omega_\alpha \neq 0 \}$. The rest of the argument is the same as for part (a). 
\end{proof}

From Theorem \ref{thm:4.7} we deduce that Condition \ref{cond:4.5} is automatic.

\begin{lem}\label{lem:4.8}
In the setting of Theorem \ref{thm:4.7}, Condition \ref{cond:4.5} holds with 
$(M,\cC_v^M,q\cE) = (T,\{0\},\mr{triv})$.
\end{lem}
\begin{proof}
Let $G^\circ$ be a connected complex reductive group with a maximal torus $T$, so that
$R(G^\circ,T) \cong R_{\sigma'}$ and 
\[
\mf t \rtimes W(R_{\sigma'}) \cong \mr{Lie}(T) \rtimes W(G^\circ,T) .
\]
By passing to a cover, we may assume that the derived group of $G^\circ$ is simply connected.
The action of $\Gamma_{\sigma'} \rtimes \Gamma_\omega$ on $\mf t \rtimes W(R_{\sigma'})$ can
be transferred to $\mr{Lie}(T) \rtimes W(G^\circ,T)$, and then (using the simply connectedness
of $G^\circ_\der$) lifted to an action on $G^\circ$ that preserves a pinning. In this way we 
build the complex reductive group $G = G^\circ \rtimes (\Gamma_{\sigma'} \rtimes \Gamma_\omega)$.

Since $W_{\mf s,\omega}$ acts faithfully on $\mf t$, $Z_G (T) = Z_{G^\circ}(T) = T$. In particular
$T$ is a quasi-Levi subgroup of $G$, and it admits a quasi-cuspidal support $(T,\{0\},\mr{triv})$.
All the parameters $k_\alpha$ for $(T,\{0\},\mr{triv})$ are equal to 2 \cite[\S 0.3]{Lus-Cusp1}.
Moreover the 2-cocycle $\natural_{\mr{triv}}$ is trivial because $\Gamma_{\sigma'} \rtimes 
\Gamma_\omega$ acts naturally on all the relevant perverse sheaves constructed from 
$(T,\{0\},\mr{triv})$. We conclude that
\[
\mh H (\mf t, W(R_\omega),k^\omega) \rtimes \Gamma_\omega \cong
\mh H (\mf t,W(R_{\sigma'}),k^{\sigma'}) \rtimes (\Gamma_{\sigma'} \rtimes \Gamma_\omega) \cong
\mh H (G,T,\mr{triv}) / (\mb r - 1) . \qedhere
\]
\end{proof}

We are ready to prove the generalized injectivity conjecture from \cite{CaSh}.

\begin{thm}\label{thm:4.9}
Let $\pi_{st}$ be a standard representation of a quasi-split reductive $p$-adic group $\mc G (F)$.
Let $\pi$ be a generic irreducible subquotient of $\pi_{st}$. Then $\pi$ is a subrepresentation
of $\pi_{st}$.
\end{thm}
\begin{proof}
Let $(\mc M (F),\omega')$ be the cuspidal support of $\pi$ (and hence of $\pi_{st}$). 
The normalized parabolic induction of a nongeneric irreducible representation is not generic 
\cite{BuHe}, so $\omega'$ must be generic. More precisely, when we choose a representative 
$(\mc U (F),\xi)$ for the Whittaker datum so that $\mc M \cap \mc U$ is a maximal 
unipotent subgroup of $\mc M$, then $\omega'$ is $(\mc U (F) \cap \mc M (F),\xi)$-generic.
Let $|\chi|$ be the absolute value of the central character of $\omega'$. Then
$|\chi| \in X_\nr^+ (\mc M (F))$ and $\omega = \omega' \otimes |\chi|^{-1}$ is unitary.
Twisting by unramified characters does not perturb the genericity of $\omega'$, so we are in the
setting of Theorems \ref{thm:4.6} and \ref{thm:4.7}.a. Condition \ref{cond:4.5} holds by 
Lemma \ref{lem:4.8}, so Theorem \ref{thm:4.6} yields the desired statement.
\end{proof}

\section{Relations with the local Langlands correspondence}
\label{sec:LLC}

To prepare for the upcoming arguments, we compare properties of (enhanced) L-parameters
for reductive $p$-adic groups and for twisted graded Hecke algebras. 

Let $\mb W_F$ be the Weil group of $F$ and let ${}^L \mc G = \mc G^\vee \rtimes \mb W_F$ be the
Langlands dual group of $\mc G (F)$. Let $\mc M (F)$ be a standard Levi subgroup of $\mc G (F)$ 
and let $\mc M^\vee \subset \mc G^\vee$ be the associated standard Levi subgroup on the dual side. 
Let $(\phi_b, q\epsilon)$ be a bounded enhanced cuspidal L-parameter for $\mc M (F)$, in the form
\[
\phi_b : \mb W_F \times SL_2 (\C) \to {}^L \mc M
\]
and with $q\epsilon$ an irreducible representation of the appropriate component group $S_\phi$.
Consider the group
\[
G_{\phi_b} = Z^1_{\mc G^\vee_\Sc} (\phi_b |_{\mb W_F})
\]
from \cite[(75)]{AMS1}. We mainly need its identity component 
\[
G_{\phi_b}^\circ = Z_{\mc G^\vee_\Sc} (\phi_b (\mb W_F))^\circ 
\]
and its Lie algebra
\[
\mf g^\vee_{\phi_b} = \mr{Lie} (G_{\phi_b}) = \mr{Lie} (\mc G^\vee_\Sc)^{\phi_b (\mb W_F)} .
\]
Let $M_{\phi_b}$ be the quasi-Levi subgroup of $G_{\phi_b}$ determined by $(\mc M^\vee, \phi_b, 
q\epsilon)$ and let $q\cE$ be the cuspidal local system (determined by $q\epsilon$) on the nilpotent 
orbit in Lie$(M_{\phi_b})$ from $\phi_b |_{SL_2 (\C)}$.
Recall from \cite[\S 3.3.1]{Hai} that the group of unramified characters $X_\nr (G)$ is naturally
isomorphic to the complex torus 
\[
X_\nr ({}^L \mc G) := (Z(\mc G^\vee)^{\mb I_F})^\circ_{\Fr_F},
\]
where $\mb I_F$ denotes the inertia subgroup of $\mb W_F$. From these data we build
\[
\mh H (\phi_b,q\epsilon) = \mh H \big( G_{\phi_b} \times X_\nr ({}^L \mc G), 
M_{\phi_b} \times X_\nr ({}^L \mc G), q\cE \big) .
\]
It is almost the same as the algebra 
\[
\mh H ( \phi_b, v, q\epsilon, \vec{\mb r}) = \mh H \big( G_{\phi_b} \times X_\nr ({}^L \mc G), 
M_{\phi_b} \times X_\nr ({}^L \mc G), q\cE, \vec{\mb r} \big) 
\]
from \cite[(3.9)]{AMS3}. The only difference is that we set all the indeterminates $\mb r_i$ from
\cite{AMS3} equal to $\mb r$, by dividing out the ideal generated by the $\mb r_i - \mb r_j$ with
$i \neq j$.

The group $X_\nr (\mc M (F)) \cong X_\nr ({}^L \mc M)$ acts naturally on L-parameters for 
$\mc M (F)$, by adjusting the value $\phi (\Fr_F)$ while keeping 
$\phi|_{\mb I_F \times SL_2 (\C)}$ fixed. We refer to this as twisting Langlands parameters by
unramified characters. Let $X_\nr^+ ({}^L \mc M) \subset X_\nr ({}^L \mc M)$ be the subgroup
corresponding to $X_\nr^+ (\mc M (F)) \subset X_\nr (\mc M (F))$. As 
\[
Z(M)^\circ \times X_\nr (\mc G^\vee) \to X_\nr (\mc M^\vee)
\]
is a finite covering \cite[Lemma 3.7]{AMS3}, $X_\nr^+ ({}^L \mc M)$ can be identified with 
a subgroup of $Z(M)^\circ \times X_\nr^+ ({}^L \mc G)$. 

Recall that an L-parameter $(y,\sigma,r)$ for $\mh H (\phi_b,q\epsilon)$ contains 
(up to conjugation) precisely the same information as $(y,\sigma_0,r)$, for $\sigma_0$ as 
in \eqref{eq:1.22}. As worked out in \cite[\S 3.1]{AMS3}, the exponential map for 
$\mc G^\vee$ provides a natural bijection from 
\begin{itemize}
\item the set of enhanced L-parameters $(y,\sigma,r,\rho)$ or $(y,\sigma_0,r,\rho)$ for $\mh H
(\phi_b ,q\epsilon)$, with $r = -\log (q_F)/2$ and $\sigma_0 \in X_\nr^+ ({}^L \mc M)$ \qquad to
\item the set of enhanced L-parameters for $\mc G (F)$ with cuspidal support in\\
$(\mc M (F), X_\nr^+ ({}^L \mc M) \phi_b, q\epsilon)$.
\end{itemize}
By construction this bijection is compatible with the cuspidal support maps.
We note that in \cite[\S 3.1]{AMS3} the role of $\sigma_0$ is played by 
$\log (\phi (\Fr_F) \phi_b (\Fr_F)^{-1})$. 
Explicitly, $\exp (y,\sigma,-\log(q_F)/2,\rho) = (\phi ,\rho)$ where
\[
\phi |_{\mb I_F} = \phi_b |_{\mb I_F} ,\; \phi (\Fr_F) = \exp (\sigma_0) \phi_b (\Fr_F) ,\; 
\phi ( \matje{1}{1}{0}{1} ) = \exp (y) .
\]
If we present $\phi$ as a Weil--Deligne parameter $\phi'$, then 
\[
\phi' (\Fr_F) = \exp (\sigma) \phi_b (\Fr_F) 
\text{ and the nilpotent operator for } \phi' \text{ is } y.
\]
Recall the definition of an open Langlands parameter from \eqref{eq:0.open} and the definition 
of ``open with respect to cuspidal supports" from \eqref{eq:0.scopen}. We call a 
Langlands parameter $\psi : \mb W_F \times SL_2 (\C) \to {}^L \mc G$ for essentially bounded if 
$\psi (\Fr_F) = (a,\Fr_F)$ with $a Z(\mc G^\vee)$ in a bounded subgroup of $\mc G^\vee / 
Z(\mc G^\vee)$. By the Langlands classification for L-parameters \cite{SiZi}, every discrete
L-parameter $\psi$ is essentially bounded. More precisely, $\psi$ can be expressed (uniquely
up to conjugation) as a bounded discrete L-parameter twisted by an element of $X_\nr^+ ({}^L \mc G)$.

\begin{lem}\label{lem:7.3}
Let $(y,\sigma,r,\rho)$ be an enhanced L-parameter for $\mh H (\phi_b,q\epsilon)$, with 
$r = -\log (q_F)/2$ and $\sigma_0 \in X_\nr^+ ({}^L \mc M)$.
\enuma{
\item $(y,\sigma,r,\rho)$ is bounded if and only if $\exp (y,\sigma,r,\rho)$ is bounded.
The same holds with essentially bounded.
\item $(y,\sigma,r)$ is open if and only if $\exp (y,\sigma,r)$ is open.
\item $(y,\sigma,r,\rho)$ is open with respect to 
$\big( G_{\phi_b} \times X_\nr ({}^L \mc G), M_{\phi_b} \times X_\nr ({}^L \mc G), q\cE \big)$  
if and only if $\exp (y,\sigma,r,\rho)$ is open with respect to cuspidal supports.
}
\end{lem}
\begin{proof}
(a) Recall from \eqref{eq:1.22} that $\sigma = \sigma_0 + \textup{d}\gamma_y ( \matje{r}{0}{0}{-r})$.
The first condition is equivalent with $\sigma_0 = 0$ because $\sigma_0 \in X_\nr^+ ({}^L M)$,
and the second condition is equivalent with $\sigma_0 = 0$ because $\phi_b$ is bounded.

Both for $\mh H (\phi_b,q \epsilon)$ and for $\mc G (F)$, essential boundedness is 
equivalent to $\sigma_0 \in Z(\mf g^\vee) \cap \log ( X_\nr^+ ({}^L \mc M))$.\\
(b) We analyse the vector space in which $y$ lives:
\begin{align*}
\mf g^{\vee,\sigma,r}_{\phi_b,N} & = \{ X \in \mf g^\vee_{\phi_b} : [\sigma,X] = 2r X \} \\
& = \big\{ X \in \mr{Lie} \big( Z_{\mc G^\vee_\Sc}(\phi_b (\mb W_F)) \big) : 
[\sigma,X] = -\log (q_F) X \big\} \\
& = \{ X \in \mf g^\vee : [\sigma,X] = \log (q_F^{-1}) X, 
\mr{Ad}(\phi_b (w)) X = X \; \forall w \in \mb W_F \} .
\end{align*}
As $\phi_b$ is bounded, Ad$(\phi_b (w))$ only has eigenvalues of absolute value 1. Hence
\[
\mf g^{\vee,\sigma,r}_{\phi_b,N} = \{ X \in \mf g^\vee : \mr{Ad}(\exp (\sigma) \phi_b (\Fr_F)) X 
= q_F^{-1} X, \mr{Ad}(\phi_b (w)) X = X \; \forall w \in \mb I_F \} ,
\]
which equals $\mf g_{\phi'}^\vee$ for $(\phi',\rho)$ the Weil--Deligne parameter associated to
$\exp (y,\sigma,r,\rho)$. This shows that the nilpotent parts of the L-parameters on both sides
can be chosen from the same vector space $\mf g^{\vee,\sigma,r}_{\phi_b,N} = \mf g^\vee_{\phi'}$.

The group whose orbits on $\mf g^\vee_{\phi'}$ determine the ``open" property is
\begin{equation}\label{eq:7.1}
Z_{\mc G^\vee}(\phi' (\mb W_F)) = Z_{\mc G^\vee} (\phi_b (\mb W_F), \exp (\sigma)) =
Z_{\mc G^\vee_\der} (\phi_b (\mb W_F), \exp (\sigma)) Z(\mc G^\vee)^{\mb W_F} .
\end{equation}
On the other hand, for $\mf g^{\vee,\sigma,r}_{\phi_b,N}$ openness comes from orbits for the group\\
$Z_{G_{\phi_b} \times X_\nr ({}^L \mc G)}(\sigma)$, which has identity component
\begin{equation}\label{eq:7.2}
Z_{G_{\phi_b}}(\sigma)^\circ \times X_\nr ({}^L \mc G) = 
Z_{\mc G^\vee_\Sc} (\phi_b (\mb W_F), \exp (\sigma) )^\circ \times X_\nr ({}^L \mc G) .
\end{equation}
On both sides, the unique open orbit in the vector space is already a single orbit for the identity
component of the acting group, by \cite[\S 2.6]{KaLu}. Since $Z(\mc G^\vee)$ acts trivially,
the identity components of the groups \eqref{eq:7.1} and \eqref{eq:7.2} have exactly the same 
orbits on $\mf g^{\vee,\sigma,r}_{\phi_b,N} = \mf g^\vee_{\phi'}$. Thus $(y,\sigma,r)$ is
open if and only if $y$ lies in the unique open orbit of \eqref{eq:7.2} in 
$\mf g^{\vee,\sigma,r}_{\phi_b,N}$, which is the same as the open orbit of \eqref{eq:7.1} in
$\mf g^\vee_{\phi'}$. That condition is equivalent to openness of $\exp (y,\sigma,r)$.\\
(c) This is shown in the same way as part (b), we only have to add a condition to both
sets of nilpotent elements. For $\mf g^{\vee,\sigma,r}_{\phi_b,N}$, Definition \ref{def:3.1} 
says we need to restrict to $X$ such that there exists an enhancement $\rho'$ for which 
$(X,\sigma,r,\rho')$ has the same cuspidal support as $(y,\sigma,r,\rho)$. From 
$\mf g^\vee_{\phi'}$ we need to restrict to those $X$ such that there exists an enhancement 
$\rho'$ for which $(\phi',X,\rho')$ has the same cuspidal support as $(\phi',y,\rho)$. 
As our bijection exp between the two kinds of enhanced L-parameters is compatible with
cuspidal supports, we end up with the same subset of $\mf g^{\vee,\sigma,r}_{\phi_b,N} = 
\mf g^\vee_{\phi'}$ on both sides. 
\end{proof}

The following result has been proven in another way in \cite[Proposition 4.1 and \S 6]{CDFZ}, 
although in less generality.

\begin{prop}\label{prop:7.4}
Every essentially bounded (e.g. bounded or discrete) L-parameter for $\mc G (F)$ is open.
\end{prop}
\begin{proof}
Every enhanced L-parameter for $\mc G (F)$ has a cuspidal support \cite{AMS1}. That can be written
as $(\mc M (F),\phi, q\epsilon)$ with $(\phi,q\epsilon)$ a cuspidal enhanced L-parameter for a
Levi subgroup $\mc M (F)$ of $\mc G (F)$. Then $\phi$ is discrete, so by \cite{SiZi} of the form
$z \phi_b$ with $\phi_b$ bounded discrete and $z \in X_\nr^+ ({}^L \mc M)$. Now Lemma \ref{lem:7.3}
allows us to transfer the issue to L-parameters for $\mh H (\phi_b,q\epsilon)$. In that setting,
we proved the desired statement in Lemma \ref{lem:1.9}.

Alternatively, one can translate the proofs of Lemmas \ref{lem:1.12} and \ref{lem:1.9} directly
to the current setting. 
\end{proof}

To get Langlands parameters into play from representations of a reductive $p$-adic group 
$\mc G (F)$, we need to assume some reasonable form of the local Langlands correspondence 
involving Hecke algebras. As in Section \ref{sec:padic} we consider a Bernstein block 
$\Rep (\mc G (F))^{\mf s}$ 
determined by a unitary supercuspidal representation $\omega$ of $\mc M (F)$. We write 
\[
\Irr (\mc G (F))^{\mf s} = \Irr (\mc G (F)) \cap \Rep (\mc G (F))^{\mf s}.
\]

\begin{cond}\label{cond:7.1}
A local Langlands correspondence is known for $\Irr (\mc G (F))^{\mf s}$ and 
$\Irr (\mc M (F))^{\mf s}$, where $\mf s = [\mc M (F),\omega]$. 
Let $(\phi_b,\rho)$ be the enhanced bounded L-parameter of $\omega$. 
There is an isomorphism between the graded Hecke algebras of geometric type
\begin{itemize}
\item $\mh H (\mf t, W_{\mf s,\omega}, k^\omega, \natural_\omega)$ from Theorem \ref{thm:4.1},
\item $\mh H (\phi_b,q\epsilon) / (\mb r - \log (q_F)/2)$ from \cite[\S 3.1]{AMS3} 
\end{itemize}
induced by isomorphisms between $(\mf t \rtimes W_{\mf s,\omega},R_\omega)$ and the analogous 
data for $\mh H (\phi_b,q\epsilon)$. 

The same holds if we twist $\omega$ by a unitary unramified character of $\mc M (F)$.

Furthermore, the local Langlands correspondence for $\Irr (\mc G (F))^{\mf s}$ can be constructed 
via Theorem \ref{thm:4.1}, the above isomorphisms (for all such twists of $\omega$) and the
parametrization of $\Irr \big( \mh H (\phi_b,q\epsilon) / (\mb r - \log (q_F)/2) \big)$ from 
\cite[Theorem 3.8]{AMS3}.
\end{cond}

We point out that Condition \ref{cond:7.1} implies most of Condition \ref{cond:4.5}, 
only not the triviality of $\cE$. We refer to 
\cite{AMS1,AMS2,AMS3} for more background. A list of cases in which Condition 
\ref{cond:7.1} has been verified can be found in \cite[Theorem 5.4]{SolKL} and in the 
introduction before Theorem \ref{thm:E}. We expect that Condition \ref{cond:7.1} 
is always fulfilled.

\begin{thm}\label{thm:7.2}
Assume Condition \ref{cond:7.1} and let $\pi \in \Irr (\mc G (F))^{\mf s}$.
\enuma{
\item If $\pi$ is tempered or essentially square-integrable, then its L-parameter is open.
\item Suppose that $\mf s = [\mc M (F),\omega]$ with $\omega$ simply generic and 
$\pi$ generic. Then the L-parameter of $\pi$ is open with respect to cuspidal supports.
\item Suppose that $\mc G (F)$ is quasi-split, $\mf s = [\mc M (F),\omega]$ and $\omega$ and 
$\pi$ are generic. Then the L-parameter of $\pi$ is open.
}
\end{thm}
\begin{proof}
(a) By Lemma \ref{lem:7.3}.a and by the known properties of the LLC imposed by Condition 
\ref{cond:7.1}, or by \cite{AMS2,AMS3}, the L-parameter of $\pi$ is bounded or discrete. 
As shown in Proposition \ref{prop:7.4} and in \cite{CDFZ}, such an L-parameter is open. \\
(b) By Theorem \ref{thm:4.1} and Corollary \ref{cor:4.4}, the associated 
$\mh H (\phi_b,q\epsilon)$-module $\pi_{\mh H}$ is generic. Then Theorem \ref{thm:2.5}.(a,c)
says that the enhanced L-parameter of $\pi_{\mh H}$ is open with respect to 
$\big( G_{\phi_b} \times X_\nr ({}^L \mc G), M_{\phi_b} \times X_\nr ({}^L \mc G), q\cE \big)$.
Let $(\phi,\rho)$ be the image (under exp) of this enhanced L-parameter, Lemma \ref{lem:7.3}.c
says it is open with respect to cuspidal supports. The constructions in 
\cite[\S 3]{AMS3} entail that $(\phi,\rho)$ is the L-parameter of 
$\pi \in \Irr (\mc G (F))$.\\
(c) We only need to modify the arguments for part (b) a little. Namely, Condition 
\ref{cond:7.1}, Theorem \ref{thm:4.7} and Lemma \ref{lem:4.8} entail that 
$\mh H (\phi_b, q\epsilon)$ comes from a trivial equivariant local system $\cE$. 
Then Theorem \ref{thm:2.5}.(b,c) and Lemma \ref{lem:7.3}.b give the stronger conclusion.
\end{proof}

We note that Theorem \ref{thm:7.2}.b proves Conjecture \ref{conj:B}.a for all the cases
listed in \cite[Theorem 5.4]{SolKL} or just before Theorem \ref{thm:E}. Similarly
Theorem \ref{thm:7.2}.c proves a part of Conjecture \ref{conj:B}.b.
\vspace{5mm}

\textbf{Acknowledgement.}
We thank the referee for his or her careful report, and especially for pointing out 
a problem in an earlier version of this paper.


\begin{thebibliography}{99}

\bibitem[Ach]{Ach} P.N. Achar,
\emph{Perverse sheaves and applications to representation theory},
Mathematical Surveys and Monographs {\bf 258}, American Mathematical Society, 2021

\bibitem[Art]{Art} J. Arthur,
\emph{The endoscopic classification of representations: orthogonal and symplectic groups},
Colloquium Publications volume {\bf 61}, American Mathematical Society, 2013

\bibitem[ABPS]{ABPSSLn} A.-M. Aubert, P.F. Baum, R.J. Plymen, M. Solleveld,
``Hecke algebras for inner forms of $p$-adic special linear groups",
J. Inst. Math. Jussieu {\bf 16.2} (2017), 351--419

\bibitem[AMS1]{AMS1} A.-M. Aubert, A. Moussaoui, M. Solleveld,
``Generalizations of the Springer correspondence and cuspidal Langlands parameters'',
Manus. Math. {\bf 157} (2018), 121--192

\bibitem[AMS2]{AMS2} A.-M. Aubert, A. Moussaoui, M. Solleveld,
``Graded Hecke algebras for disconnected reductive groups", pp. 23--84 in:
\emph{Geometric aspects of the trace formula, W. M\"uller, S. W. Shin, N. Templier (eds.)},
Simons Symposia, Springer, 2018

\bibitem[AMS3]{AMS3} A.-M. Aubert, A. Moussaoui, M. Solleveld,
``Affine Hecke algebras for Langlands parameters",
arXiv:1701.03593v6, 2024

\bibitem[AMS4]{AMS4} A.-M. Aubert, A. Moussaoui, M. Solleveld,
``Affine Hecke algebras for classical $p$-adic groups",
arXiv:2211.08196, 2022

\bibitem[AuXu]{AuXu}, A.-M. Aubert, Y. Xu,
``The explicit local Langlands correspondence for $G_2$",
arXiv:2208.12391v3, 2023

\bibitem[BeDe]{BeDe} J. Bernstein, P. Deligne,	
``Le "centre" de Bernstein",
pp. 1--32 in: \emph{Repr\'esentations des groupes r\'eductifs sur un corps local}, 
Travaux en cours, Hermann, Paris, 1984

\bibitem[BeLu]{BeLu} J. Bernstein, V. Lunts,
\emph{Equivariant sheaves and functors}, Lecture Notes in Mathematics {\bf 1578}, 
Springer-Verlag, Berlin, 1994     

\bibitem[BuHe]{BuHe} C.J. Bushnell, G. Henniart,
``Generalized Whittaker models and the Bernstein center",
Amer. J. Math. {\bf 125.3} (2003), 513--547

\bibitem[CaSh]{CaSh} W. Casselman, F. Shahidi,
``On irreducibility of standard modules for generic representations",
Ann. Scient. \'Ec. Norm. Sup (4) {\bf 31} (1998), 561--589

\bibitem[ChGi]{ChGi} N. Chriss, V. Ginzburg,
\emph{Representation theory and complex geometry},
Birkh\"auser, 1997

\bibitem[CDFZ]{CDFZ} C. Cunningham, S. Dijols, A. Fiori, Q. Zhang,
``Generic representations, open parameters and ABV-packets for $p$-adic groups",
arXiv:2404.07463, 2024

\bibitem[CFZ]{CFZ} C. Cunningham, A. Fiori, Q. Zhang,
``Toward the endoscopic classification of unipotent representations of $p$-adic $G_2$",
arXiv:2101.04578, 2021

\bibitem[Dij]{Dij} S. Dijols,
``The generalized injectivity conjecture",
Bull. Soc. Math. France {\bf 150.2} (2022), 251-–345

\bibitem[Eve]{Eve} S. Evens,
``The Langlands classification for graded Hecke algebras",
Proc. Amer. Math. Soc. {\bf 124.4} (1996), 1285--1290

\bibitem[GrPr]{GrPr} B.H. Gross, D. Prasad,
``On the decomposition of a representation of $SO_n$ when restricted to $SO_{n-1}$",
Canad. J. Math {\bf 44.5} (1992), 974--1002

\bibitem[GrRe]{GrRe} B.H. Gross, M. Reeder,
``Arithmetic invariants of discrete Langlands parameters",
Duke Math. J. {\bf 154.3} (2010), 431--508

\bibitem[Hai]{Hai} T.J. Haines, 
``The stable Bernstein center and test functions for Shimura varieties'',
pp. 118--186 in: \emph{Automorphic forms and Galois representations},
London Math. Soc. Lecture Note Ser. {\bf 415}, Cambridge University Press, 2014

\bibitem[Hei]{Hei} V. Heiermann,	
``Op\'erateurs d'entrelacement et alg\`ebres de Hecke avec param\`etres d'un groupe 
r\'eductif $p$-adique - le cas des groupes classiques",
Selecta Math. {\bf 17.3} (2011), 713--756

\bibitem[HeMu]{HeMu} V. Heiermann, G. Mui\'c,
``On the standard modules conjecture",
Math. Z. {\bf 255.4} (2007), 847-–853 	

\bibitem[HeOp]{HeOp} V. Heiermann, E. Opdam,
``On the tempered L-function conjecture",
Amer. J. Math. {\bf 135.3} (2013), 777--799 

\bibitem[KaLu1]{KaLu1} D. Kazhdan, G. Lusztig,
\emph{Schubert varieties and Poincaré duality}, pp. 185--203 in: 
\emph{Geometry of the Laplace operator}, Proc. Sympos. Pure Math. {\bf 36},
American Mathematical Society, Providence RI, 1980 

\bibitem[KaLu2]{KaLu} D. Kazhdan, G. Lusztig,
``Proof of the Deligne--Langlands conjecture for Hecke algebras",
Invent. Math. {\bf 87} (1987), 153--215

\bibitem[Kos]{Kos} B. Kostant,
``The principal three-dimensional subgroup and the Betti numbers of a complex simple Lie group",
Amer. J. Math. {\bf 81} (1959), 973--1032

\bibitem[Lus1]{Lus-Int} G. Lusztig,
``Intersection cohomology complexes on a reductive group",
Invent. Math. {\bf 75.2} (1984), 205--272

\bibitem[Lus2]{Lus-Cusp1} G. Lusztig,
``Cuspidal local systems and graded Hecke algebras'',
Publ. Math. Inst. Hautes \'Etudes Sci. {\bf 67} (1988), 145--202

\bibitem[Lus3]{Lus-Gr} G. Lusztig,
``Affine Hecke algebras and their graded version",
J. Amer. Math. Soc {\bf 2.3} (1989), 599--635

\bibitem[Lus4]{Lus-Cusp2} G. Lusztig,
``Cuspidal local systems and graded Hecke algebras. II'',
pp. 217--275 in: \emph{Representations of groups},
Canadian Mathematical Society Conference Proceedings {\bf 16}, 1995

\bibitem[Lus5]{Lus-Uni} G. Lusztig,
``Classification of unipotent representations of simple $p$-adic groups",
Int. Math. Res. Notices {\bf 11} (1995), 517-589

\bibitem[Lus6]{Lus-Uni2} G. Lusztig,
``Classification of unipotent representations of simple $p$-adic groups II",
Represent. Theory {\bf 6} (2002), 243--289

\bibitem[Lus7]{Lus-open} G. Lusztig,
``Open problems on Iwahori--Hecke algebras",
EMS Newsletter (2020) and arXiv:2006.08535

\bibitem[MoRe]{MoRe} C. M\oe glin, D. Renard,
``Sur les paquets d'Arthur des groupes classiques et unitaires non quasi-d\'eploy\'es",  
pp. 341--361 in: \emph{Relative Aspects in Representation Theory, Langlands 
Functoriality and Automorphic Forms}, 
Lecture Notes in Mathematics \textbf{2221}, 2018


\bibitem[OpSo]{OpSo} E. Opdam, M. Solleveld,
``Hermitian duals and generic representations for affine Hecke algebras",
arXiv:2309.04829, 2023

\bibitem[RaRa]{RaRa} A. Ram, J. Rammage,
``Affine Hecke algebras, cyclotomic Hecke algebras and Clifford theory",
pp. 428--466 in: \emph{A tribute to C.S. Seshadri (Chennai 2002)}, 
Trends in Mathematics, Birkh\"auser, 2003

\bibitem[Ree]{Ree} M. Reeder,
``Whittaker models and unipotent representations of $p$-adic groups",
Math. Ann. {\bf 308} (1997), 587--592

\bibitem[Ren]{Ren} D. Renard,
\emph{Repr\'esentations des groupes r\'eductifs p-adiques},
Cours sp\'ecialis\'es {\bf 17},
Soci\'et\'e Math\'ematique de France, 2010

\bibitem[Rie]{Rie} K. Rietsch,
``An introduction to perverse sheaves",
Fields Institute Communications {\bf 40} (2004), 391--429

\bibitem[Rod]{Rod} F. Rodier,
``Whittaker models for admissible representations of reductive p-adic split groups", 
pp. 425--430 in: \emph{Harmonic analysis on homogeneous spaces}, 
Proc. Sympos. Pure Math. AMS {\bf 26} (1973)

\bibitem[Shah]{Shah} F. Shahidi,	
``A proof of Langlands' conjecture on Plancherel measures;	 complementary series of $p$-adic groups",
Ann. Math. {\bf 132.2} (1990), 273--330

\bibitem[Shal]{Shal} J.A. Shalika, 
``The multiplicity one theorem for $GL_n$",
Ann. of Math. (2) {\bf 100} (1974), 171--193 

\bibitem[SiZi]{SiZi} A.J. Silberger, E.-W. Zink,
``Langlands Classification for L-Parameters",
J. Algebra {\bf 511} (2018), 299--357

\bibitem[Sol1]{SolGHA} M. Solleveld,
``Parabolically induced representations of graded Hecke algebras",
Algebras Represent. Theory {\bf 15.2} (2012), 233--271

\bibitem[Sol2]{SolAHA} M. Solleveld, ``On the classification of irreducible 
representations of affine Hecke algebras with unequal parameters", 
Represent. Theory {\bf 16} (2012), 1--87

\bibitem[Sol3]{SolHecke} M. Solleveld,
``Affine Hecke algebras and their representations",
Indag. Math. {\bf 32.5} (2021), 1005--1082

\bibitem[Sol4]{SolEnd} M. Solleveld,
``Endomorphism algebras and Hecke algebras for reductive $p$-adic groups",
J. Algebra {\bf 606} (2022), 371--470

\bibitem[Sol5]{SolLLCunip} M. Solleveld,
``A local Langlands correspondence for unipotent representations",
Amer. J. Math. {\bf 145.3} (2023), 673--719

\bibitem[Sol6]{SolRamif} M. Solleveld,
``On unipotent representations of ramified $p$-adic groups",
Represent. Theory {\bf 27} (2023), 669--716 


\bibitem[Sol7]{SolKL} M. Solleveld,
``Graded Hecke algebras, constructible sheaves and the $p$-adic Kazhdan--Lusztig conjecture", 
J. Algebra {\bf 667} (2025), 865--910

\bibitem[Sol7]{SolSGHA} M. Solleveld,
``Graded Hecke algebras and equivariant constructible sheaves on the nilpotent cone",
Quarterly J. Math. {\bf 76.1} (2025), 109--146 

\bibitem[Sol9]{SolQS} M. Solleveld,
``On principal series representations of quasi-split reductive $p$-adic groups", 
Pacific J. Math. {\bf 334.2} (2025), 269--327

\bibitem[Vog]{Vog} D. Vogan,	
``The local Langlands conjecture",
pp. 305--379 in: \emph{Representation theory of groups and algebras},
Contemp. Math. {\bf 145}, American Mathematical Society, Providence RI, 1993

\end{thebibliography}
\end{document}